\documentclass[11pt,english]{amsart}

\renewcommand{\sl}{\text{sl}}
\usepackage[T1]{fontenc}
\usepackage{float}
\usepackage{thmtools}
\usepackage{comment}

\usepackage[english]{babel}
\usepackage{amssymb,url,xspace}
\usepackage{graphicx}
\usepackage{subfig}
\usepackage{xcolor}
\usepackage{textcmds}
\usepackage[section]{placeins}
\usepackage[hidelinks]{hyperref}

\usepackage{amsmath}
\usepackage{amssymb}
\usepackage{amsfonts}
\usepackage{amsthm}

\declaretheorem[name=Theorem,numberwithin=section]{theo}
\newtheorem*{theo*}{Theorem}
\newtheorem{coro}[theo]{Corollary}
\newtheorem*{coro*}{Corollary}
\newtheorem{conj}[theo]{Conjecture}
\newtheorem*{conj*}{Conjecture}
\newtheorem{prop}[theo]{Proposition}
\newtheorem{lemm}[theo]{Lemma}

\newtheorem*{ques*}{Question}
\newtheorem{exem}[theo]{Example}

\theoremstyle{definition}
\newtheorem{defi}[theo]{Definition}

\theoremstyle{remark}
\newtheorem{rema}[theo]{Remark}

\DeclareCaptionType{equ}[][]
\def\HF{\text{HF}}
\def\CFK{\text{CFK}}
\def\HFKh{\widehat{\text{HFK}}}
\def\MT{\text{MT}}
\def\mindeg{\text{mindeg}}
\def\maxdeg{\text{maxdeg}}
\def\Span{\text{span}}
\def\rank{\text{rank}}

\begin{document}

\title{On Fox's trapezoidal conjecture}
\date{\today}

\author{Soheil Azarpendar}
\address{Mathematical Institute, University of Oxford, Andrew Wiles Building,
		Radcliffe Observatory Quarter, Woodstock Road, Oxford, OX2 6GG, UK}
\email{azarpendar@maths.ox.ac.uk}

\author{András Juhász}
\address{Mathematical Institute, University of Oxford, Andrew Wiles Building,
		Radcliffe Observatory Quarter, Woodstock Road, Oxford, OX2 6GG, UK}
\email{juhasza@maths.ox.ac.uk}

\author{Tam\'as K\'alm\'an}

\address{Department of Mathematics, 
	Tokyo Institute of Technology,
	H-214, 2-12-1 Ookayama, Meguro-ku, Tokyo 152-8551, Japan\\
	and International Institute for Sustainability with Knotted Chiral Meta Matter (WPI-SKCM$^2$),
Hiroshima University
1-3-1 Kagamiyama, Higashi-Hiroshima, Hiroshima 739-8526, Japan}

\email{kalman@math.titech.ac.jp}

\thanks{TK was supported by a Japan Society for the Promotion of Science (JSPS) Grant-in-Aid for Scientific Research C (no.\ 23K03108).}


\begin{abstract}
We investigate Fox's trapezoidal conjecture for alternating links. We show that it holds for diagrammatic Murasugi sums of special alternating links, where all sums involved have length less than three (which includes diagrammatic plumbing). It also holds for links containing a large twist region, which we call twist-concentrated. Furthermore, we show some weaker inequalities between consecutive coefficients of the Alexander polynomial of an alternating 3-braid closure, and extend this to arbitrary alternating links. We then study an extension of the trapezoidal conjecture due to Hirasawa and Murasugi, which states that the stable length of the Alexander polynomial of an alternating link can be bounded from above using the signature. We estabilish this and determine when equality holds for diagrammatic Murasugi sums of special alternating knots where each sum has length less than three, and also for twist-concentrated 3-braids. Finally, we study the behavior of the Hirasawa--Murasugi inequality under concordance.
\end{abstract}

\maketitle
\tableofcontents
\pagebreak

\section{Overview}

 In Section~\ref{Section1}, we focus on the following conjecture of Fox \cite{fox1962some} about the Alexander polynomial of alternating knots. This is known as \textit{Fox's trapezoidal conjecture}. 

\begin{restatable}{conj}{Fox} \cite{fox1962some}\label{foxconjecture}
    Let $(a_1,\ldots,a_n)$ be the sequence of absolute values of the coefficients of the Alexander polynomial $\Delta_K$ of an alternating link $K$. With $m=\lfloor \frac{n}{2} \rfloor$, the following properties hold:
    \begin{enumerate} 
    \item $a_1 \leq a_2 \leq \dots \leq a_m \geq \dots \geq a_{n-1} \geq a_n$;
    \item if $a_i=a_{i+1}$ for $i < m$, then $a_i =a_{i+1}= \dots = a_m$.
    \end{enumerate}
\end{restatable}

A sequence with such properties is called \emph{trapezoidal}; see Figure \ref{trapezoidal}. 

In Subsection~\ref{survey}, we go through the main known results about the trapezoidal conjecture in chronological order and briefly mention the methods used to derive each result. The following table summarizes Subsection~\ref{survey}.\\

\makebox[\textwidth][c]
{
\renewcommand{\arraystretch}{1.4}
\begin{tabular}{ |p{5cm}||p{5cm}|p{5cm}|}
 \hline
 \multicolumn{3}{|c|}{Results about trapezoidal conjecture } \\
 \hline
Reference & Class of alternating links& Method\\
 \hline
\textcolor{red}{\textbf{Parris}}\cite{Parris}   & Alternating pretzel knots    &Formula for $\Delta_K$\\
\textcolor{red}{\textbf{Hartley}}\cite{hartley_1979}&   Two-bridge links  & Relating a formula for $\Delta_K$ to extended Schubert diagrams \\
\textcolor{red}{\textbf{Murasugi}}\cite{murasugi_1985} &Alternating arborescent links & Skein formula for plumbing trees and induction\\
\textcolor{red}{\textbf{Ozsváth, Szabó}}\cite{OSineq}   & Alternating knots with $g(K)\leq 2$ & Computation of $\HF^{+}(S^3_{0}(K))$ using large surgery formula\\
 \textcolor{red}{\textbf{Jong}}\cite{Jong},  \textcolor{red}{\textbf{Stoimenow}}\cite{Stoimenow2011DiagramGG} & Alternating knots with $g(K)\leq 4$ & Reduction to a finite list of genus generators using $\Bar{t'_2}$ move\\
 \textcolor{red}{\textbf{Alrefai, Chbili}}\cite{small3braids}& Alternating 3-braids of length less than 4  & Formula for $\Delta_K$ using Burau representation\\
  \textcolor{red}{\textbf{Hafner, Mészáros, and Vidinas}}\cite{KarolaLogconcavityOT}& Special alternating links & Introduction of multivariable refinement of $\Delta_K$ \\
 \hline
\end{tabular}\par
}
\hfill \break

In Subsection~\ref{stablisesection}, we use Crowell's formula for the Alexander polynomial and arguments similar to Jong\cite{Jong} to derive the trapezoidal conjecture for alternating links that contain a (linearly) large twist region. To be more precise, let $\MT(L)$ be the size of the biggest coherent twist region in an alternating link; see Definition~\ref{MTdef} and the paragraph preceding it. We call an alternating link \emph{twist-concentrated} if it satisfies the inequality
\[
\MT(L)-3 \geq g(L) + \frac{|L|}{2},
\]
where $|L|$ is the number of connected components of L. 
Such links satisfy Fox's trapezoidal conjecture.

More generally, we show that the first $n$ inequalities of the trapezoidal conjecture follow from the existence of a twist region with $n+3$ crossings. The main results of this subsection are Theorems \ref{twistconcentrated} and \ref{twistconcentratedgeneral}.

In Subsection~\ref{plumbingofspecial}, we use the decomposition of alternating links into diagrammatic Murasugi sums of special alternating links. Murasugi \cite{Murasugisums} proposed an algorithm for such a decomposition. We use this to extend the result of Hafner, Mészáros, and Vidinas~\cite{KarolaLogconcavityOT} to a certain class of Murasugi sums. In general, it is not possible to derive a formula for the Alexander polynomial of the Murasugi sum of two knots using the Alexander polynomials of the summands. If we restrict the sums to diagrammatic plumbings, then such a formula can be derived as it is done in Proposition~\ref{Alexanderformulaplumbing}. More generally, this can be done for diagrammatic Murasugi sums of length less than 3; see Subsection \ref{plumbingofspecial} 
Using this formula, we prove the trapezoidal conjecture for diagrammatic Murasugi sums of special alternating links such that the lengths of all Murasugi sums are less than 3. The main results of this subsection are Theorems~\ref{plumbingofspecialthm} and~\ref{Trapezoidalsumovertree}.

In Subsection~\ref{3braids}, we focus on alternating 3-braids which form the simplest class of non-special alternating links, and we try to extend the main idea of Hafner, Mészáros, and Vidinas \cite{KarolaLogconcavityOT}. Their main idea is to construct a multivariable refinement of the Alexander polynomial such that all of its coefficients are equal to $1$ and its support is M-convex. This means that this multivariable polynomial will be Lorentzian, and this opens the door to techniques developed by Brändén and Huh~\cite{Huh}; see Subsection~\ref{3braids} for a detailed explanation. We consider a natural generalization of the multivariable refinement to alternating 3-braids. We show that, unfortunately, this multivariable polynomial cannot be Lorentzian except when the mentioned alternating 3-braid is a connected sum of two torus links. Despite this, one can use combinatorial arguments to derive certain inequalities between its coefficients, which in turn result in inequalities between the coefficients of the Alexander polynomial. These inequalities are generally either weaker than the statement of the trapezoidal conjecture (by a multiplicative factor) or they are independent from it. Finally, since this argument only uses the local combinatorial structure of the alternating 3-braids, one can generalize this to all alternating links. These results are stated in Theorems~\ref{badinequality} and~\ref{badinequalitygeneral}.

In Section~\ref{Section2}, we discuss a stronger version of the trapezoidal conjecture proposed by Hirasawa and Murasugi~\cite{hirasawa2013various}. The main addition is a conjectured inequality between the stable length of the sequence of coefficients of the Alexander polynomial and the signature of the link, explained in the following.

\begin{restatable}{conj}{HMrestatable}\label{HMconjecture}
Let $L$ be an alternating link and $\Delta_L(t)=\sum\limits_{i=0}^{l-1}(-1)^i a_i t^{i}$ be its Alexander polynomial. Then there is $i_0 \in \{1,\dots,l/2+1\}$ such that
\[
a_0 < \cdots <a_{i_0-1} = a_{i_0}=\cdots=a_{l-i_0}>\cdots>a_{l-1}.
\]
Moreover,
\[
\left\lfloor \frac{|\sigma(L)|+1}{2} \right\rfloor \geq \left\lfloor \frac{l-2
(i_0-1)}{2} \right\rfloor.
\]
\end{restatable}

We call this inequality the \textit{H-M inequality}. We call those alternating links for which the H-M inequality becomes an equality \textit{H-M sharp}. We define the \textit{stable length of $L$} to be 
\[
sl(L) := l-2(i_0-1).
\]

In Subsection~\ref{HMintro}, we use the results developed in Section~\ref{Section1} to derive corollaries about the H-M inequality. Using the results of Subsection~\ref{stablisesection}, one can show that the H-M inequality for twist-concentrated alternating links $L$ reduces to the problem for the rest of the alternating links. In the case of certain twist-concentrated links for which we can compute the signature, such as twist-concentrated 3-braids, this argument proves the H-M inequality. These results are stated in Corollary~\ref{ReductionoftwistsinHM} and Theorem~\ref{HM3-braids}.

We then use the results of Subsection \ref{plumbingofspecial}, and prove the H-M inequality for the diagrammatic Murasugi sum of special alternating links when the lengths of all Murasugi sums are less than 3. This is the content of Corollary~\ref{H-Mplummbing}.

In Subsection~\ref{HMsharp}, we investigate the properties of H-M sharp links. Refining the arguments of the previous subsection, one can derive a diagrammatic characterization of H-M sharp knots, which are diagrammatic Murasugi sums of special alternating links such that the lengths of all Murasugi sums are less than $3$. This result is stated in Corollary~\ref{H-Msharpplum}.

In Subsection~\ref{dual}, we use the Fox--Milnor equation to form a relation between the H-M inequality and concordance. Consider the minimum degree of the Alexander polynomial in the algebraic concordance class of an alternating knot $K$, denoted by $d_c(K)$. It turns out that $d_c(K)+1$ gives an upper bound for the stable length $sl(K)$. The H-M inequality can be viewed as an optimization problem; i.e., finding the maximal stable length in a given algebraic concordance class. From this perspective, computing $d_c(K)$ acts as a dual problem for the H-M inequality. This leads to proof of the H-M inequality for certain concordance classes, as stated in Corollaries~\ref{HMdual} and~\ref{amphicheiral}. This perspective relaxes the alternating condition and thus might be useful.

\section{Fox's trapezoidal conjecture}\label{Section1}
\subsection{Introduction}\label{survey}

 A link $K$ in $S^3$ is called \emph{alternating} if it has a projection $D$ in which the undercrossings and overcrossings alternate as one moves along any component. Let $\Delta_K(t) \in \mathbb{Z}[t]$ be the (single-variable) Alexander polynomial of $K$. We also use the notation $\tilde{\Delta}_{K} \in \mathbb{Z}[t,t^{-1}]$ to denote the \emph{symmetrized Alexander polynomial} of the link $K$. Although this is not standard notation, we need it for clarity.

 In 1962, Fox conjectured that the sequence of absolute values of the coefficients of $\Delta_K$ is trapezoidal for alternating K.

\Fox*

See Figure~\ref{trapezoidal} for an example of a trapezoidal sequence.

\begin{figure}[ht]
\centering
\includegraphics[scale=0.3]{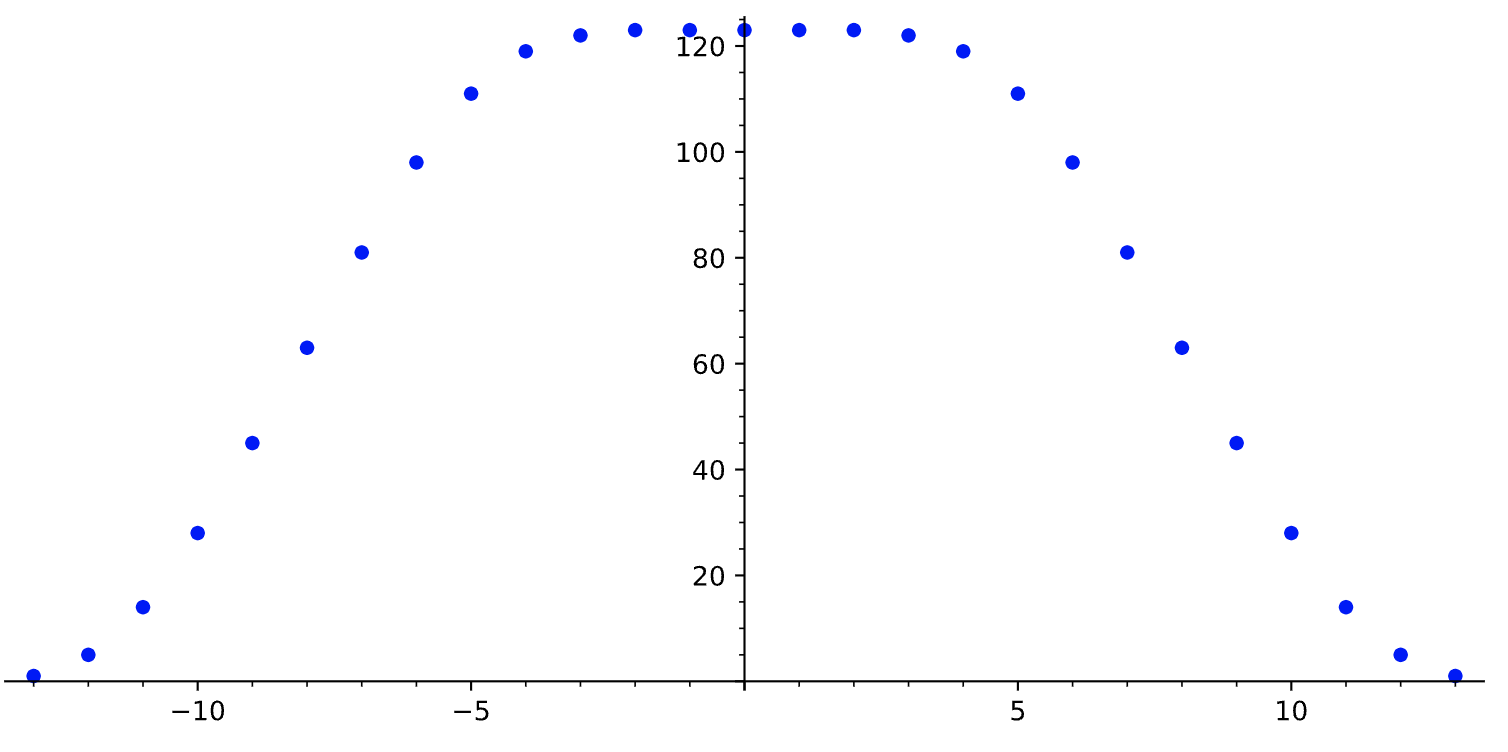}
\caption{Graph of a trapezoidal sequence.}
\label{trapezoidal}
\end{figure}

Note that the sequence of coefficients of $\Delta_K$ is symmetric; i.e., ${a_i = a_{n-i+1}}$ for $i \in \{1,\dots,n\}$.  Murasugi~\cite{murasugi_1985} proved that, for an alternating $K$, signs of this sequence alternate, which means, using the notation of Conjecture~\ref{foxconjecture}, that 
\[
\Delta_K(t) = \pm \sum\limits_{1\leq i \leq n}  (-1)^{i-1}a_i t^{i-1}.
\]
Stoimenow proposed the following strengthening of Conjecture~\ref{foxconjecture}.
 
\begin{conj}\cite{Stoimenow2005NewtonlikePO}\label{Stoimenowconj}
Let $(a_1,\ldots,a_n)$ be the sequence of absolute values of the coefficients of the Alexander polynomial $\Delta_K$ of an alternating link $K$ in $S^3$. Then $(a_1,\ldots,a_n)$ is log-concave; i.e., for every $i \in \{2,\dots,n-1\}$, we have
\[
\ a_{i-1}a_{i+1}\leq a_i^2.
\]
\end{conj}
 
 \textcolor{red}{\textbf{Parris}}~\cite{Parris} in his PhD thesis used Fox calculus to derive a formula for the Alexander polynomial of pretzel knots which in turn proves the trapezoidal conjecture for this class.
 
\textcolor{red}{\textbf{Hartley}}~\cite{hartley_1979} proved the trapezoidal conjecture for two-bridge knots $D(p,q)$. The Alexander polynomial $\Delta_{D(p,q)}$ can be computed by the following formula. Let $p$, $q$ be coprime integers satisfying $0<q<p$ and $q$ odd. Then 
 \begin{equation}\label{Minkus}
     \Delta_{D(p,q)}(t)=\sum\limits_{0\leq k \leq p-1} (-1)^{k} \ t^{s_k} \ , \ s_k=\sum\limits^{k}_{i=0} (-1)^{\left\lfloor \frac{iq}{p} \right\rfloor}.
 \end{equation}
Minkus~\cite{Minkus1982TheBC} proved this formula using Fox calculus on a Schubert normal diagram of $D(p,q)$. Using a similar argument, Hartley provided another formula for $\Delta_{D(p,q)}$ which can be read from an extended diagram obtained from unwinding Schubert's normal diagram; see Figure \ref{Hartley}. Note that, in the extended diagram, all of the diagonal segments have slope $\frac{q}{p}$. 
\begin{figure}[h]
\centering
\includegraphics[scale=0.4]{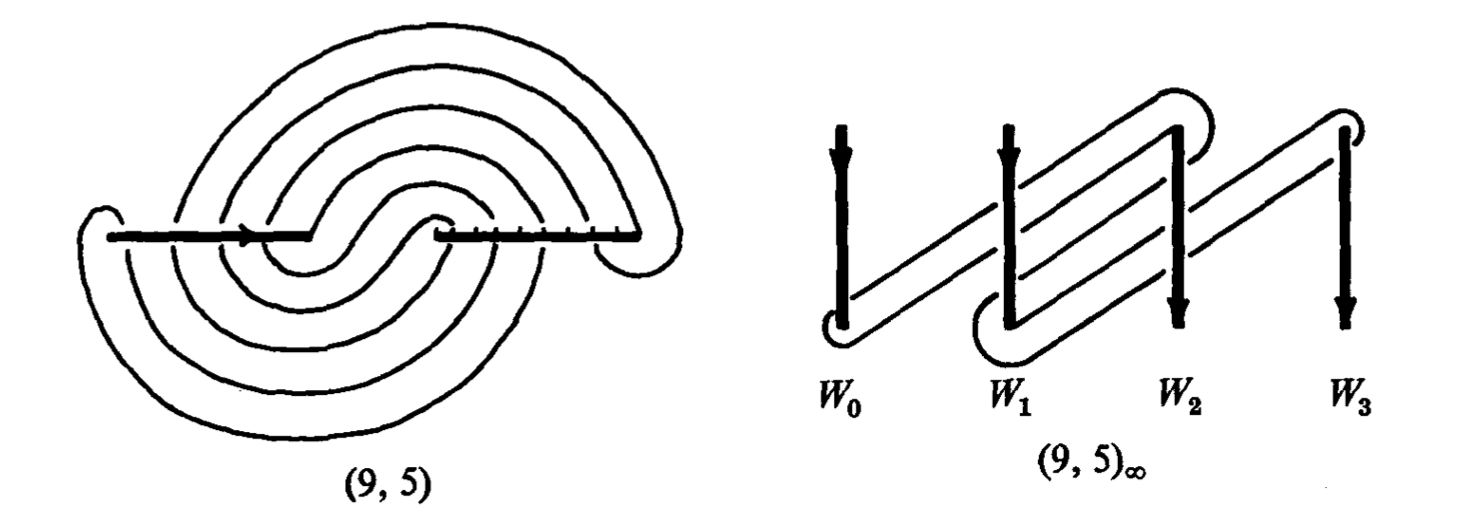}
\caption{Hartley's extended diagram~\cite{hartley_1979}.}\label{Hartley}
\end{figure}
\begin{prop}\cite{hartley_1979}
Let $\alpha_i$ be the number of segments joining the $i$-th and $(i+1)$-st vertical lines in the extended diagram of $D(p,q)$. Then 
\[
\Delta_{D(p,q)}=\sum (-1)^{i} \alpha_i t^i.
\]
\end{prop}
Hartley uses other combinatorial properties of the extended diagrams to devise an inductive argument. \textcolor{red}{\textbf{Chen}} \cite{Chen2017OnTA} uses the same argument to derive stronger results; see Section~\ref{Section2} for more details.

\textcolor{red}{\textbf{Murasugi}}~\cite{murasugi_1985} proved the trapezoidal conjecture for arborescent alternating links. These are alternating links that can be constructed from the plumbing of oriented twisted bands over a tree (for details, see Subsection~\ref{plumbingofspecial}). Murasugi uses skein relations on plumbing trees to derive an inductive argument. Figure~\ref{skein} shows an example of such a skein triple.

As depicted in Figure~\ref{skein}, changing the orientation of a crossing reduces the number of twists in a band (weight of a vertex) by 2 and smoothing a crossing, results in a connected sum of the adjacent bands.

Furthermore, Murasugi derived a sufficient criterion for an algebraic link to be alternating. This relies on a trick which relates the plumbing tree $T$ to the Tait graph of a diagram of the associated algebraic link $l(T)$. More precisely, Murasugi showed that if one augments $T$ by an exterior vertex $\hat{v}$, and connects $\hat{v}$ to $V(T)$ such that in the augmented graph, the weight of each vertex is equal to the signed count of its adjacent edges, then $l(T)$ is isotopic to the result of applying the median construction to the planar dual of $T \cup \{\hat{v}\}$; see Figure~\ref{plumbing-tait}.

\begin{figure}[h]
\centering
\includegraphics[scale=0.3]{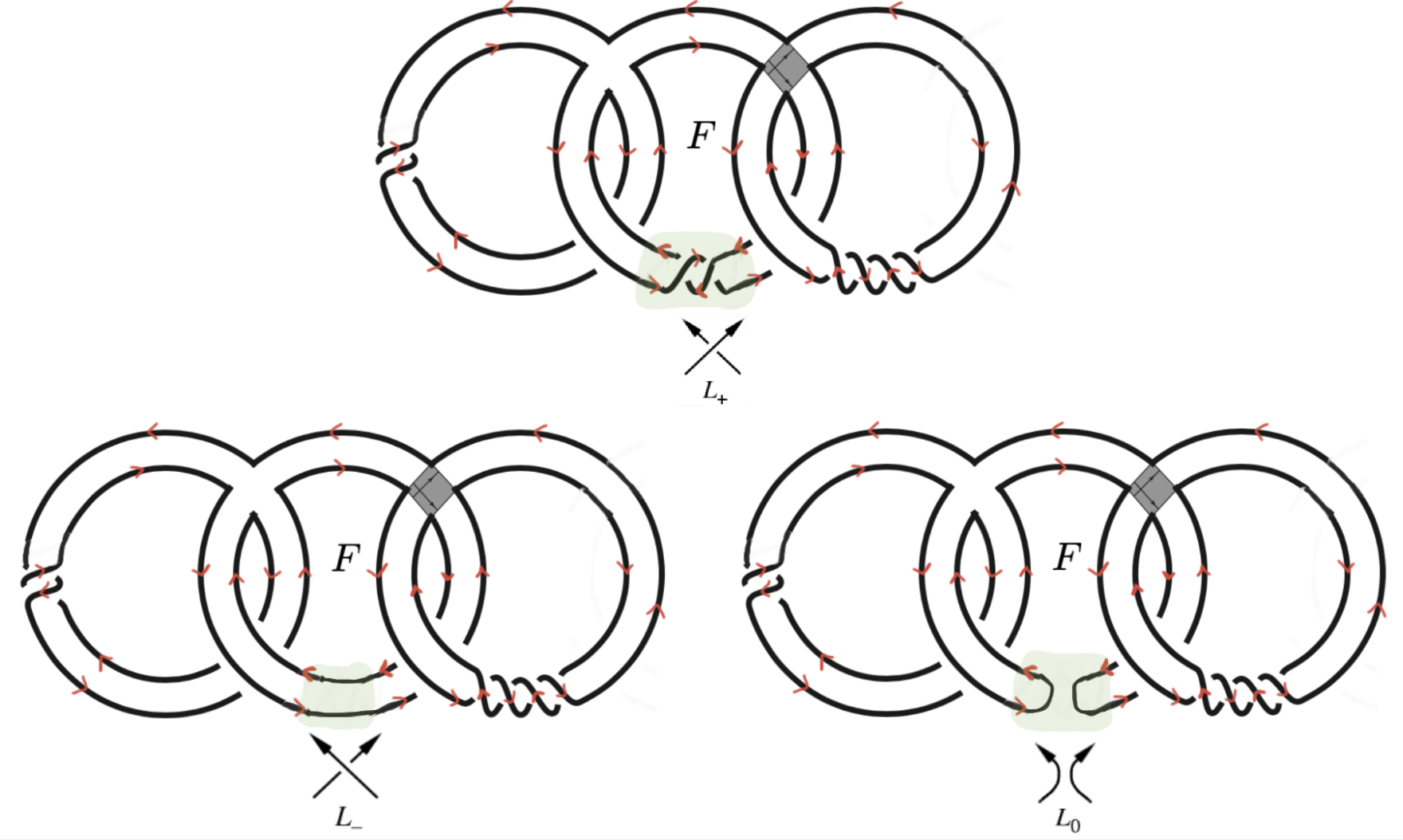}
\caption{Skein relation on plumbing trees.}\label{skein}
\end{figure}

\begin{figure}[h]
\centering
\includegraphics[scale=0.3]{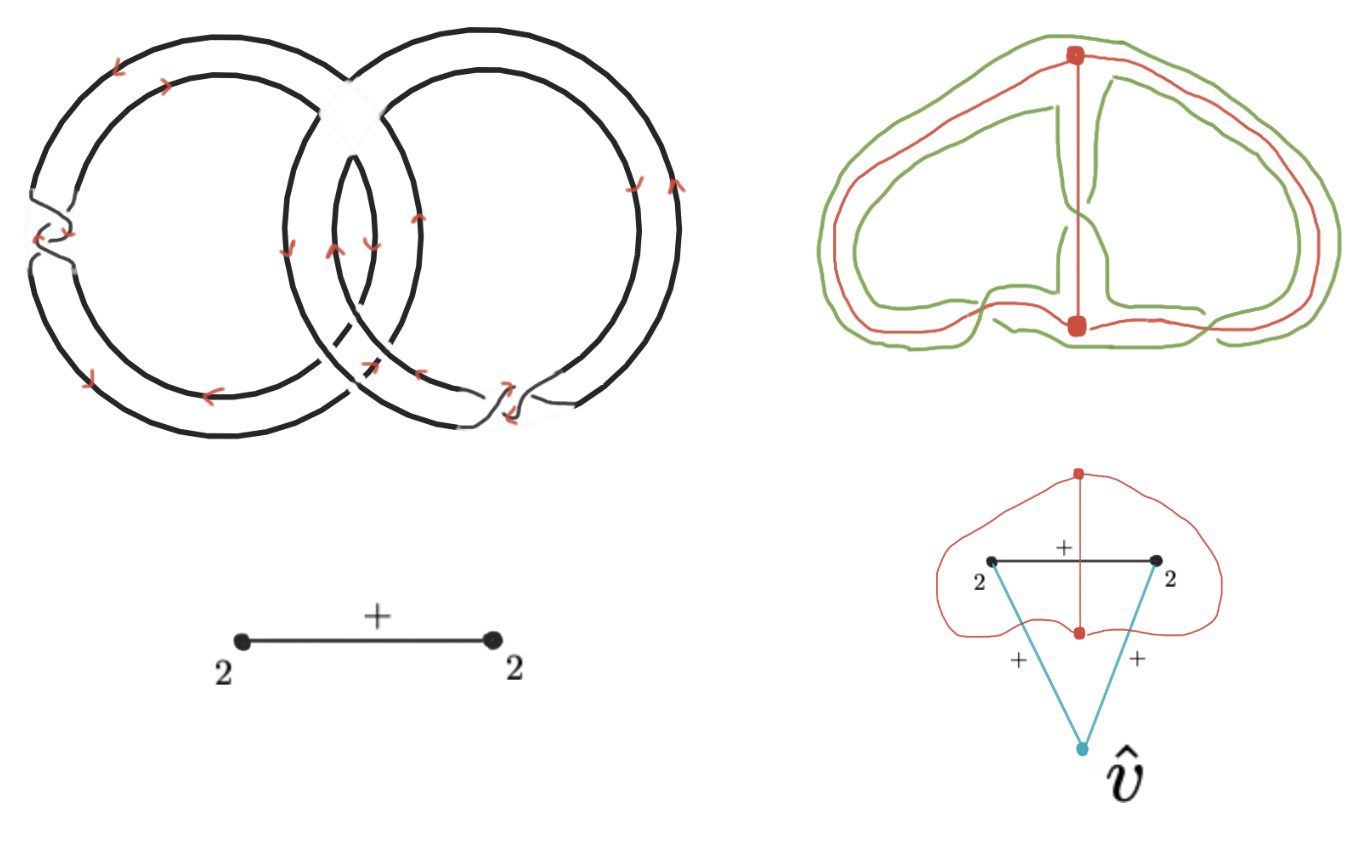}
\caption{Alternative construction of an algebraic link (both black and green diagrams represent the trefoil).}\label{plumbing-tait}
\end{figure}

\textcolor{red}{\textbf{Ozsváth and Szabó}} \cite{Szabó2003} used Heegaard Floer homology to prove the trapezoidal conjecture for alternating knots of genus two. They proved that alternating knots are HF-thin; i.e., 
\[
\HFKh(S^3,K,s)\cong \mathbb{Z}^{|a_s|}\  \text{supported in Maslov grading}\  s+\frac{\sigma}{2},
\]
where $\tilde{\Delta}_K(T)=a_0+\sum\limits_{s>0}a_s(T^s+T^{-s})$ and $\sigma$ is the signature of $K$. The proof comes from construction of a Heegaard diagram for a knot where the Heegaard states are in one-to-one correspondence with Kauffmann states. For alternating knots, this Heegaard diagram has the property that $M(\mathbf{\underbar{x}})-A(\mathbf{\underbar{x}})=\frac{\sigma(K)}{2}$ for any intersection point $\mathbf{\underbar{x}}$, where $A(\mathbf{\underbar{x}})$ is the Alexander grading and $M(\mathbf{\underbar{x}})$ is the Maslov grading of $\mathbf{\underbar{x}}$. As a result, all the differentials in the hat version of the knot Floer chain complex vanish. Ozsváth and Szabó go one step further and prove the following.  

\begin{prop}\cite{Szabó2003}
Let $K$ be an alternating knot, oriented so that $\sigma = \sigma(K)\leq 0$. For all $s>0$, let 
\[
\delta(\sigma,s):=max\left(0,\left\lceil \frac{|\sigma|-2|s|}{4} \right\rceil\right) 
\]
and ${t_s(K):= \sum\limits^{\infty}_{j=1} j a_{|s|+j}}$. Then we have a $\mathbb{Z}[U]$-module isomorphism 
\[
\HF^{+}(S_{0}^3(K),s)\cong \mathbb{Z}^{b_s} \oplus \left(\mathbb{Z}[U]/U^{\delta(\sigma,s)}\right),
\]
where $(-1)^{s+\frac{\sigma}{2}}b_s=\delta(\sigma,s)-t_s(K)$. Furthermore, for $s=0$, we have
\[
\HF^{+}(S_{0}^3(K),0)\cong \mathbb{Z}^{b_0} \oplus \mathcal{T}^{+}_{-1/2} \oplus \mathcal{T}^{+}_{-2\delta(\sigma,0)+\frac{1}{2}},
\]
where $\HF^+(S^3)=\mathcal{T}^{+}_{0}$ and the subscript denotes a shift in the $\mathbb{Q}$-grading.
\end{prop}

Their proof uses the large surgery formula to compute $\HF^{+}(S^3_{n}(K),s)$ for sufficiently large $n$, and then applies the surgery exact sequence: 
\[
\cdots \rightarrow \HF^{+}(S^3) \rightarrow \HF^{+}(S^3_{0}(K),s)\rightarrow \HF^{+}(S^3_{n}(K),[s]) \rightarrow \cdots
\]
Note that a corollary of this computation is that $b_s\geq 0$. Using the inequality $b_0 \geq 0$ for alternating knots of genus two implies the trapezoidal conjecture.

\begin{rema}
One needs to combine the large surgery formula with some homological algebra to calculate $\HF^{+}(S^3_{n}(K),s)$. The following (Mayer--Vietoris) exact sequence of subcomplexes of $\CFK^{\infty}(S^3,K)$ is useful:
\[
0 \rightarrow C\{i\geq 0 \ \text{or}\  j\geq s\} \rightarrow C\{i\geq 0\} \oplus C\{j \geq s\} \rightarrow C\{i\geq 0 \ \text{and}\  j\geq s\}\rightarrow 0
\]
\begin{figure}[h]
\centering
\includegraphics[scale=0.35]{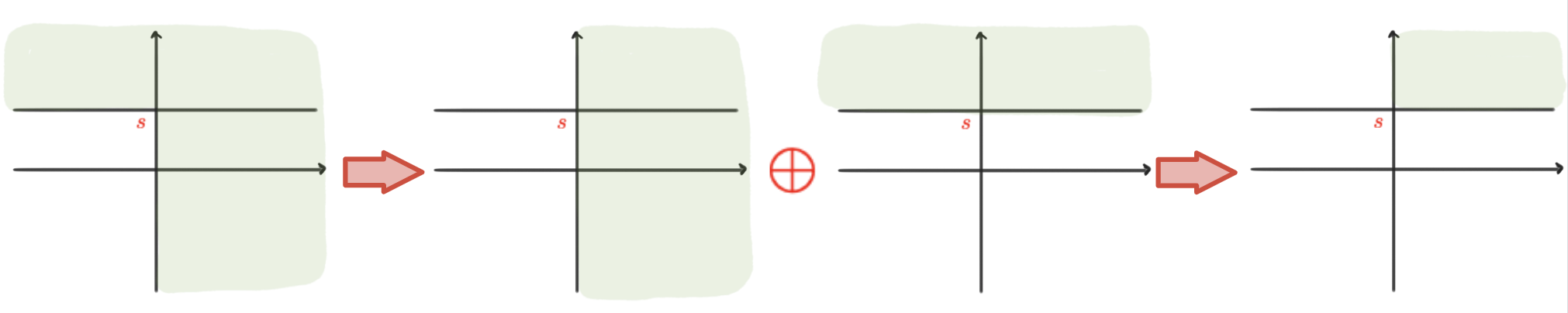}
\caption{Exact sequence of subcomplexes of $\CFK^{\infty}$.}\label{M-V large surgery}
\end{figure}\\
For an alternating knot, one can use $M(\mathbf{\underbar{x}})-A(\mathbf{\underbar{x}})=\frac{\sigma(K)}{2}$ to conclude that 
\[
M([\mathbf{\underbar{x}},i,j]) = M(\mathbf{\underbar{x}}) + 2i = A(\mathbf{\underbar{x}})+ \frac{\sigma}{2} +2i = (j-i)+\frac{\sigma}{2}+2i=i+j+ \frac{\sigma}{2}.
\]
In particular, this means that $H_{*}(C\{i\geq 0 \ \text{and}\  j\geq s\})$ is supported in degrees higher than $s + \frac{\sigma}{2}$. Combining this fact with the long exact sequence of homologies associated to the above short exact sequence results in: 
\[
\begin{split}
H_{\leq s +\frac{\sigma}{2}-2}(C\{i\geq 0 \ \text{or}\  j\geq s\}) &\cong H_{\leq s +\frac{\sigma}{2}-2}(C\{i\geq 0\}) \oplus H_{\leq s +\frac{\sigma}{2}-2}(C\{j \geq s\})\\
&\cong \HF^{+}_{\leq s +\frac{\sigma}{2}-2}(S^3) \oplus \HF^{+}_{\leq -s +\frac{\sigma}{2}-2}(S^3).
\end{split}
\]
Recall that, based on the large surgery formula, we have 
\[
H_{*}(C\{i\geq 0 \ \text{or}\  j\geq s\}) \cong \HF^{+}(S^3_{n}(K),[s])
\]
for $n$ sufficiently large.
\end{rema}

\textcolor{red}{\textbf{Ni}}~\cite{Ni} used these results to prove an extremal case of inequalities in the  trapezoidal conjecture. Let $K$ be an alternating knot and $\Delta_K(T)=\sum\limits_{i=-g}^{g} a_i T^i$. Then Ni showed that, if $|a_g|=|a_{g-1}|$, then $K$ or its mirror is the torus knot $T_{2g+1,2}$. This follows from the fact that knot Floer homology detects of fibredness.

\textcolor{red}{\textbf{Jong}}\cite{Jong} gave a combinatorial proof of the trapezoidal conjecture for alternating knots of genus at most two. This proof relies on Stoimenow's characterization of knot diagrams of genus two \cite{Stoimenowgenus2}. Analysis of Seifert graphs proves that all knot diagrams with Seifert genus two come from applying flypes and $\overline{t'_2}$ moves to a finite list of diagrams called generators.

\begin{figure}[ht]
\centering
\includegraphics[scale=0.4]{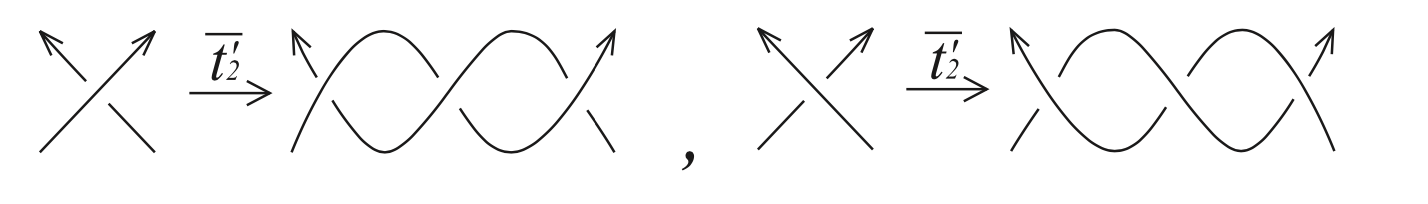}
\caption{The move $\overline{t'_2}$; see Jong~\cite{Jong}.}\label{t'2}
\end{figure}

Note that, by a theorem of Gabai~\cite{Gabai1986GeneraOT}, the Seifert genus of a reduced alternating diagram is equal to the genus of the knot. As a result, one can conclude that all alternating knots of genus two come from applying flypes and $\overline{t'_2}$ moves to alternating diagrams of Stoimenow's list of generators.

Jong combined this characterization with Crowell's spanning tree formula for the Alexander polynomial of alternating links. We explain this formula in detail in Subsection~\ref{stablisesection}. Jong calculated the effect of $\overline{t'_{2}}$ on Crowell's weighted graph and proved the trapezoidal conjecture by an inductive argument.

Note that Stoimenow's theorem is true for any fixed genus, so one can prove the trapezoidal conjecture by checking it for a finite list of generators. Performing such a task, even with the help of a computer, is not trivial. \textcolor{red}{\textbf{Stoimenow}}~\cite{Stoimenow2011DiagramGG} used several computational tricks to prove the trapezoidal conjecture for all alternating knots with genus less than or equal to $4$.

\textcolor{red}{\textbf{Alrefai and Chbili}}~\cite{small3braids} used Burau representation to derive closed-form formulas for the Alexander polynomial of alternating 3-braids with length less than or equal to $3$; i.e., for braid words of the form 
\[
\sigma_1^{p_1}\sigma_2^{-q_1}\sigma_1^{p_2}\sigma_2^{-q_2}\sigma_1^{p_3}\sigma_2^{-q_3}.
\]
Similar ideas were used by \textcolor{red}{\textbf{Alsukaiti and Chbili}}~\cite{alsukaiti2023alexander} to prove the conjecture for 3-braids of the form $(\sigma_1\sigma_2^{-1})^n$. 

\textcolor{red}{\textbf{Banfield}}~\cite{Banfield2022ChristoffelWA} used Minkus's formula for the Alexander polynomial of two-bridge link (equation~\ref{Minkus}) to prove that the sequence of coefficients of this polynomial is not only trapezoidal but log-concave. They noticed that Minkus's formula identifies this sequence with the vertical slices of folded Christoffel paths. The Christoffel path with slope $\frac{q}{p}$ is the path in a $p \times q$ grid which is closest to the diagonal and remains under it. The closeness here means that there are no grid points in the region bounded by the path and the diagonal. This path is the best "integer approximation" of the line with slope $\frac{q}{p}$; see Figure~\ref{Christoffel}. Banfield then used facts about the characterization of Christoffel words; i.e., words in $\{x,y\}^{*}$ generated by following steps of a Christoffel path. Using these, they constructed an inductive argument that completed the proof.

\textcolor{red}{\textbf{Hafner, Mészáros, and Vidinas}}~\cite{KarolaLogconcavityOT} proved the trapezoidal conjecture for special alternating knots. Special alternating knots are knots that can be constructed using the median construction on a bipartite plane graph $G$. Equivalently, these are alternating knots in which all crossings are positive. We talk about their method in detail in Subsection~\ref{3braids}.

\begin{figure}[ht]
\includegraphics[scale=0.4]{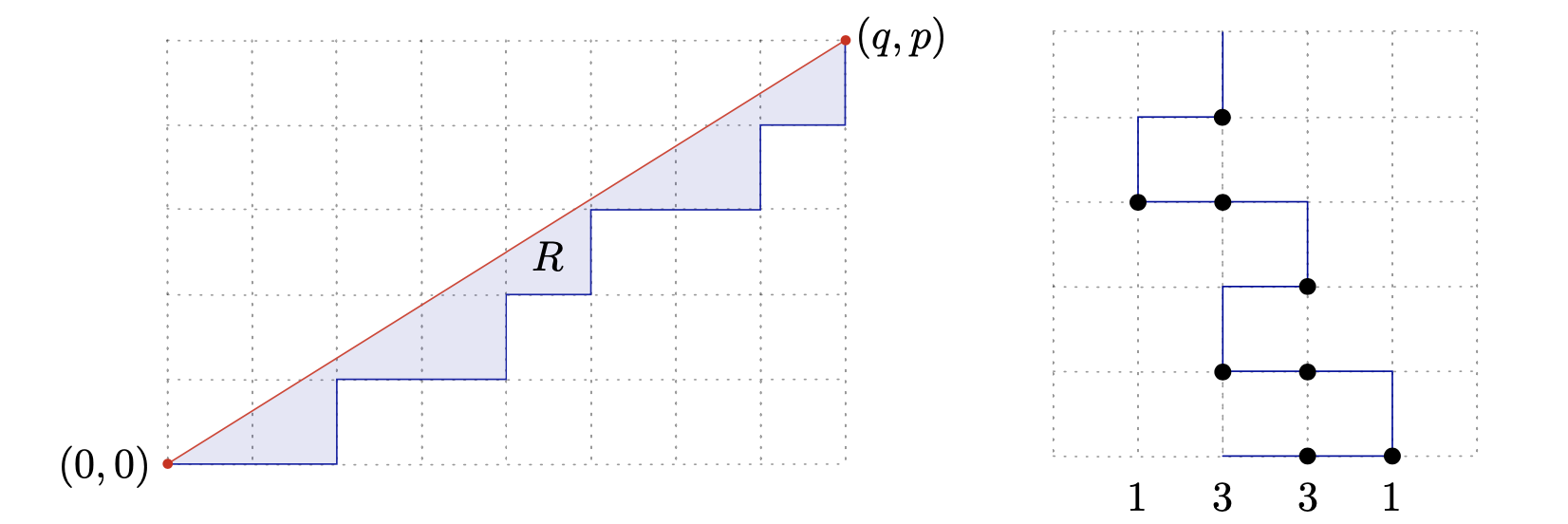}
\caption{Christoffel paths (ordinary and folded); see Banfield~\cite{Banfield2022ChristoffelWA}.}\label{Christoffel}
\end{figure}

\subsection{Stabilization and highly-twisted links}\label{stablisesection}

The aim of this subsection is to prove the trapezoidal conjecture for alternating links with (linearly) large twist regions; see Theorem~\ref{twistconcentrated}. A neighborhood in a link diagram is called an \textit{coherent twist region} if it is locally isomorphic to Figure~\ref{twistregion}. These twist regions can be constructed by twisting two parallel strands of a link. 

\begin{figure}[h]
\centering
\includegraphics[scale=0.35]{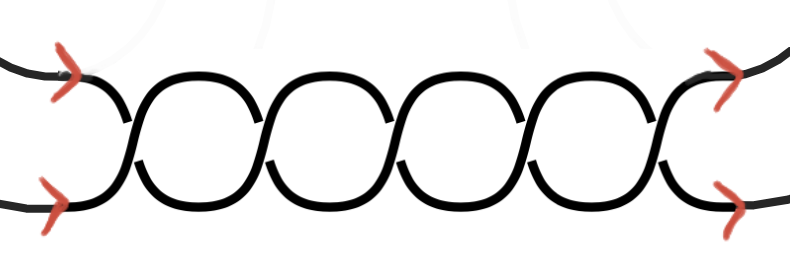}
\caption{A coherent twist region.}\label{twistregion}
\end{figure}

\begin{defi}\label{MTdef}
We define the \textit{maximal twist number} $\MT(L)$ of an alternating link $L$ to be the maximum of the number of crossings in an coherent twist region in any reduced alternating diagram of $L$.
\end{defi}


\begin{theo}\label{twistconcentrated}
Let $L$ be an alternating link and 
\[
\Delta_L(t)\dot{=}\sum\limits_{i=0}^{2g(L)+|L|-1} (-1)^{i}a_i t^i,
\]
where $a_i \ge 0$ for $i \in \{0,\dots,2g(L) + |L| - 1\}$.
If
\begin{enumerate}
    \item $MT(L)$ is odd and $\MT(L)-2 \geq g(L) + |L|/2$, or
    \item $MT(L)$ is even and $\MT(L)-3 \geq g(L) + |L|/2$,
\end{enumerate}
then the sequence $(a_i)$ is trapezoidal.
\end{theo}

The proof of Theorem~\ref{twistconcentrated} follows from a stabilization phenomenon for the Alexander polynomial's coefficients after applying a large number of twists. This was shown by Lambert-Cole \cite{LambertCole2016TwistingMA} using skein relations. Furthermore, Lambert-Cole generalized this to knot Floer homology. In the rest of this section, we reprove this fact for alternating links using Crowell's formula. This proof gives us some additional information about the stabilized values of the coefficients, which will in turn be used to prove Theorem~\ref{twistconcentrated}. First, we need to describe Crowell's formula for the Alexander polynomial of alternating links.

Let $D$ be an alternating diagram of a link $L$ with $m$ crossings $c_1,...,c_m$. We can assume that $m \geq 2$. Consider the underlying planar graph of the diagram $D$, which we denote by $G_D$ and call the \textit{Crowell graph} associated to $D$. The vertices of $G_D$ are crossings $c_1,...,c_m$ and the edges are the arcs of the diagram $D$. We orient the edges of this graph such that the two under-crossing (resp.\ over-crossing) arcs at each crossing correspond to incoming (resp.\ outgoing) edges. Furthermore, we label the edges with weights $t$ and $1$ such that, as one traverses the over-strand at crossing $c_i$ (following the link's orientation), the weight $t$ (resp.\ $1$) appears on the right (resp.\ left). We choose an arbitrary vertex, say $c_k$, of this graph as the root. Let $\mathcal{T}(G_D,c_k)$ be the set of all maximal oriented trees rooted at $c_k$. For a tree $T \in \mathcal{T}(G_D,c_k)$, we define its weight $W(T)$ to be the product of weights of all of its edges. The \textit{Crowell polynomial} associated to $(G_D,c_k)$ is 
\[
P_{G_D,c_k}(t) := \sum_{T \in \mathcal{T}(G_D,c_k)} W(T).
\]
Note that the Crowell polynomial $P_{G,v}(t)$ can be defined for any weighted, oriented graph $G$ rooted at a vertex $v$.

\begin{figure}[h]
\centering
\includegraphics[scale=0.4]{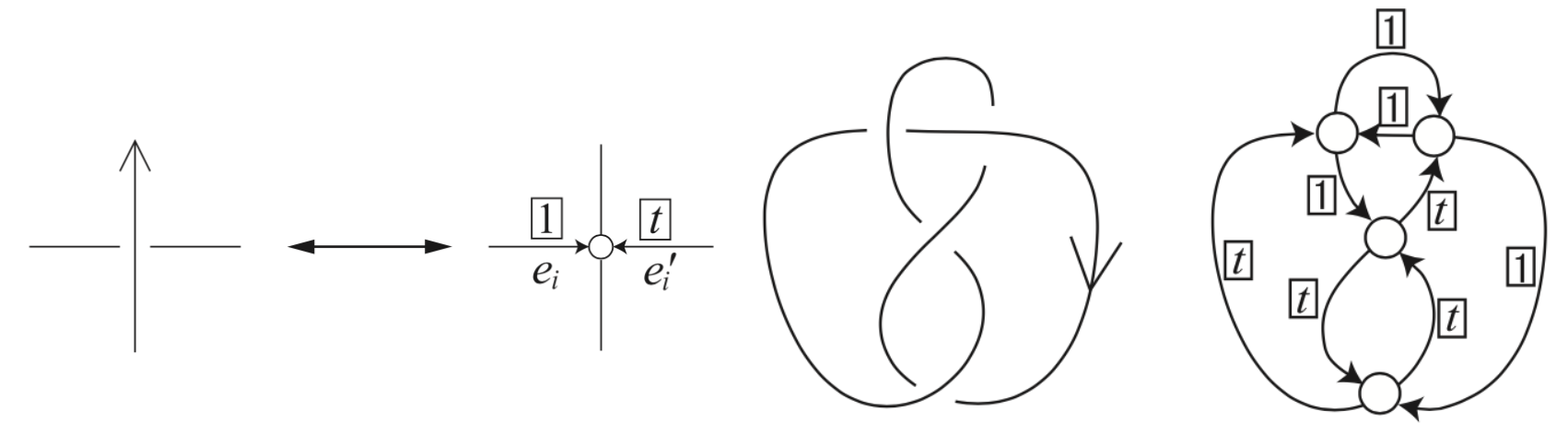}
\caption{Crowell's method~\cite{Jong}.}\label{}
\end{figure}

\begin{theo}\label{Crowell}\cite{Crowell}
    Let $\Delta_L(t)$ be the normalized Alexander polynomial of an alternating link $L$ with alternating diagram $D$. If $c_k$ is a crossing of $D$, then
    \[
    \Delta_{L}(-t) \dot{=} P_{G_D,c_k}(t),
    \]
    where $\dot{=}$ means equality up to multiplication by units of the Laurent polynomial ring $\mathbb{Z}[t,t^{-1}]$; i.e., by $\pm t^k$. 
\end{theo}

The absolute values of the coefficients of $\Delta_L(t)$ and $\Delta_L(-t)$ agree. In the rest of this subsection, we are going to focus on the coefficients of the Crowell polynomial.

Assume $\MT(L)=n$, and let $R$ be an coherent twist region with $n$ crossings $c_1,\dots,c_{n}$, labelled from left to right, in some reduced alternating diagram $D$ of $L$. We define $D_{k}$ as the link diagram constructed by replacing $R$ with an coherent twist region with $k$ crossings $c_1,\dots,c_{k}$, and let $L_k$ be the corresponding link. Let $G_{k}$ be the Crowell graph associated with $D_{k}$. Let $v_1,\dots, v_{k}$ be the vertices of $G_k$ corresponding to the crossings $c_1, \dots, c_{k}$. 

Consider the graph $G'_{k} = G_{k} \setminus \{v_{k-1}\}$ for $k \geq 3$, where we also remove the edges adjacent to $v_{k-1}$; see Figure~\ref{twistinggraph}. To define $G'_2$, we add an edge to $G_{2} \setminus \{v_{1}\}$, as follows. Consider the two edges to the left of $v_1$, and denote the incoming edge by $e_{li} := (v_{li},v_1)$ and the outgoing edge by $e_{lo} := (v_1,v_{lo})$. Let $G'_{2} := G_{2} \setminus \{v_{1}\} \cup \{e_{io}\}$, where $e_{io} := (v_{li},v_{lo})$. Let the weight of the new edge $e_{io}$ be the product of weights of $e_{li}$ and $e_{lo}$. We can view the graph $G'_{k}$ as the non-oriented smoothing of $G_{k}$ at $v_{k-1}$. Let $L_{k}(c_{k-1},||)$ denote the $1$-smoothing of $L_k$ at $c_{k-1}$. Ignoring the weights and the edge orientations, $G'_{k}$ is isomorphic to the Crowell graph corresponding to $L_{k}(c_{k-1},||)$, after deleting two loop edges when $k \ge 3$ and one loop edge when $k = 2$. Note that the two graphs have different weights and orientations; see Figure~\ref{twistinggraph}.

\begin{figure}[h]
\centering
\includegraphics[scale=0.4]{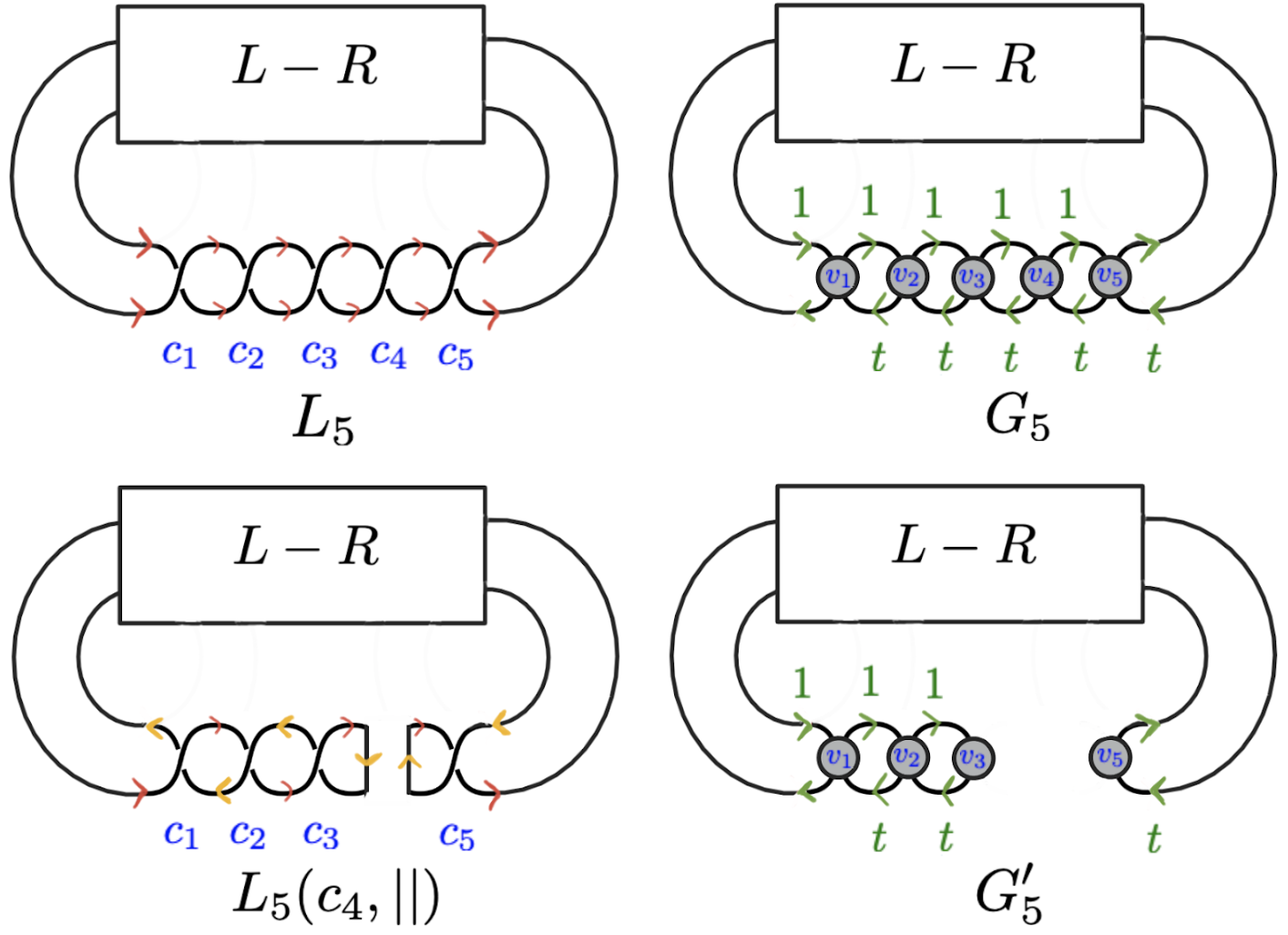}
\caption{Twist region in the Crowell graph.}\label{twistinggraph}
\end{figure}

The following lemma is the recursive relation at the heart of the stabilization phenomenon. It is an unoriented skein relation for the Crowell polynomial.

\begin{lemm}\label{CrowellSkein}
Let $P_{k}=P_{G_{k},v_{k}}$ and $P'_{k}=P_{G'_{k},v_{k}}$; i.e., the Crowell polynomials of the graphs $G_k$ and $G_k'$ rooted at $v_{k}$. Then
\[
P_{k}=tP_{k-1} + P'_{k}.
\]
\end{lemm}

\begin{proof}
We decompose $\mathcal{T}(G_{k},v_{k})$ into two sets $\mathcal{T}_R$ and $\mathcal{T}_L$ of oriented spanning trees depending the edge leading to $v_{k-1}$.

Let $\mathcal{T}_R$ be the set of oriented spanning trees containing the edge $(v_k,v_{k-1})$. By contracting this edge to the root vertex $v_k$ and deleting the resulting loop edge, one can build a one-to-one matching between this set and $\mathcal{T}(G_{k-1},v_{k-1})$. The matching decreases the exponent of $t$ in $W(T)$ by one; see Figure~\ref{TR}.

\begin{figure}[h]
\centering
\includegraphics[scale=0.3]{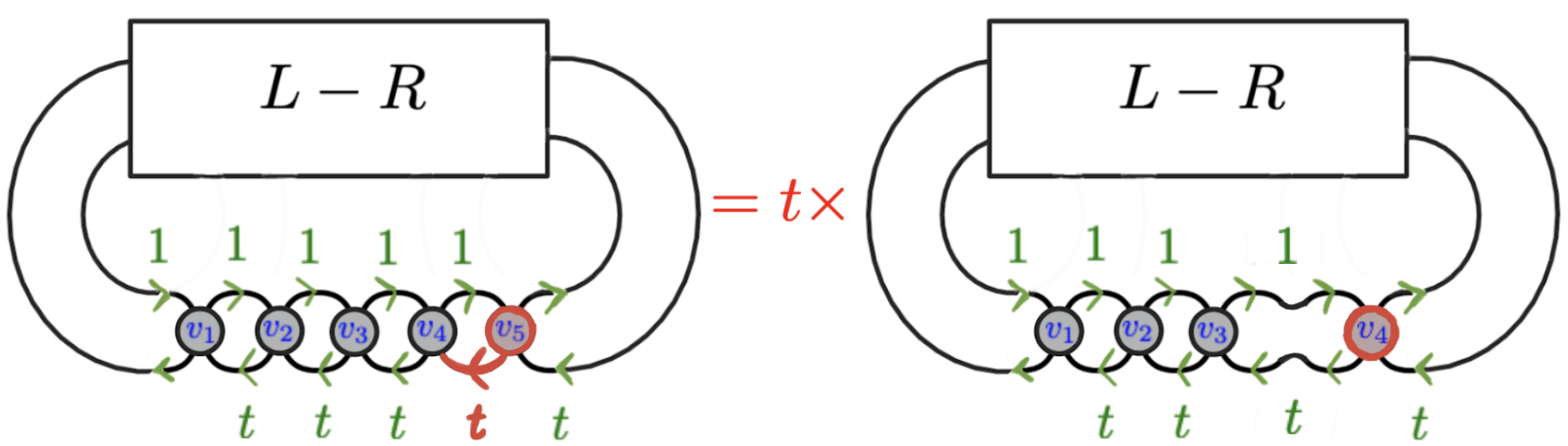}
\caption{The set $\mathcal{T}_R$ is in bijection with $\mathcal{T}(G_{k-1},v_{k-1})$. Trees containing the red edge are included in $\mathcal{T}_R$.}\label{TR}
\end{figure}

We write $\mathcal{T}_L$ for the set of oriented spanning trees containing the edge $(v_{k-2},v_{k-1})$. These trees do not contain any other edge adjacent to $v_{k-1}$. As a result, $\mathcal{T}_L$ is in bijection with $\mathcal{T}(G'_{k},v_{k})$. This bijection preserves the weights; see Figure~\ref{TL}.

\begin{figure}[h]
\centering
\includegraphics[scale=0.3]{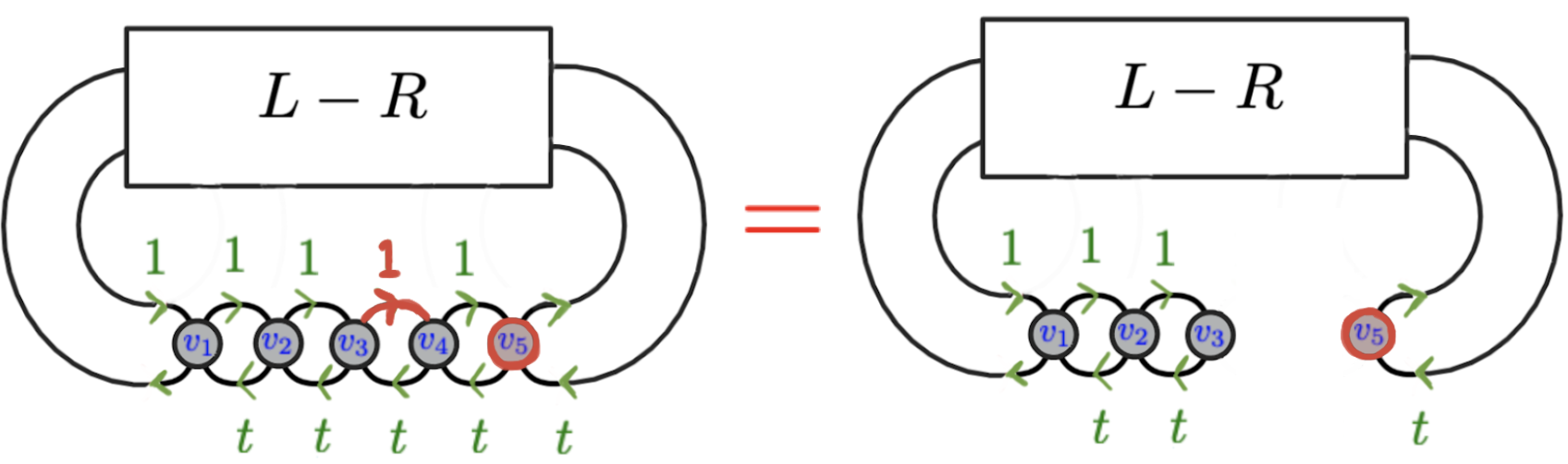}
\caption{The set $\mathcal{T}_L$ is in bijection with $\mathcal{T}(G'_{k},v_{k})$. Trees containing the red edge are included in $\mathcal{T}_L$.}\label{TL}
\end{figure}

By definition of $P_k$, we have 
\[
\begin{split}
P_k &=\sum_{T \in \mathcal{T}(G_{k},v_{k})} W(T) \\
&= \sum_{T \in \mathcal{T}_R} W(T)+\sum_{T \in \mathcal{T}_L} W(T) \\
&= \sum_{T \in \mathcal{T}(G_{k-1},v_{k-1})} t W(T)+\sum_{T \in \mathcal{T}(G'_{k},v_{k})} W(T)\\
&= t P_{k-1} + P'_{k},
\end{split}
\]
as claimed.
\end{proof}

\begin{defi}\label{half-trap}
We say that the sequence of positive integers $(a_1,\ldots,a_m)$ is \textit{half-trapezoidal} if
    \begin{enumerate} 
    \item $a_1 \leq a_2 \leq \dots \leq a_{m-1} \leq a_m $;
    \item if $a_i = a_{i+1}$, then $a_i = a _{i+1} = \dots = a_m$;
    \end{enumerate} 
    i.e., if it strictly increases and then possibly stabilizes. 
\end{defi}

Note that, if $(a_1,\ldots,a_n)$ is a symmetric sequence of positive integers and $(a_1,\dots,a_m)$ is half-trapezoidal for $m=\lfloor n/2 \rfloor$, then $(a_1,\ldots,a_n)$ is trapezoidal.

\begin{proof}[Proof of Theorem~\ref{twistconcentrated}]
Using the notation of Theorem~\ref{Crowell} and Lemma~\ref{CrowellSkein}, for $L = L_n$, we have $\Delta_L(-t) \dot{=} P_n$. Note that when $L_n$ is alternating, $L_1$ is also alternating and hence $\Delta_{L_1}(-t) \dot{=} P_1$.

We prove the Theorem~\ref{twistconcentrated} when $n$ is odd. We are going to use the fact that $|L_n|=|L_1|$ for odd $n$. The same argument can be repeated for even $n$, but one needs to use $L_2$ instead of $L_1$.

We show that the sequence of coefficients of $P_n$ is trapezoidal if $n-2 \geq g(L) + |L|/2$. By applying Lemma~\ref{CrowellSkein} repeatedly, we can deduce that
\begin{equation}\label{iterated}
P_n=t^{n-1}P_1+(P'_n + t P'_{n-1}+ \dots + t^{n-2} P'_{2}).
\end{equation}
Note that the only difference between $G'_k$, for different values of $k$, is the length of the chain of vertices $v_1,\dots,v_{k-2}$. 
Any tree in $\mathcal{T}(G'_{k},v_{k})$ must contain the edges $(v_1,v_2),\dots,(v_{k-3},v_{k-2})$. 
Deleting the vertices $v_1,\dots,v_{k-2}$ and the incident edges from each maximal rooted tree, then replacing $e_{li}$ with $e_{io}$ if both $e_{li}$ and $e_{io}$ are in the original tree results in a one-to-one correspondence between $\mathcal{T}(G'_{k},v_{k})$ and $\mathcal{T}(G'_{2},v_{2})$; see Figure~\ref{smoothingcollapse}.

This correspondence preserves weights, so $P'_k = P'_2$. Combining this with equation~\eqref{iterated} gives that 
\begin{equation}\label{iterated2}
P_n=t^{n-1}P_1+(1+\dots+t^{n-2})P'_2.
\end{equation}
Let $\mindeg(P)$ and $\maxdeg(P)$ be the minimal and maximal power of $t$ that appear in a polynomial $P(t)$, respectively. Let $(a'_i)$ denote the sequence of coefficients of $P'_2$ and $(c'_i)$ denote the sequence of coefficients of ${(1+\cdots+t^{n-2})P'_2}$. The first $n-1$ terms of $(c'_i)$ are the cumulative sums of the sequence $(a'_i)$; i.e., 
\[
c'_{\mindeg(P'_2)+j} = a'_{\mindeg(P'_2)}+\dots+a'_{\mindeg(P'_2)+j}
\]
for all $j \in \{0, \dots, n-2\}$.
Note that the coefficients of Crowell polynomials are positive; see Remark~\ref{positivecoeff}. As a result, their cumulative sums form a half-trapezoidal sequence. This means that the first $n-1$ coefficients of $(1+\dots+t^{n-2})P'_2$ form a half-trapezoidal sequence.

\begin{figure}[h]
\centering
\includegraphics[scale=0.3]{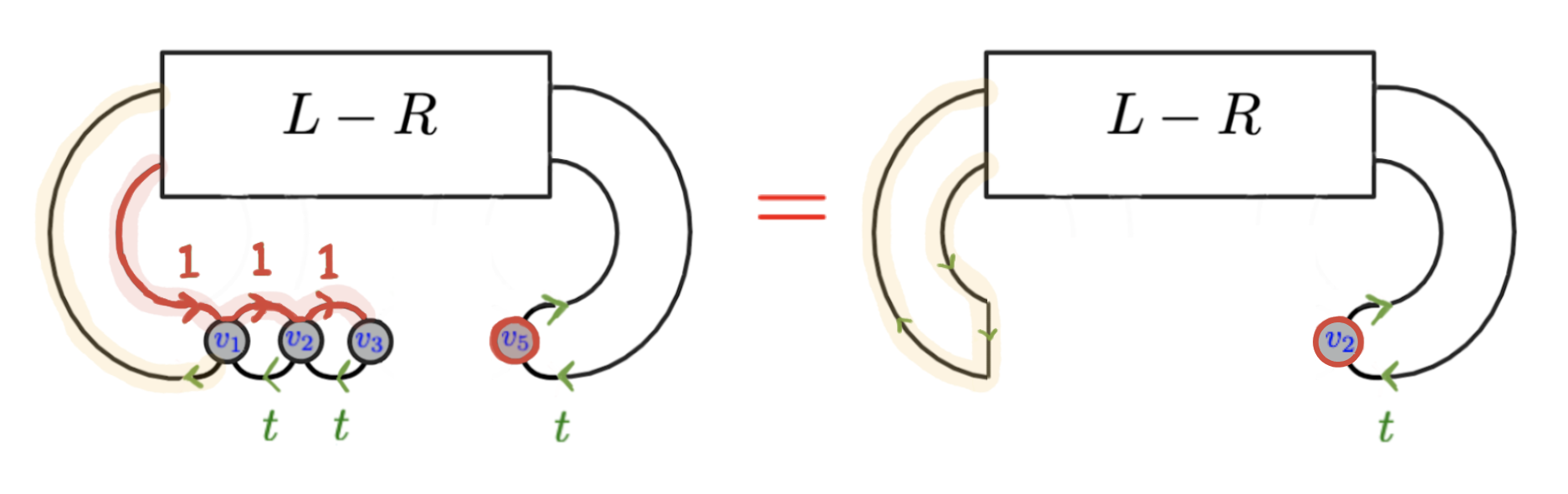}
\caption{The bijection between $\mathcal{T}(G'_{k},v_{k})$ and $\mathcal{T}(G'_{2},v_{2})$. The oriented spanning trees in $\mathcal{T}(G'_{k},v_{k})$ contain the red edges. The edges $e_{lo}$ and $e_{io}$ are shown in yellow.}\label{smoothingcollapse}
\end{figure}

We are going to show that the first $n-2$ coefficients of $P_n$ agree with the first $n-2$ terms of $(c'_i)$ and hence they are also half-trapezoidal. We need to separate the terms that $(1+\dots+t^{n-2})P'_2$ contribute to $P_n$ from the ones coming from $t^{n-1}P_1$. To be more precise, we will prove that
\begin{equation}\label{seperatedterms}
n-1+ \mindeg(P_1) = \mindeg(t^{n-1}P_1) = \mindeg(P_n)+n-1.
\end{equation}
In fact, one can prove that, for large enough $n$, we have $\mindeg(P_n) = \mindeg(P_1)$ by analysing spanning trees, but to get to our desired lower bound for $n$, we use an alternative argument. We are going to show that
\begin{equation}\label{highdegeq}
\maxdeg(P_n) = \maxdeg(t^{n-1}P_1) \text{, and}
\end{equation}
\begin{equation}\label{Spaneq}
\Span(P_n) - \Span(t^{n-1}P_1)=n-1.  
\end{equation}
Equations~\eqref{highdegeq} and~\eqref{Spaneq} and the fact that $\Span(P) = \maxdeg(P) - \mindeg(P)$ give us equation~\eqref{seperatedterms}.

Note that any oriented spanning tree in $\mathcal{T}(G'_{2},v_{2})$ is also a spanning tree in $\mathcal{T}(G_{1},v_{1})$, which means all the terms in $P'_2$ also contribute to in $P_1$. As a result,
$\maxdeg(P_1) \geq \maxdeg(P'_2)$, hence
\[
\maxdeg(t^{n-1}P_1) > \maxdeg \bigl((1+\dots+t^{n-2})P'_2 \bigr).
\]
Combining this with equation~\eqref{iterated2} implies that the term with the highest degree in $P_n$ comes from $t^{n-1}P_1$. This is equation~\eqref{highdegeq}.
 
We can compare the span of $P_1$ and $P_n$ by comparing the genera of $L_1$ and $L_n$. The links are both alternating and the span of the Alexander polynomial of an alternating link $J$ is known to be $2g(J) + |J| - 1$. We can write  
\[
\Span(P_n) - \Span(P_1) = 2(g(L_n)-g(L_1)) + |L_n| - |L_1|.
\]
Since $n$ is odd, $L_n$ and $L_1$ differ in a number of full twists and $|L_n| = |L_1|$. To compare their genera, we can compare their Seifert graphs (i.e., the result of the Seifert algorithm). The Seifert graph of $L_n$ has $n$ parallel edges (i.e., half-twisted bands)  between two vertices (i.e., Seifert cycles) where $L_1$ has one edge; see Figure~\ref{seifertcycs}. Since $L_n$ and $L_1$ are alternating, their genera is realised by the Seifert algorithm and is equal to half of the rank of the first homology of the Seifert graph.  We can deduce that $g(L_n)-g(L_1)=\frac{n-1}{2}$ which gives us equation~\eqref{Spaneq}.

\begin{figure}[h]
\centering
\includegraphics[scale=0.3]{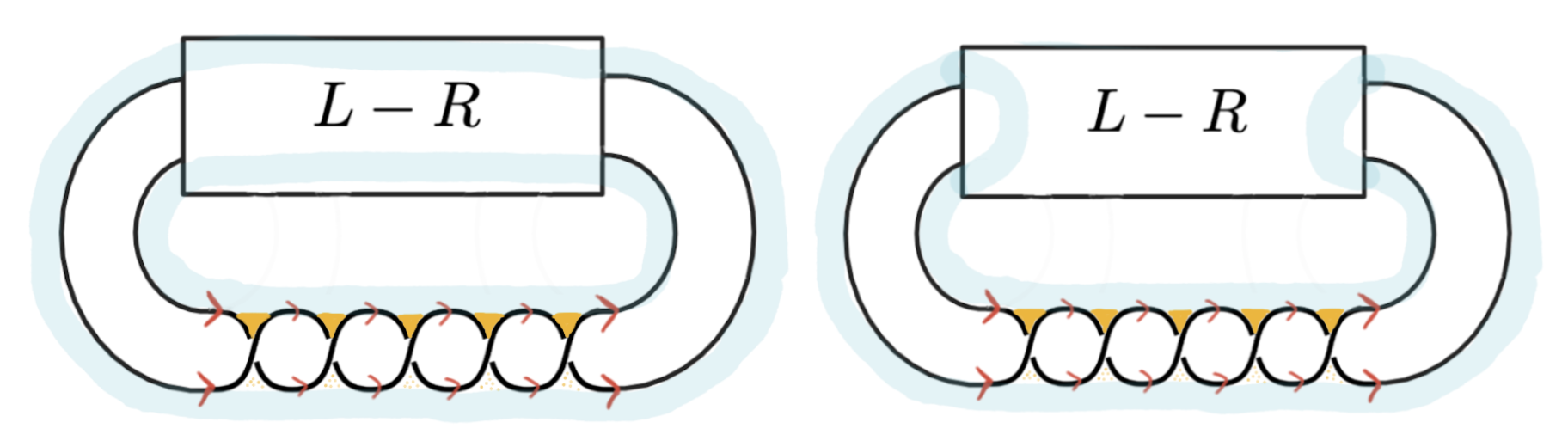}
\caption{A twist region is a family of parallel bands between a pair of Seifert cycles (not necessarily different cycles as seen in the right figure).}\label{seifertcycs}
\end{figure}
 
We have proved that the first $n-2$ coefficients of $P_n$ come from 
 \[
 {(1+\dots+t^{n-2})P'_2} 
 \]
 and are hence half-trapezoidal. Furthermore, $\Span(P_n) = 2g(L_n) + |L_n| - 1$. Now, using the assumption that $n-2 \geq g(L_n) + |L_n|/2$, we can deduce that all the coefficients in the first half of $\Span(P_n)$ are half-trapezoidal. The symmetry of $P_n$ means that it is trapezoidal. 
\end{proof}

\begin{rema}\label{positivecoeff}
  It is well know that the coefficient of $t^i$ in the Crowell polynomial $P$ of an alternating knot is positive for $i \in [\mindeg(P), \maxdeg(P)]$. It is non-negative by definition and proven to be non-zero. It turns out that this is true for the Crowell polynomial of any weighted, oriented, rooted graph $(G,v)$. This follows from the fact that $\mathcal{T}(G,v)$ is the basis of a greedoid and hence it has a connected adjacency graph (See Corollary 5.7. of \cite{BJORNER198544}). This means that, for any $T$, $T' \in \mathcal{T}(G,v)$, there exist two ordered sets of edges $\{e_1,\dots,e_n\} \subseteq E(T)$ and $\{e'_1,\dots,e'_n\} \subseteq E(T')$ such that 
  \[
  T \setminus \{e_1,\dots,e_t\} \cup \{e'_1,\dots,e'_t\} \in \mathcal{T}(G,v)
  \]
  for every $t \in \{1, \dots, n\}$, and
  \[
  T \setminus \{e_1,\dots,e_n\} \cup \{e'_1,\dots,e'_n\} = T'.
  \]
\end{rema}

We end Subsection~\ref{stablisesection} by mentioning some corollaries of Theorem~\ref{twistconcentrated}. First, the proof of Theorem~\ref{twistconcentrated} immediately implies a more general result:

\begin{theo}\label{twistconcentratedgeneral}
Let $L$ be an alternating link and 
\[
\Delta_L(t)\dot{=}\sum\limits_{i=0}^{2g(L)+|L|-1} (-1)^{i}a_i t^i. 
\]
The sequence $\{a_i \ :  \ 0\leq i \leq \MT(L)-3\}$ is half-trapezoidal; i.e., it is strictly increasing until it possibly stabilizes.  
\end{theo}

It is not straightforward to find knots that satisfy the conditions of Theorem~\ref{twistconcentrated}. If $L$ is a non-trivial link, then $c(L)/2 \ge g(L) + |L|/2$, and it is easy to check whether the stronger condition ${\MT(L)-3 \geq c(L)/2}$ holds. For instance, if $B=\sigma_{i_1}^{n_1} \cdots \sigma_{i_k}^{n_k}$ is an alternating braid, then $\MT(L) = \max \{n_j : 1 \leq j \leq k\}$. We obtain the following corrollary:

\begin{coro}\label{twisconc3braids}
    Let $L_B$ be the closure of the braid $B=\sigma_1^{p_1}\sigma_2^{-q_1}\cdots\sigma_1^{p_n}\sigma_2^{-q_n}$. Assume that 
    \[
    p_1 \geq \sum\limits_{i=2}^n p_i + \sum_{j=1}^n q_j + 6.
    \]
    Then the trapezoidal conjecture holds for $L_B$. 
\end{coro}

We will see some other corollaries of Theorem~\ref{twistconcentrated} in the form of bounds on the length of the stable part in Section~\ref{Section2}.
In light of the results of this subsection, the following definition is useful.

\begin{defi}\label{twistconcentrateddefi}
    We call a link $L$ \textit{twist-concentrated} if it satisfies the inequality 
    \[
    \MT(L)- 3 \ge g(L)+ |L|/2.
    \]
\end{defi}

\subsection{Plumbings of special alternating knots}\label{plumbingofspecial}

In this subsection, we prove the trapezoidal conjecture for diagrammatic plumbings of special alternating links, using the fact that the conjecture holds for special alternating links~\cite{KarolaLogconcavityOT}. We explain the proof for a plumbing of a pair of special alternating links, but it can be extended to all diagrammatic plumbings. First, we define the plumbing operation, and, more generally, Murasugi sums.
The following is the classical definition of Murasugi sum; see \cite{Kawauchi1996ASO}, \cite{Gabai1983TheMS} and \cite{Hongler2005AMD}.

\begin{defi}\label{GabaiMurasugi}
Consider two links $L_1$ and $L_2$ in $S^3$ with Seifert surfaces $S_1$ and $S_2$. Let $S^3 = B_1 \cup_{S^2} B_2$ be the standard decomposition of the 3-sphere into 3-balls. Assume $S_{i} \subset B_{i}$ such that  $S_{1} \cap S_{2}$ is a disk ${D \subset S^2}$. The disk $D$ is the interior of a $2n$-gon; i.e., $\partial D$ is decomposed into $2n$ arcs $a_1,b_1,\dots,a_n,b_n$, labelled counter-clockwise, such that $a_i \subseteq \partial S_1$ and $a_i$ is properly embedded in $S_2$, and $b_i \subseteq \partial S_2$ and $b_i$ is properly embedded in $S_1$ for $i \in \{1,\dots,n\}$. The compact oriented surface $S=S_1 \cup_{D} S_2$ is called the Murasugi sum of $S_1$ and $S_2$ along the disk $D$. The link $L=\partial S$ is called the Murasugi sum of $L_1$ and $L_2$ along $D$; see Figures~\ref{Murasugisum} and~\ref{6gonMurasugisum}. Define the \emph{gonality} of the Murasugi sum to be $2n$.
\end{defi}

\begin{figure}[h]
\centering
\includegraphics[scale=0.4]{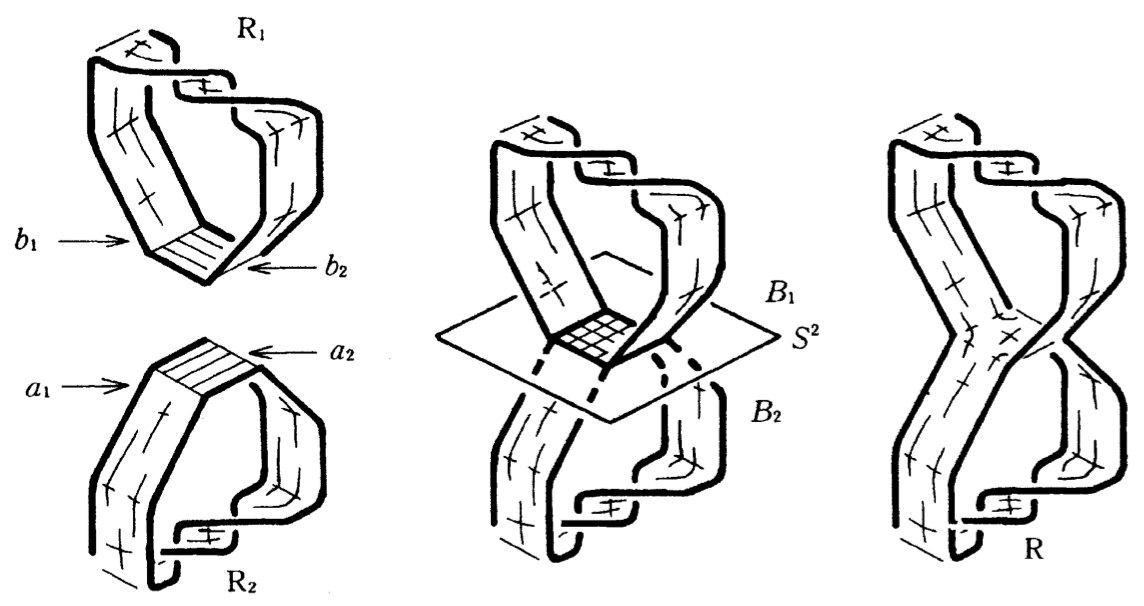}
\caption{A 4-gon Murasugi sum; i.e., plumbing~\cite{Kawauchi1996ASO}.}\label{Murasugisum}
\end{figure}

\begin{figure}[h]
\centering
\includegraphics[scale=0.3]{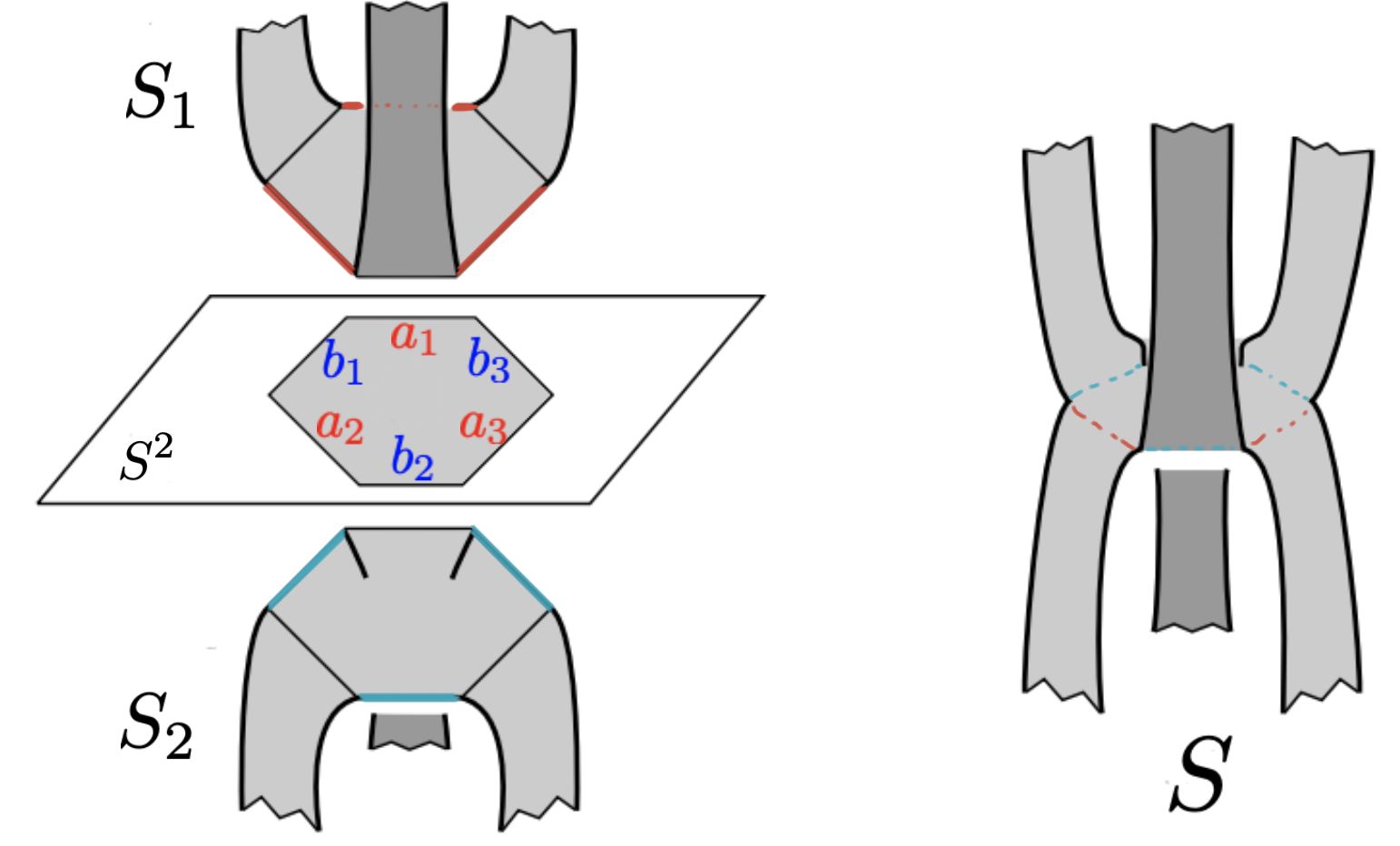}
\caption{A 6-gon Murasugi sum~\cite{able2021construction}.}\label{6gonMurasugisum}
\end{figure}

Taking 2-gon Murasugi sum is simply the connected sum operation. The 4-gon Mutasugi sum operation is called \textit{plumbing}. As mentioned before, links that can be constructed as boundaries of plumbings of twisted bands over a tree are called \textit{arborescent links}. Figures~\ref{skein} and~\ref{plumbing-tait} show simple examples of arborescent links. General arborescent links look like Figure~\ref{generalalgebraic}. Another class of such links is that of 2-bridge links; see Figure~\ref{2bridgealgebraic}.

\begin{figure}[h]
\centering
\includegraphics[scale=0.4]{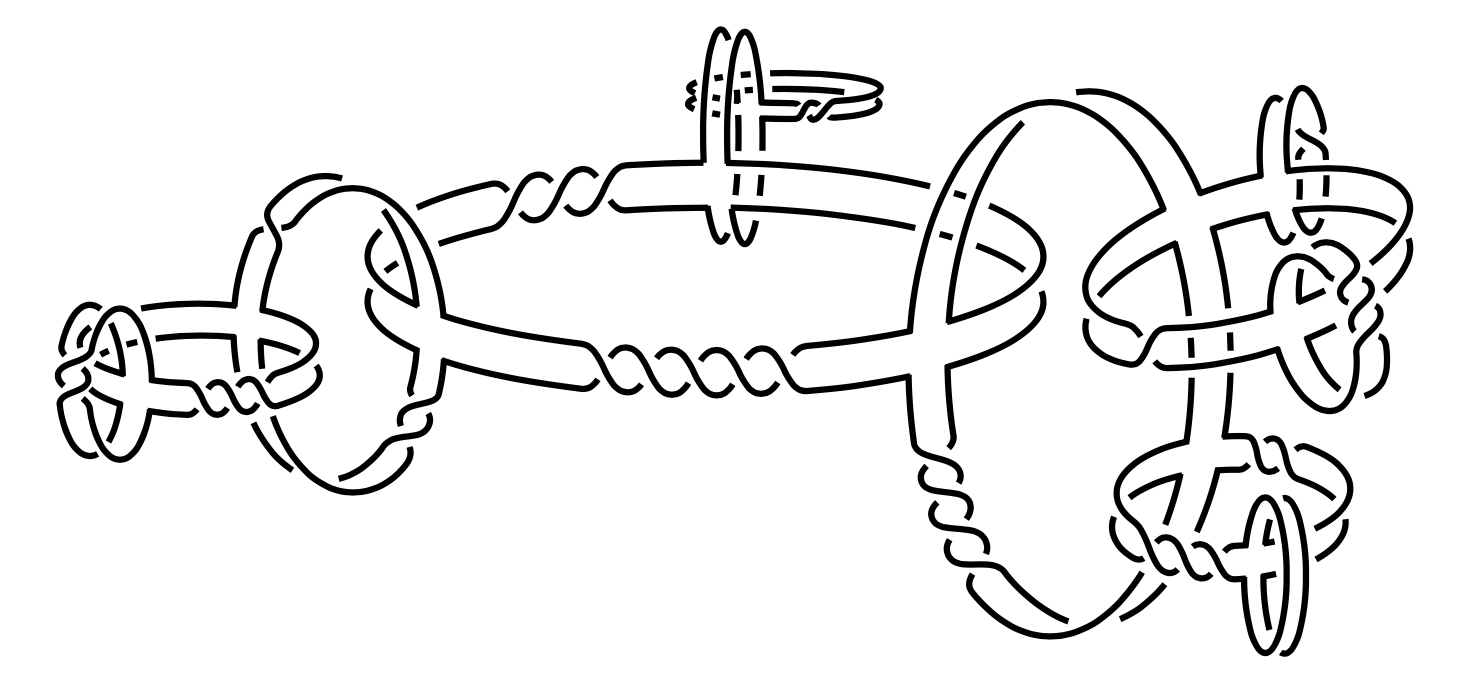}
\caption{A general arborescent link \cite{BonSieb}.}\label{generalalgebraic}
\end{figure}

\begin{figure}[h]
\centering
\includegraphics[scale=0.4]{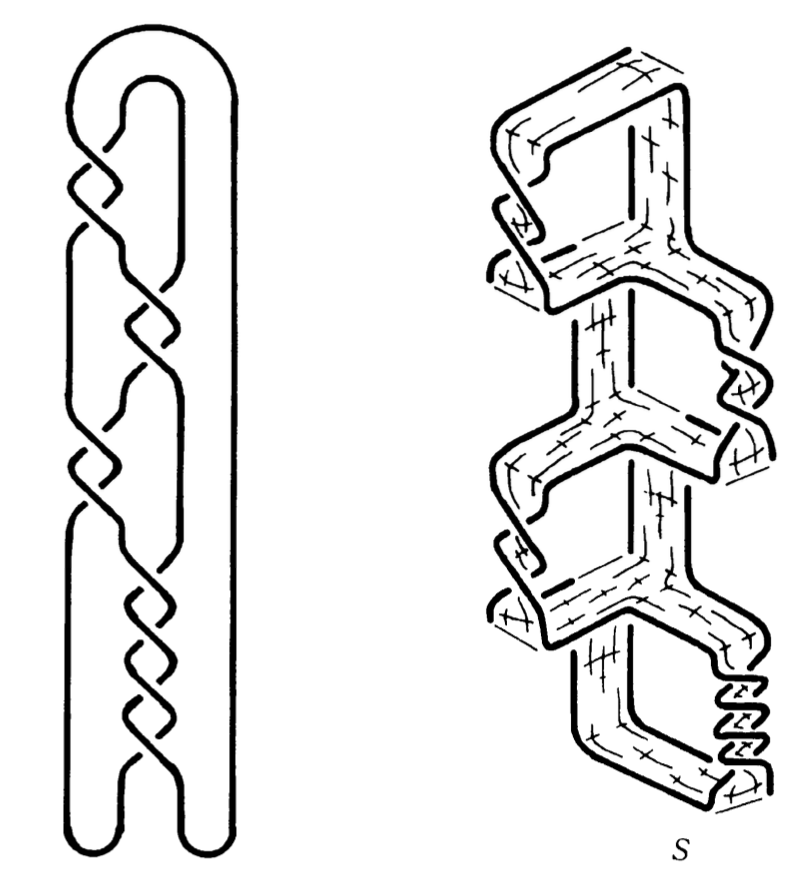}
\caption{An illustration showing why 2-bridge links are arborescent \cite{Thurston2bridge}.}\label{2bridgealgebraic}
\end{figure}

Gabai proved a series of important results about Murasugi sums, which we summarize in the following theorem.

\begin{theo}\label{Gabai}\cite{Gabai1983TheMS}
The Murasugi sum of two incompressible (resp.\ minimal genus) Seifert surfaces is incompressible (resp.\ minimal genus), and, as a result, genus is additive under Murasugi sums along minimal genus Seifert surfaces. Furthermore, $S$ is a fibre surface for $L$ if and only if $S_i$ is a fibre surface for $L_i$ for $i \in \{1,2\}$.
\end{theo}

%

Definition~\ref{GabaiMurasugi} is the most general form of the Murasugi sum operation, but it has some limitations. One of these is that the operation cannot easily be seen in a diagram. Murasugi~\cite{Murasugisums} originally used a more restrictive definition, which we call \emph{diagrammatic Murasugi sum}. We now recall that definition; see also \cite{Hongler2004OnTT} and \cite{Hongler2005AMD}. It is more convenient to phrase it in terms of a \textit{Murasugi decomposition algorithm} that decomposes any link diagram as a Murasugi sum, as follows.

Consider a link $L$ with a connected diagram $D_L$, which consists of crossings $c_1,\dots,c_n$ and arcs $e_1,\dots,e_{2n}$. One can also view the diagram as a 4-valent plane graph with vertices $c_i$, edges $e_j$, and regions $r_1,\dots,r_{n+2}$. The set $\{e_1,\dots,e_{2n}\}$ can be decomposed into Seifert cycles $T_1,\dots,T_q$ using Seifert's algorithm. One can split Seifert cycles into two types. We say that $T_i$ is a \textit{type~1} Seifert cycle if it bounds a region $r_k$, otherwise it is \textit{type~2}.

If all the Seifert cycles are type~1, then we call the link diagram \textit{special}. Applying Seifert's algorithm to a special diagram gives a Seifert surface which is obtained by considering the disks bounded by the Seifert cycles of $L$ and attaching half-twisted bands over each crossing of $L$. Note that this definition is the same as the median construction applied to the adjacency graph of the regions of the same color in a checkerboard coloring of $S^2 \setminus D_L$. We call this graph the \textit{Tait graph}. In the literature, this surface is also known as the \emph{checkerboard} or \emph{Tait surface}. For the Tait surface to be orientable, the Tait graph needs to be bipartite. Figure~\ref{turnspecial} shows a special diagram and a two-coloring (blue and green) of the vertices of the Tait graph (regions with the same checkerboard color).
Any link has a special diagram. See Burde~\cite{Burdespecialalgorithm}, Stoimenow~\cite{Stoimenow2011DiagramGG}, and Hirasawa~\cite{Hirasawaspecialdiagram} for algorithms to turn any diagram of a link into a special one, which is shown in Figure~\ref{turnspecial}.

\begin{figure}[h]
\centering
\includegraphics[scale=0.3]{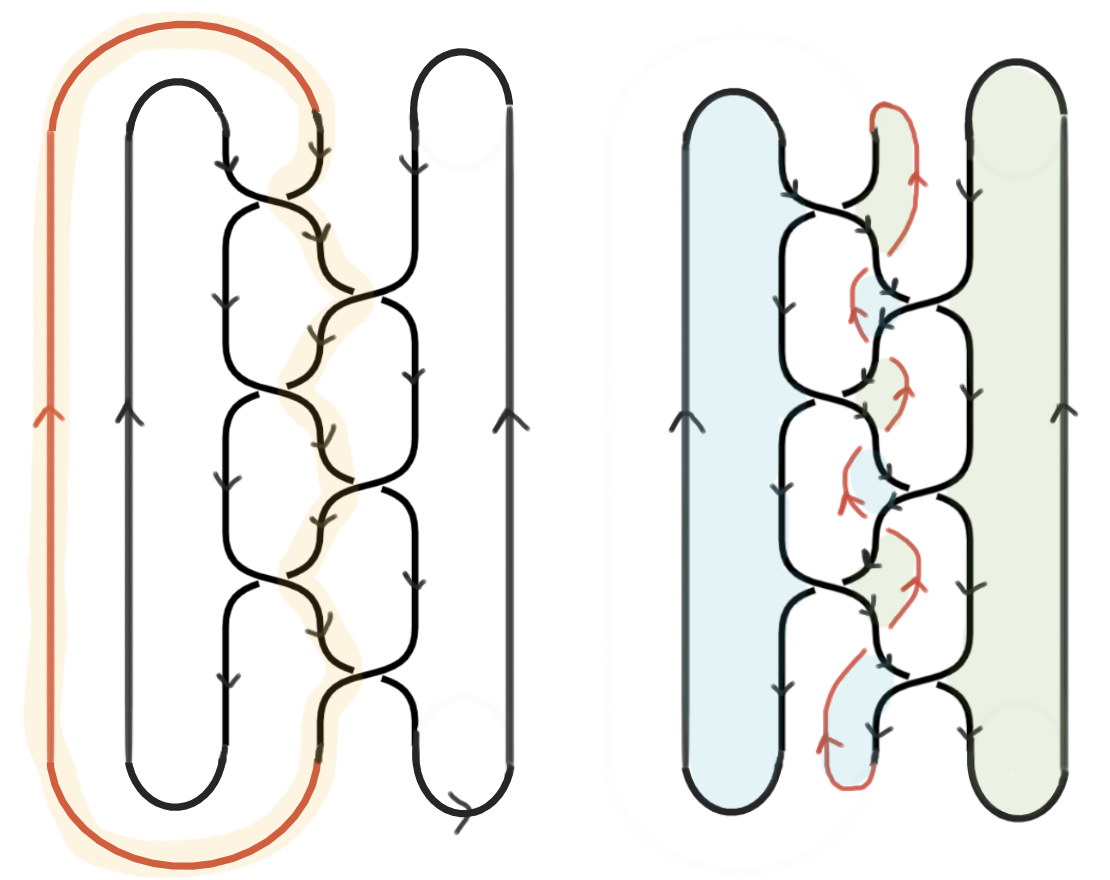}
\caption{Non-special and special diagrams of a link. The left diagram is non-special as it contains a type~2 Seifert cycle, colored yellow. One can turn the diagram into the special diagram on the right by moving the red segment.}\label{turnspecial}
\end{figure}

Let $C_1,\dots,C_m$ be the type~2 Seifert cycles in $D$. Each $C_i$ is a simple closed curve in $S^2$ that divides $S^2$ into disks $C_i^{+}$ and $C_i^{-}$. Let $C_i^{+}$ be the disk containing $\infty$. For $\gamma_i \in \{+,-\}$, let $D(\gamma_1,\cdots,\gamma_m)$ be the closure of $C_1^{\gamma_1} \cap \dots \cap C_m^{\gamma_m}$. Exactly $m+1$ of these sets will be nonempty, which we call \emph{Seifert domains} and denote by $D_1,\dots,D_{m+1}$; see Figure~\ref{Seifertdomains}. For every $i \in \{1,\dots,n\}$, at least two of the four regions adjacent to crossing $c_i$ belong to the same Seifert domain $D_j$ for some $j \in \{1,\dots,m+1\}$. In this case, we say that $c_i$ \textit{belongs} to $D_j$. Each vertex belongs to exactly one Seifert domain. For $j \in \{1,\dots,m+1\}$, let $L_j$ be the link obtained by smoothing all the crossings along $\partial D_j$ that do not belong to $D_j$. We call $\{L_1,\dots,L_{m+1}\}$ the \textit{Murasugi decomposition} of the diagram $D_L$, and denote it by 
\[
L=L_1 * L_2 * \cdots * L_{m+1}.
\]
Note that one can construct the links $L_j$ by adding vertical segments in the middle of the half-twisted bands which are used in Seifert's algorithm; see Figure~\ref{Murasugidecompose}. In this way, one can see that gluing the canonical Seifert surfaces for $L_1,..., L_{m+1}$ results in a canonical Seifert surface for $L$. It is not trivial that this gluing is a Murasugi sum in the sense of Definition~\ref{GabaiMurasugi}; see Remark~\ref{diagrammaticisaspecialcase}. For an example of a Murasugi decomposition; see Figure~\ref{Murasugidecompose}.

\begin{rema}\label{diagrammaticisaspecialcase}
Murasugi's $*$-product is a special case of Murasugi sum in the sense of Definition~\ref{GabaiMurasugi}. We do not give a proof of this fact, but we show it holds for the example presented in Figure~\ref{Murasugidecompose}. This can be seen in Figure~\ref{Murasugisumandstar}. Another example of this fact can be seen in Figure~\ref{Murasugisumandstar2}. 
\end{rema}

\begin{figure}[h]
\centering
\includegraphics[scale=0.3]{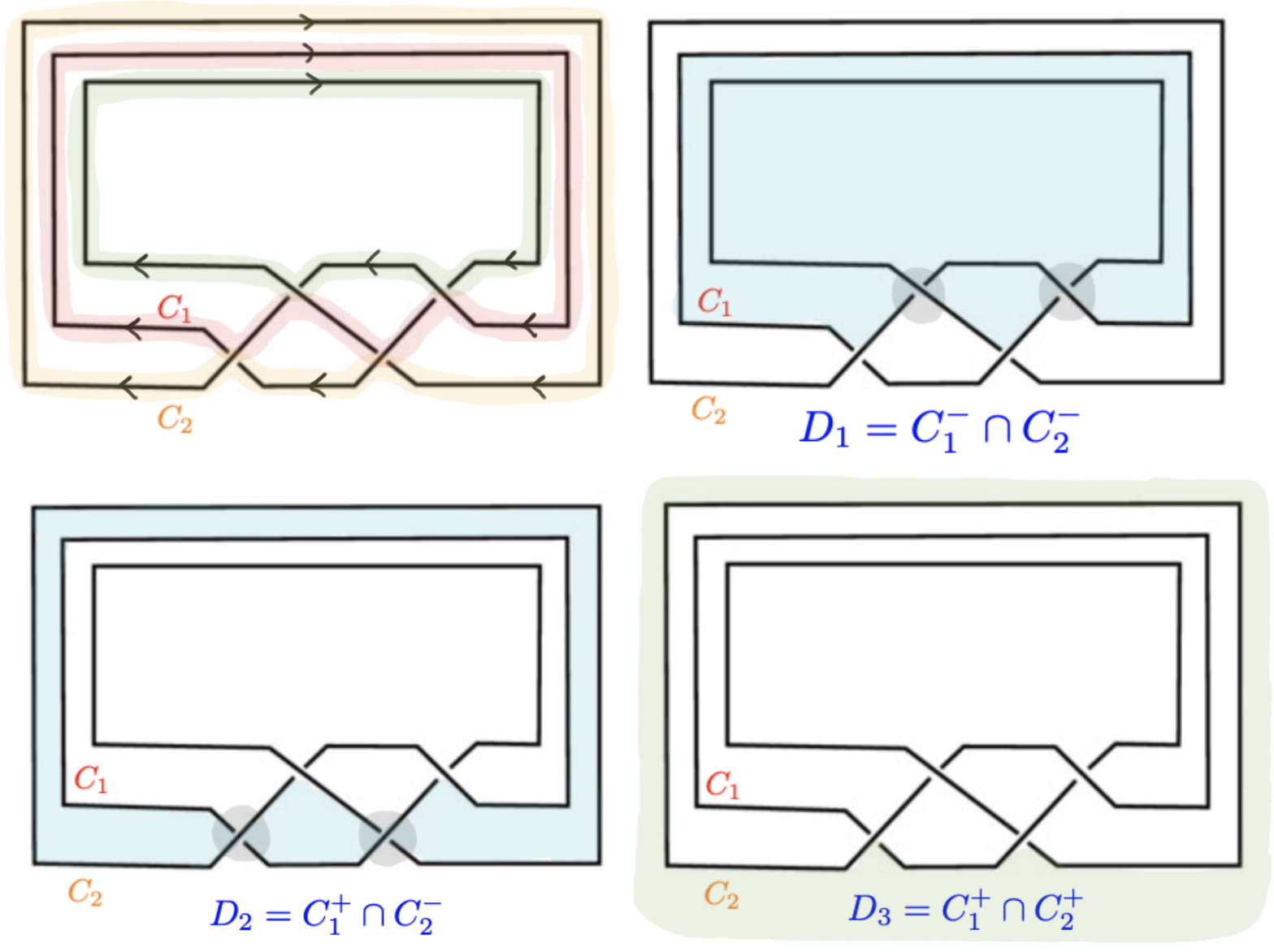}
\caption{Seifert domains and their crossings. Note that none of the crossings belong to $D_3$.}\label{Seifertdomains}
\end{figure}

\begin{figure}[h]
\centering
\includegraphics[scale=0.3]{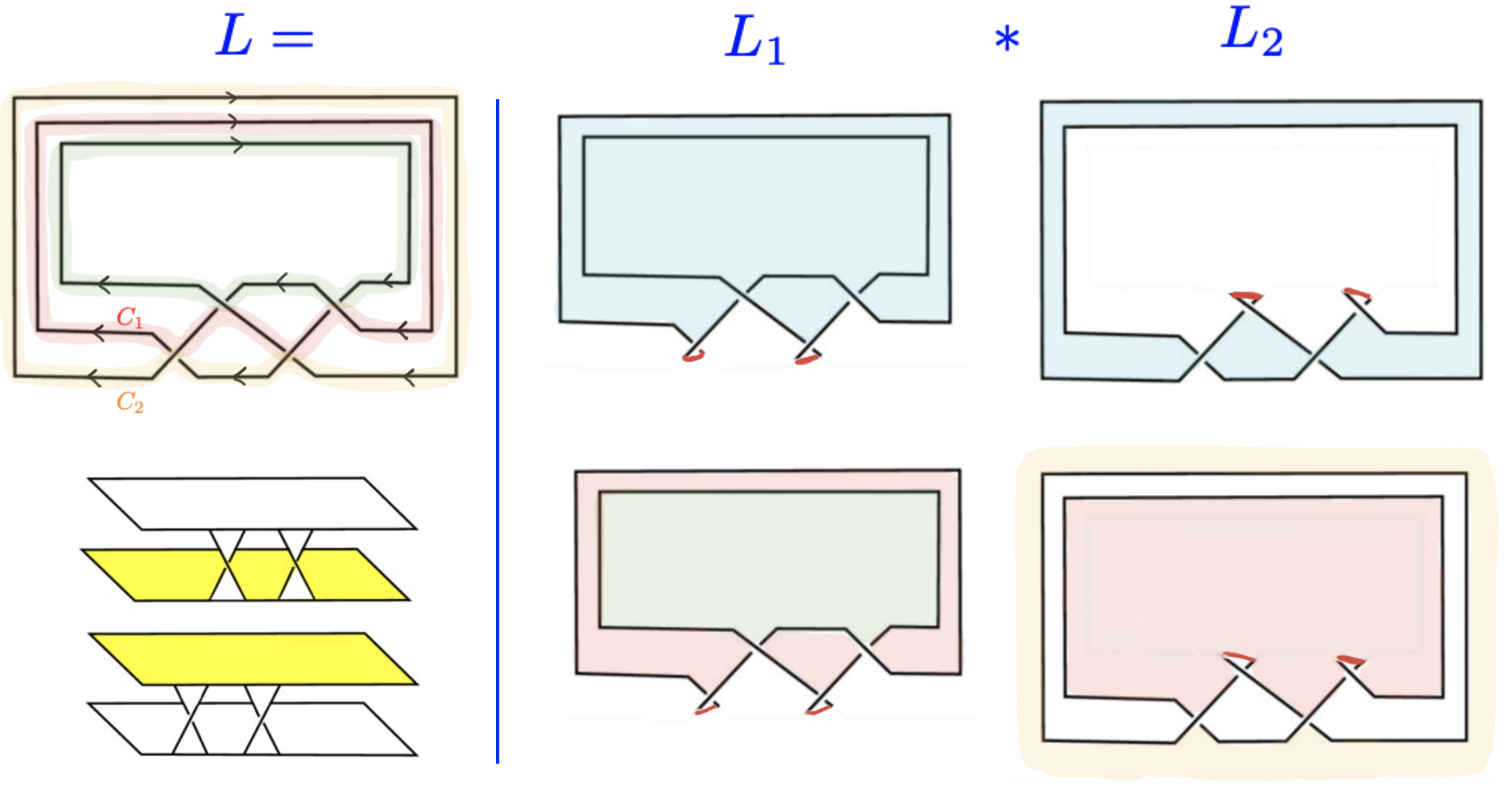}
\caption{An example of a Murasugi decomposition. The first row shows two of the Seifert domains and the associated Murasugi summands. Note that the third summand $L_3$ is trivial and is hence not depicted. The leftmost picture in the second row shows the Murasugi decomposition as a Murasugi sum (related to Figure~\ref{Murasugisumandstar}). Other figures in the second row depict Seifert surfaces of the Murasugi summands constructed by cutting the canonical Seifert surface of $L$.}\label{Murasugidecompose}
\end{figure}

\begin{figure}[h]
\centering
\includegraphics[scale=0.3]{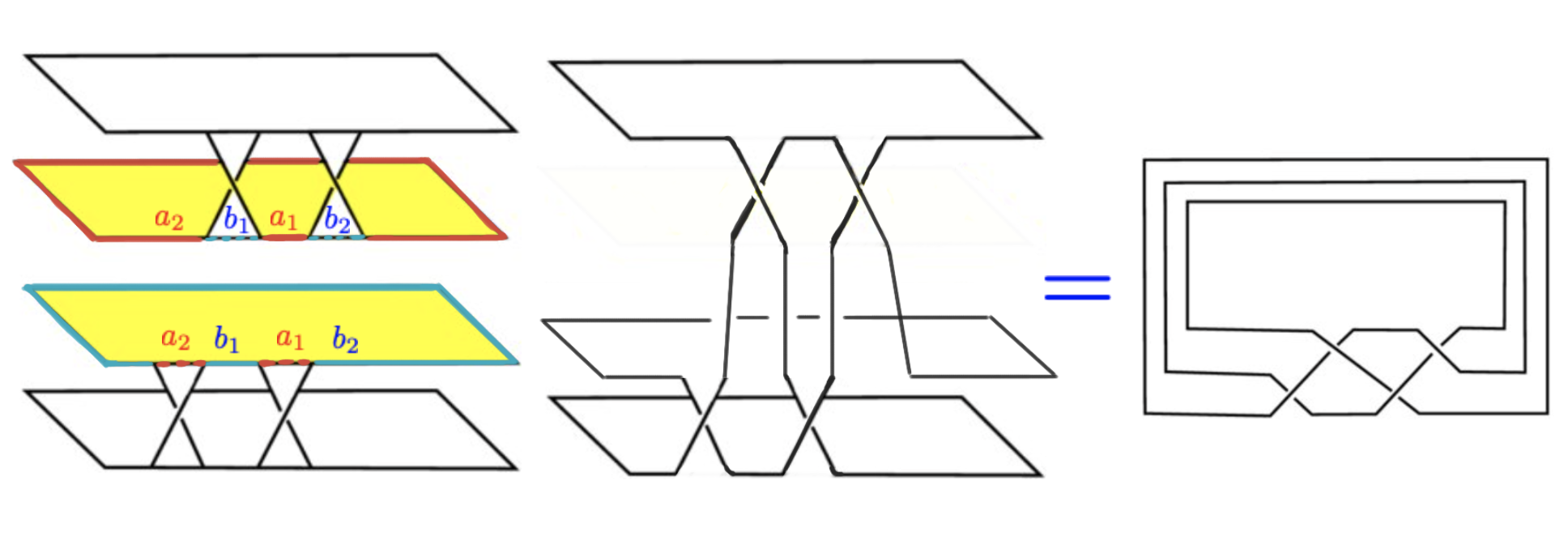}
\caption{The $*$-product is a special case of Murasugi sum (shown in the example of Figure~\ref{Murasugidecompose}).}\label{Murasugisumandstar}
\end{figure}
\begin{figure}[h]
\centering
\includegraphics[scale=0.45]{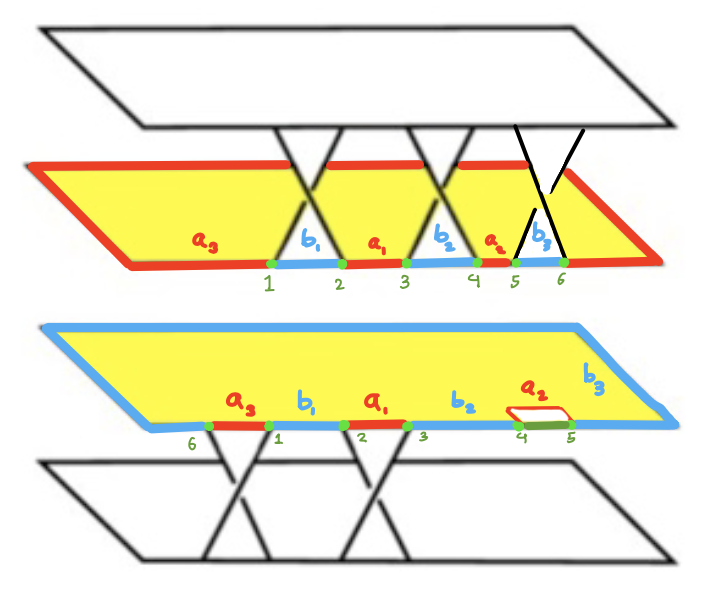}
\includegraphics[scale=0.35]{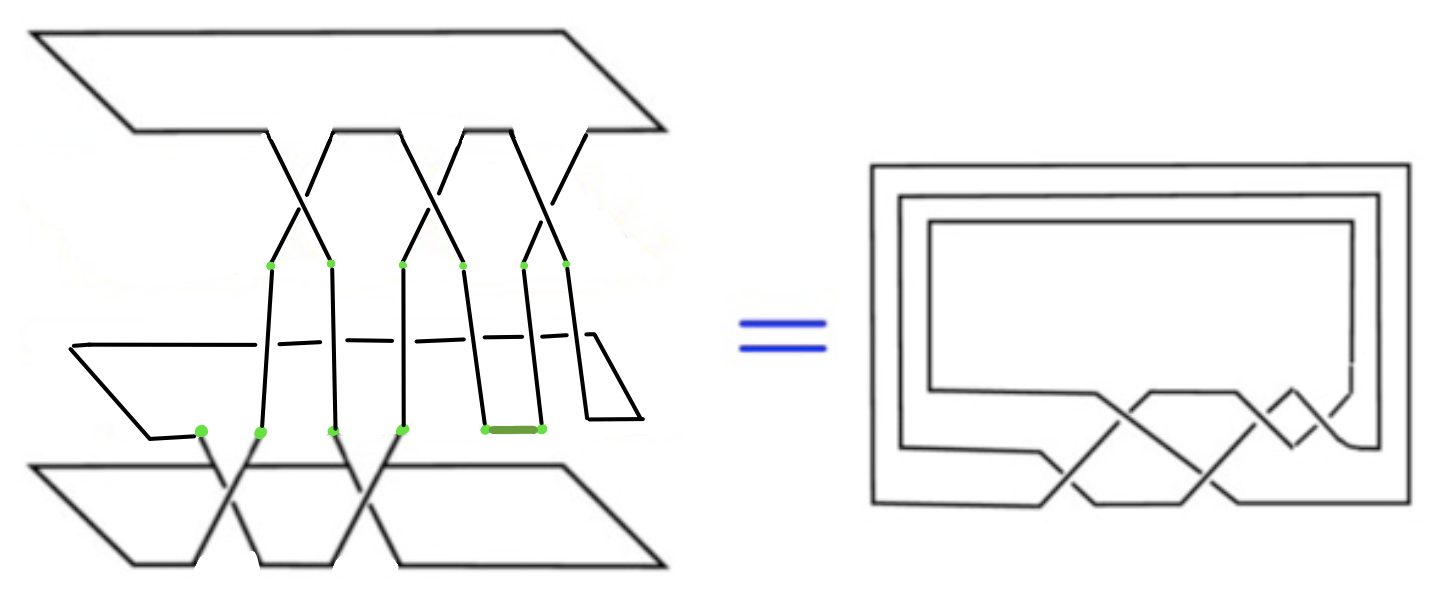}
\caption{A second example showing that $*$-product is a special case of Murasugi sum. Note how the gluing disk of the Murasugi sum is slightly different from the one in Figure~\ref{Murasugisumandstar}}\label{Murasugisumandstar2}
\end{figure}

\begin{rema}\label{diagramgonality}
Each of the $*$-products (i.e., diagrammatic Murasugi sums) in the Murasugi decomposition happens \emph{over a Seifert cycle of type~2}. To be more precise, consider a Seifert cycle $C$ of type~2 which decomposes $S^2$ into disks $C^{+}$ and $C^{-}$. Note that each of the Seifert domains is either in $C^{+}$ or in $C^{-}$. Assume that $D_{1},\dots, D_{i} \subseteq C^{-}$ and $D_{i+1},\dots, D_{m+1} \subseteq C^{+}$. Furthermore, assume $C$ sits between the domains $D_i$ and $D_{i+1}$. Let $L^{+}$ and $L^{-}$ be the result of the restriction of the link $L$ to $C^{+}$ and $C^{-}$, respectively. Then 
\[
L=L^{-} * \ L^{+},
\]
where the Murasugi sum is taken over the disk associated with the Seifert cycle $C$ in the canonical Seifert surface; see the yellow disks in Figures~\ref{Murasugisumandstar} and \ref{Murasugisumandstar2}. This follows from the fact that applying the Murasugi decomposition algorithm to $L^{-}$ and $L^{+}$ gives us
\[
L^{-} = L_1 * \cdots * L_i \  \text{ and } L^{+} = L_{i+1} * \cdots * L_{m+1}.
\]
Let the gonality of the Murasugi sum over the Seifert cycle $C$ be $2s$. We can compute $2s$ from the diagram as follows. Crossings on $C$ are of two types: ones that belong to $D_i$ and ones that belong to $D_{i+1}$ (shown in red and blue in Figure~\ref{diagramPlumbingpatch}). Moving along $C$, one alternately encounters $x_j$ crossings belonging to $D_i$, then $y_j$ crossings belonging to $D_{i+1}$ for $j = 1,\dots,l$, for some $l$. Using this notation, we have
\begin{equation}\label{gonalityequation}
s=\sum_{j=1}^{l} x_j + \sum_{j=1}^{l} y_j - l. 
\end{equation}
We call $l$ the \textit{length} of the diagrammatic Murasugi sum.

\begin{figure}[h]
\centering
\includegraphics[scale=0.3]{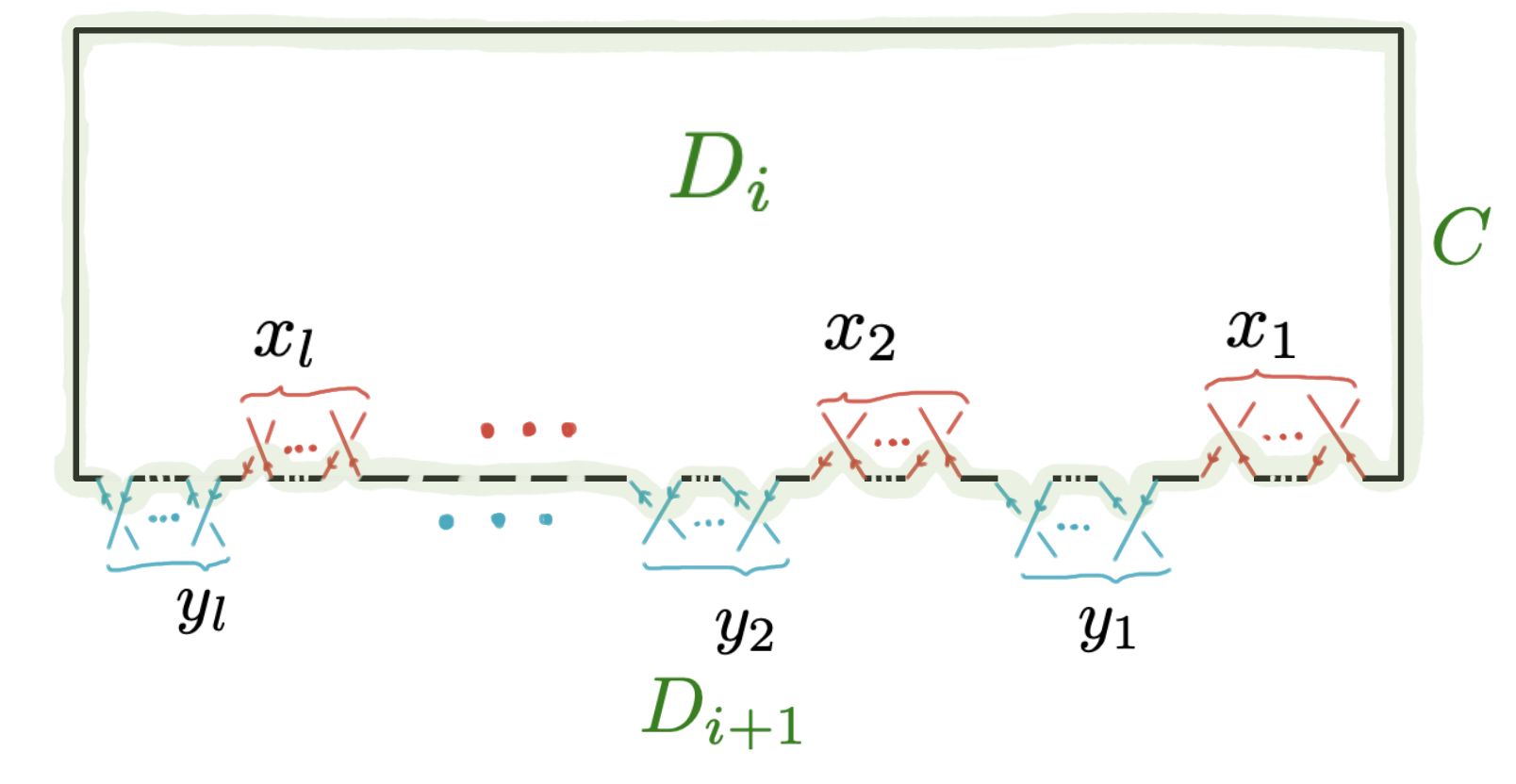}
\caption{Local picture of a type~2 Seifert cycle $C$.}\label{diagramPlumbingpatch}
\end{figure}

\end{rema}

Murasugi's decomposition has an important property. If $D_L$ is a reduced alternating diagram, then all summands $L_i$ will be alternating as well. This leads to the first part of the following result. 

\begin{prop}\cite{Murasugisums}\label{Murasugidecomp}
Any alternating link can be written as a $*$-product (i.e., diagrammatic Murasugi sum) of special alternating links. Furthermore, signature and genus are additive under this decomposition.
\end{prop}

Note that, applying Seifert's algorithm to an alternating diagram leads to a minimal genus Seifert surface, and the diagrammatic Murasugi sum is taken using such surfaces. Based on Theorem~\ref{Gabai}, genus is additive under Murasugi decomposition of an alternating link. Murasugi proved that signature is also additive under such a decomposition. We will revisit these facts in Subsection~\ref{dual}.


\begin{defi}
We say that a link diagram is \emph{decorated} if it has a marked edge, or, equivalently, two marked adjacent regions.
Fix a checkerboard coloring of the regions of a decorated link diagram. The black (resp.\ white) \emph{Tait graph} has vertices the black (resp.\ white) regions, and we add an edge between two regions through each crossing they share. We use the notation $B$ (resp.\ $W$) for this graph. The graph $B$ (resp.\ $W$) is rooted at the vertex corresponding to the marked region; see Figure~\ref{Taitgraphs}. Note that $B$ and $W$ are planar duals. 
    
\begin{figure}[h]
\centering
\includegraphics[scale=0.2]{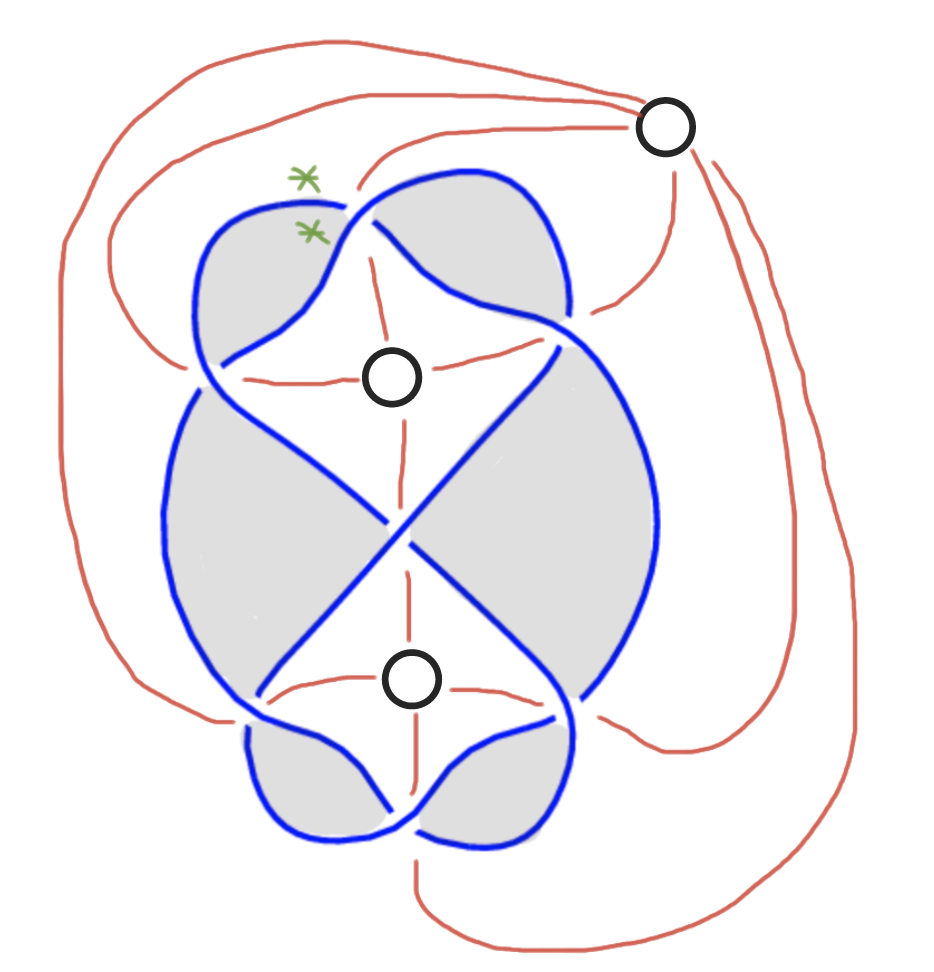}
\includegraphics[scale=0.2]{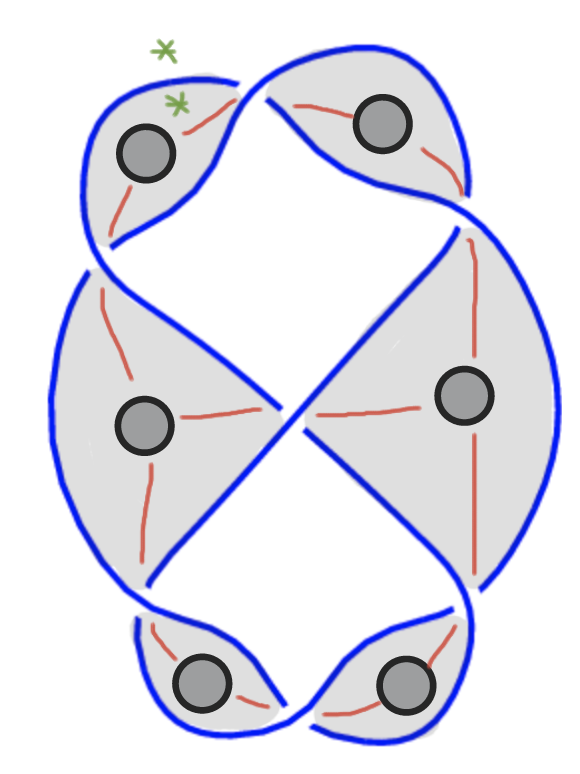}
\caption{The white (left) and black (right) Tait graphs in a marked link diagram.}\label{Taitgraphs}
\end{figure}

\end{defi}

Let $\mathcal{T}_B$ (resp.\ $\mathcal{T}_W$) be the set of rooted spanning trees of $B$ (resp.\ $W$), oriented away from the root. 
There is a natural one-to-one correspondence between $\mathcal{T}_B$ and $\mathcal{T}_W$ coming from planar duality. Let $\mathcal{D}_{(B,W)}$ denote the set of all dual pairs in $\mathcal{T}_B \times \mathcal{T}_W$; i.e.,
\[
\mathcal{D}_{(B,W)} := \{\, (T, T') \in \mathcal{T}_B \times \mathcal{T}_W  \,:\,  T \text{, } T'  \text{ are planar duals}\,\}.
\]
We call each dual pair a \emph{Kauffman state}, and write $\mathcal{D}_{(B,W)}$ for the set of all Kauffman states.

Kauffman's state-sum formula computes the Alexander polynomial as a weighted sum of Kauffman states. The weights are assigned as follows. A directed edge of either of the Tait graphs passes through a crossing $c$. It is directed towards one of the four corners around the crossing $c$. Figure~\ref{KKauffmanweights} shows the weights assigned to each corner. The weight of a rooted tree is defined as the product of the weights assigned to its edges.

\begin{figure}[h]
\centering
\includegraphics[scale=0.2]{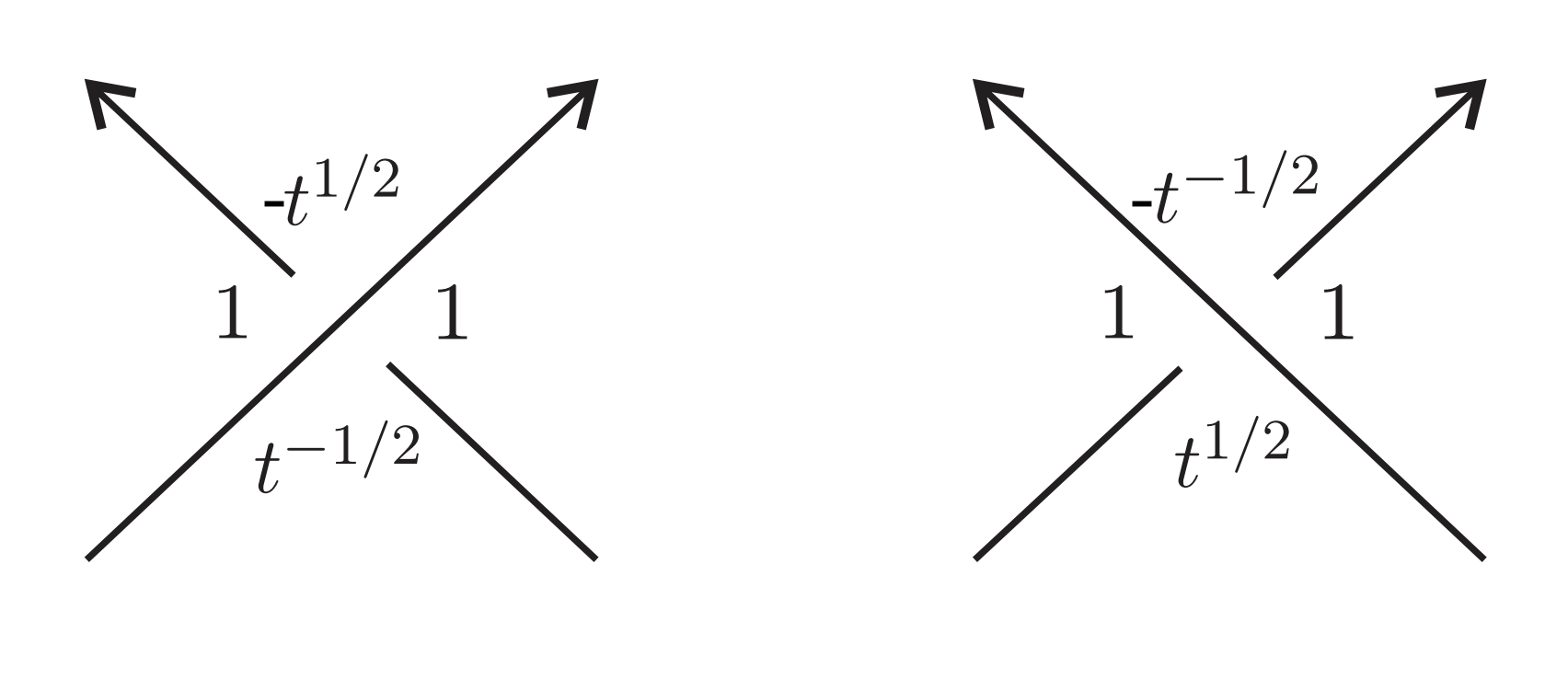}
\caption{Kauffman weights.}\label{KKauffmanweights}
\end{figure}

Kauffman's state-sum formula for the Alexander polynomial is
\begin{equation}
\tilde{\Delta}_{K}(t)=\sum\limits_{(T,T') \in \mathcal{D}_{(B,W)}} \ w(T) \cdot w(T').
\end{equation}
We are now ready to prove the main theorem of this subsection.

\begin{theo}\label{plumbingofspecialthm}
The trapezoidal conjecture holds for any link $L = L_1 * L_2$ which is a diagrammatic plumbing (4-gon Murasugi sum) of two special alternating links $L_1$ and $L_2$.
\end{theo}

\begin{proof}
First, we give a diagrammatic characterization of the plumbing. Choose a reduced alternating diagram $D$ of $L$. Since we only have two summands, there is only one type~2 Seifert cycle $C$, and the Seifert domains are the two disks it bounds in $S^2$. Using the notation of equation~\eqref{gonalityequation}, assume first that $l=1$. If $y_1 = 1$, the crossing connecting $C$ to its exterior is reducible (also called nugatory), contradicting the assumption that $D$ is reduced. Hence $y_1 \geq 2$. Similarly, $x_1 \geq 2$. By equation~\eqref{gonalityequation}, the gonality of the Murasugi sum is at least six, which contradicts the fact that we have a 4-gon Murasugi sum. So $l \geq 2$, which, together with equation~\eqref{gonalityequation}, implies that $l=2$ and $x_j=y_j=1$ for $j \in \{1,2\}$. So the diagram $D$ looks like Figure~\ref{diagramplumbing} in a neighbourhood of $C$.
  
Note that Seifert cycles in a special diagram bound regions of the link diagram. As a result, one can describe the diagrammatic plumbing of two special alternating links as follows. Consider the special alternating links $L_1$ and $L_2$, and let $R_i$ be a region of $L_i$ for $i \in \{1, 2\}$ such that $R_i$ has degree two as a vertex of the Tait graph of $L_i$ (i.e., it has two crossings along its boundary). Change $R_1$ to the infinity region (i.e., the region containing $\infty \in S^2 = \mathbb{R}^2 \cup \{\infty\}$) in the diagram of $L_1$. Then place the diagram of $L_1$ inside region $R_2$ in the diagram of $L_2$, identifying the boundaries of $R_1$ and $R_2$ according to Figure~\ref{diagramplumbing}. See Figure~\ref{diagramplumbingexample} for an example.

\begin{figure}[h]
\centering
\includegraphics[scale=0.2]{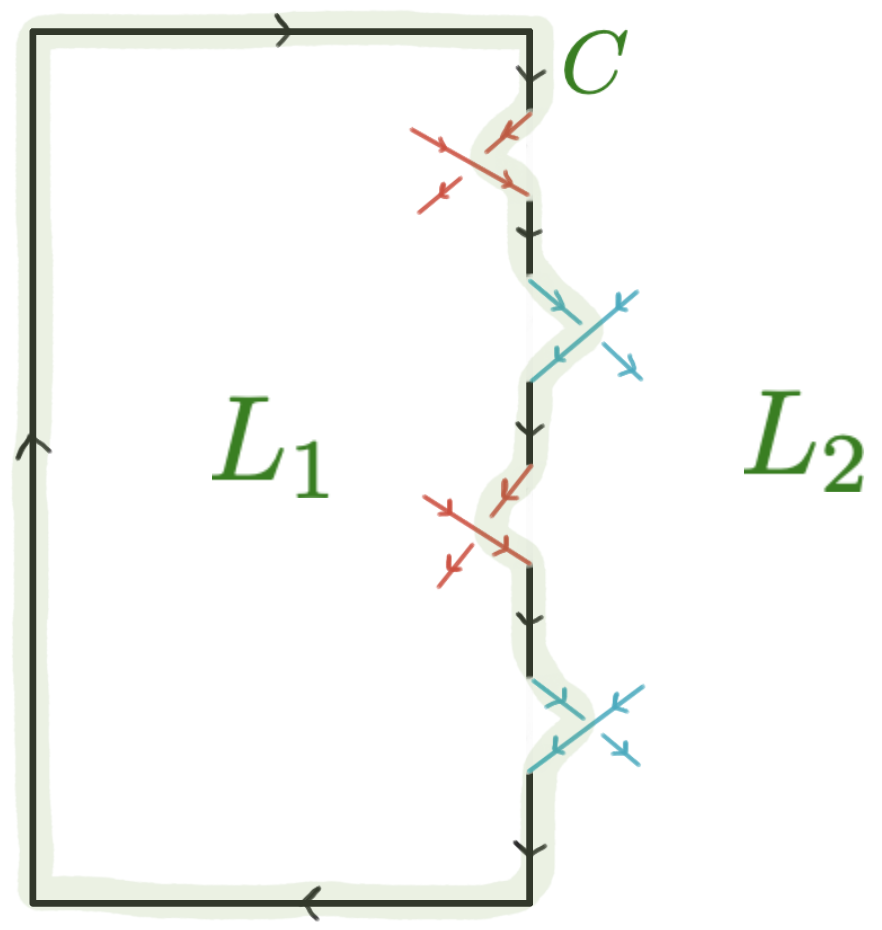}
\caption{Plumbing in a Murasugi decomposition.}\label{diagramplumbing}
\end{figure}

Using this diagrammatic characterization, we can see the effect of the operation on Tait graphs. Let $B_{L_i}$ and $W_{L_i}$ denote the black and white Tait graphs of the special alternating link $L_i$, respectively, for $i \in \{1,2\}$. Without loss of generality, assume that $R_1$ and $R_2$ are both black regions in the respective colorings of the diagram of $L_1$ and $L_2$. We also use $R_1$ and $R_2$ to denote the corresponding vertices of $B_{L_1}$ and $B_{L_2}$, respectively. Furthermore, we can assume that the checkerboard coloring of $L_2$ extends to the coloring of $L$, and let $B_L$ and $W_L$ denote the black and white Tait graphs of $L$. These graphs can be described as follows. See Figures~\ref{Taitblowup} and~\ref{generalplumbedtrees} for an example and some notation.

\begin{figure}[h]
\centering
\includegraphics[scale=0.3]{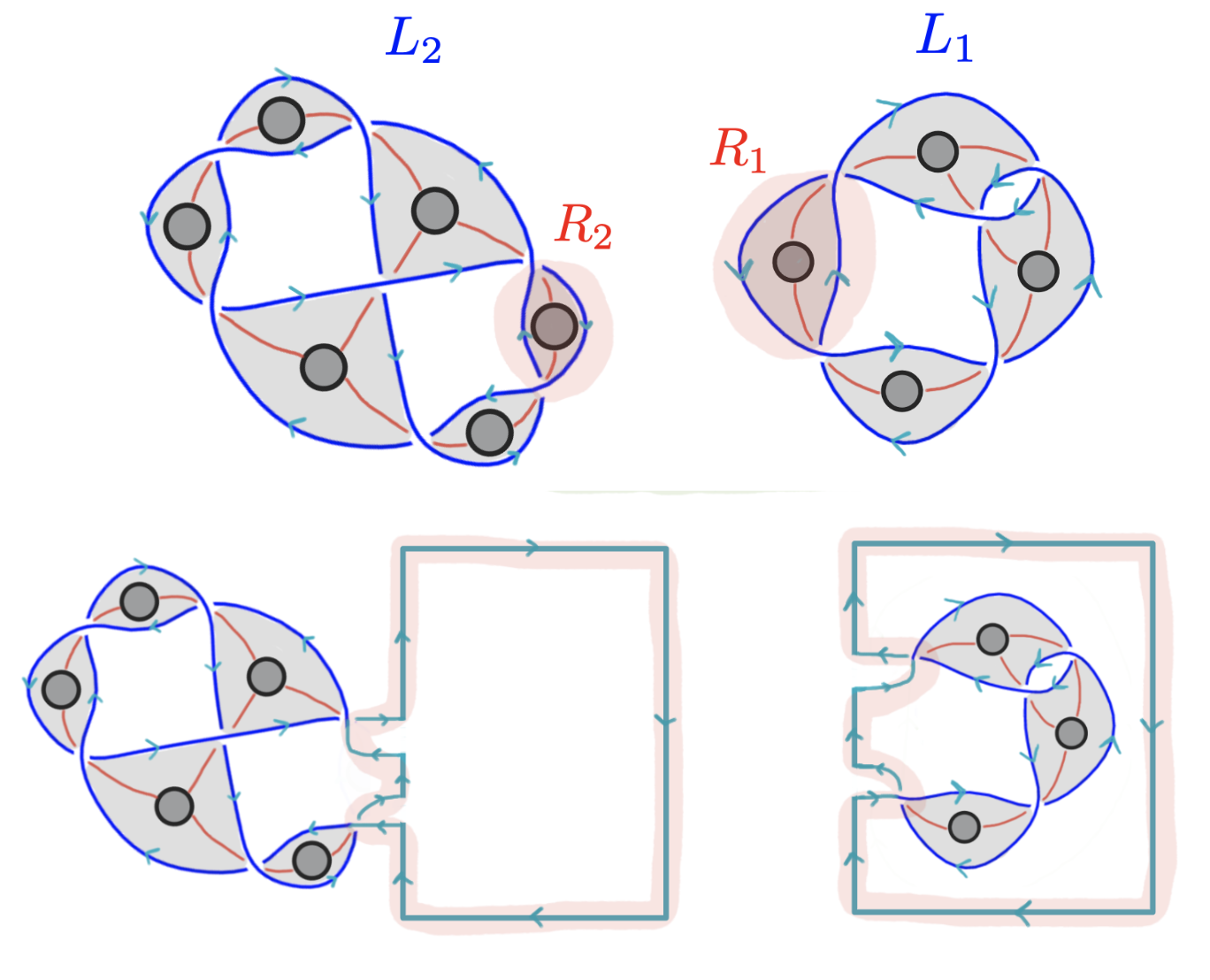}
\includegraphics[scale=0.3]{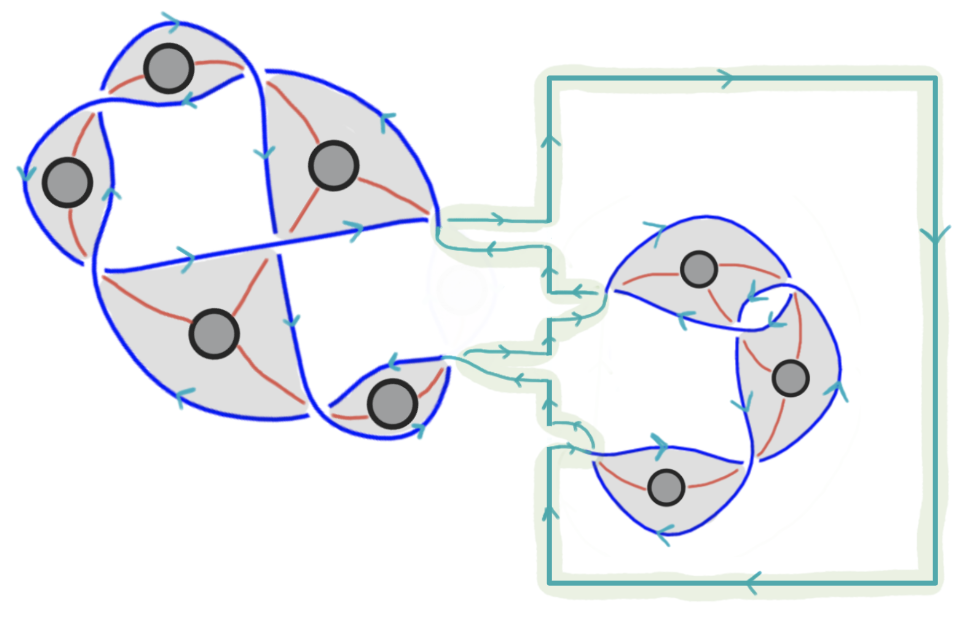}
\caption{Example of diagrammatic plumbing. The links from left to right are $L_2$ and $L_1$. The red highlighted regions are $R_2$ and $R_1$. The green Seifert cycle is $C$, whose neighborhood is shown in Figure~\ref{diagramplumbing}.}\label{diagramplumbingexample}
\end{figure}

The gaph $B_L$ can be constructed from $B_{L_2}$ by replacing the vertex $R_2$ with the graph $W_{L_1}$, as follows. Let $w$ and $w'$ be the neighbors of $R_1$ in $W_{L_1}$, and let $b$ and $b'$ be the neighbors of $R_2$ in $B_{L_2}$. 
We remove $R_1$ from $W_{L_1}$ and $R_2$ from $B_{L_2}$, then add the edges $(b,w)$ and $(b',w')$; see Figure~\ref{generalplumbedtrees}. We call this operation a \emph{Tait blow-up of the region $R_2$}. The graph $W_L$ can be constructed using a similar process, with the roles of the two links reversed.

The red stars in Figure~\ref{generalplumbedtrees} are the roots (i.e., marked vertices). The yellow vertex $y$ is associated with the unbounded region. It also shows a pair of dual spanning trees $(T,T') \in \mathcal{T}_{B_L}\times \mathcal{T}_{W_L}$, colored red and yellow, respectively. Note that the crossings in Figure~\ref{generalplumbedtrees} both belong to $L_2$; cf.\ Figure~\ref{Taitblowup}.

\begin{figure}[h]
\centering
\includegraphics[scale=0.2]{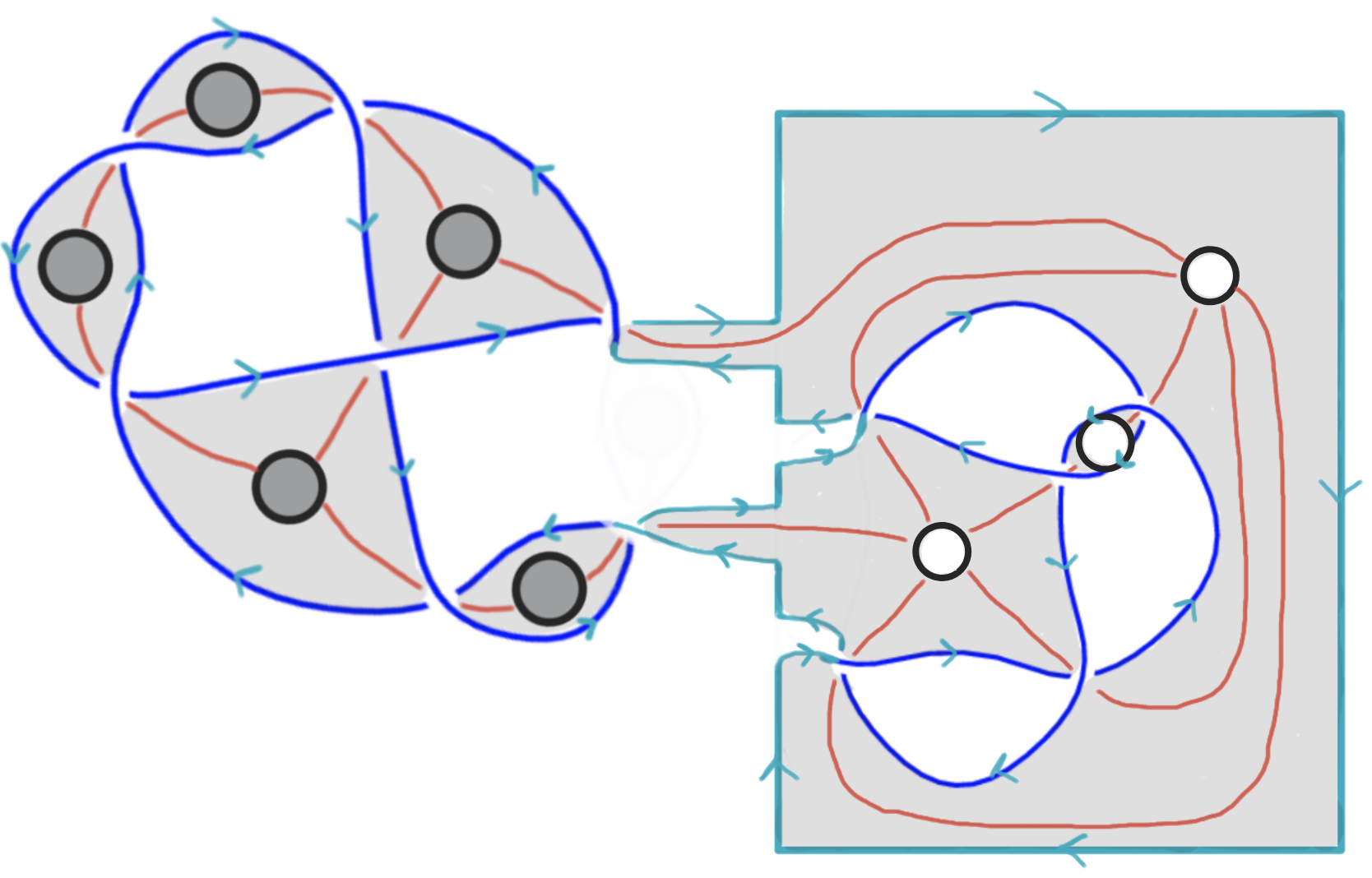}
\caption{Tait blow-up (Tait graph of a plumbing).}\label{Taitblowup}
\end{figure}

\begin{figure}[h]
\centering
\includegraphics[scale=0.3]{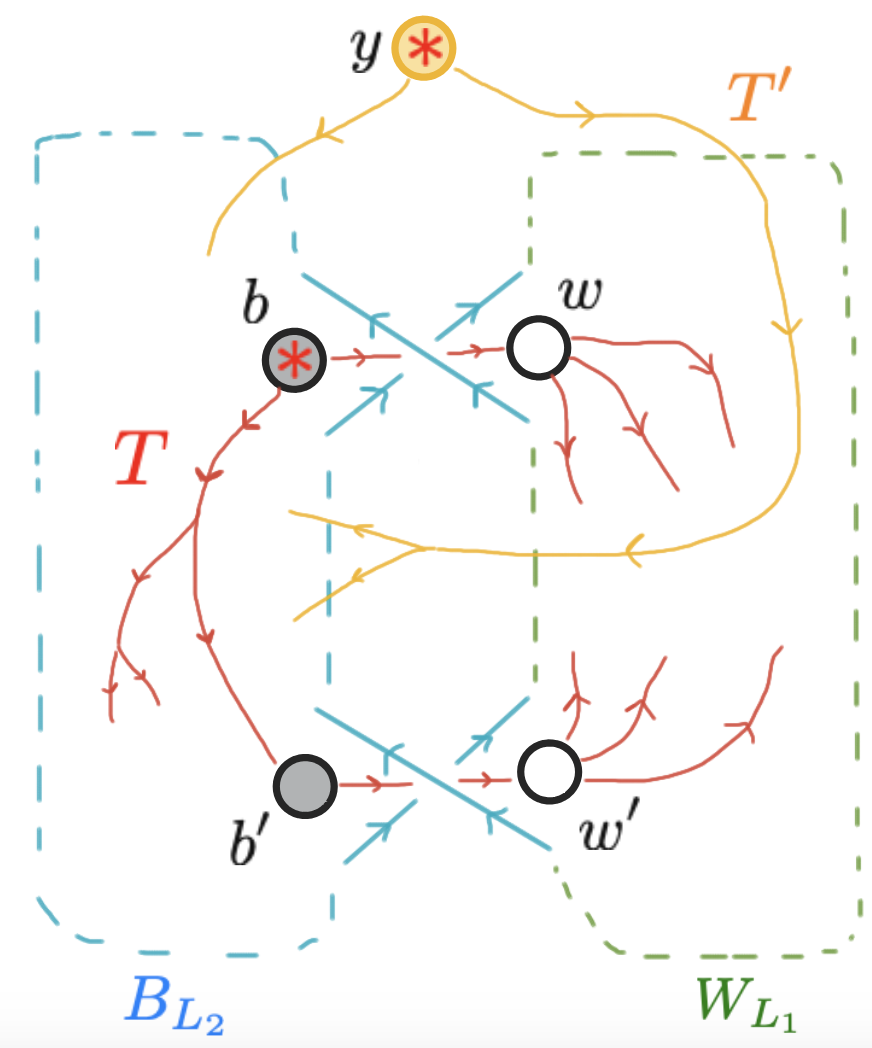}
\caption{A pair of dual trees in $\mathcal{T}_{B_L}\times \mathcal{T}_{W_L}$.}\label{generalplumbedtrees}
\end{figure}

Since $T$ is a spanning tree rooted at $b$, it contains unique directed paths $\beta$ from $b$ to $b'$ and $\omega$ from $w$ to $w'$.
Depending on $\beta$ and $\omega$, there are three distinct possibilities for a pair of dual spanning trees $(T,T') \in \mathcal{T}_{B_L}\times \mathcal{T}_{W_L}$; see Figure~\ref{Trichotomypftrees}:

\noindent\textbf{Case (i):} $\beta$ passes through $w$ and then $w'$.

\noindent\textbf{Case (ii):} $\beta \subseteq B_{L_2}$ and $\omega \subseteq W_{L_1}$. This case has two subcases:

\textbf{Case (ii).1:} $T$ contains the edge $(b,w)$.

\textbf{Case (ii).2:} $T$ contains the edge $(b',w')$.

\noindent\textbf{Case (iii):} $\beta \subseteq B_{L_2}$ and $T$ contains both $(b,w)$ and $(b',w')$. 

For $* \in \{(i), (ii), (ii).1, (ii).2, (iii)\}$, we denote the set of $(T,T')$ in Case~$*$ by $\mathcal{D}^{*}_{(B_L,W_L)}$. By analyzing the possible options for the paths $\beta$ and $\omega$, we obtain the partition
\[
\mathcal{D}_{(B_L,W_L)} = \mathcal{D}^{(i)}_{(B_L,W_L)} \sqcup \mathcal{D}^{(ii)}_{(B_L,W_L)} \sqcup \mathcal{D}^{(iii)}_{(B_L,W_L)}.
\]

\begin{figure}[h]
\centering
\includegraphics[scale=0.35]{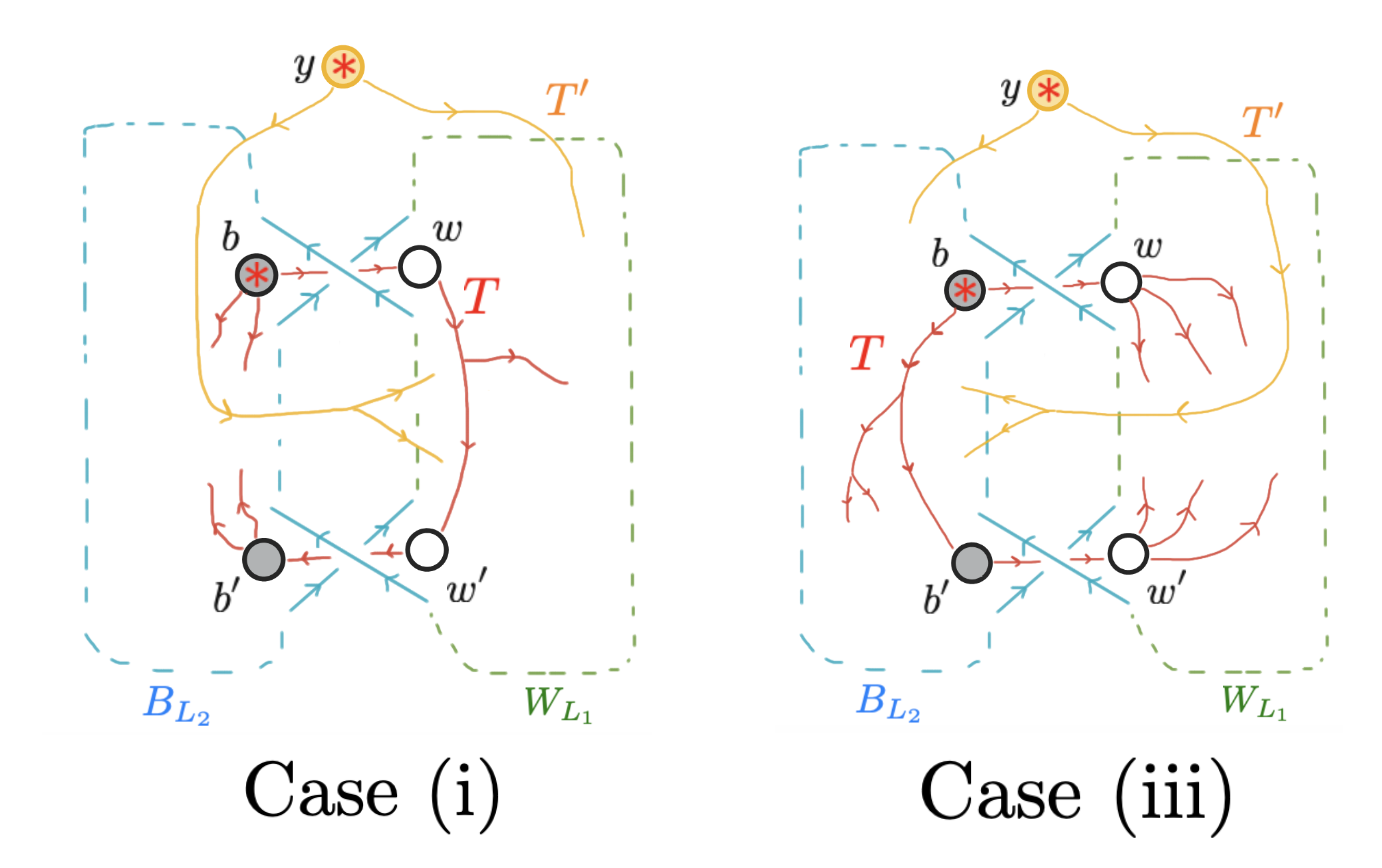}
\includegraphics[scale=0.35]{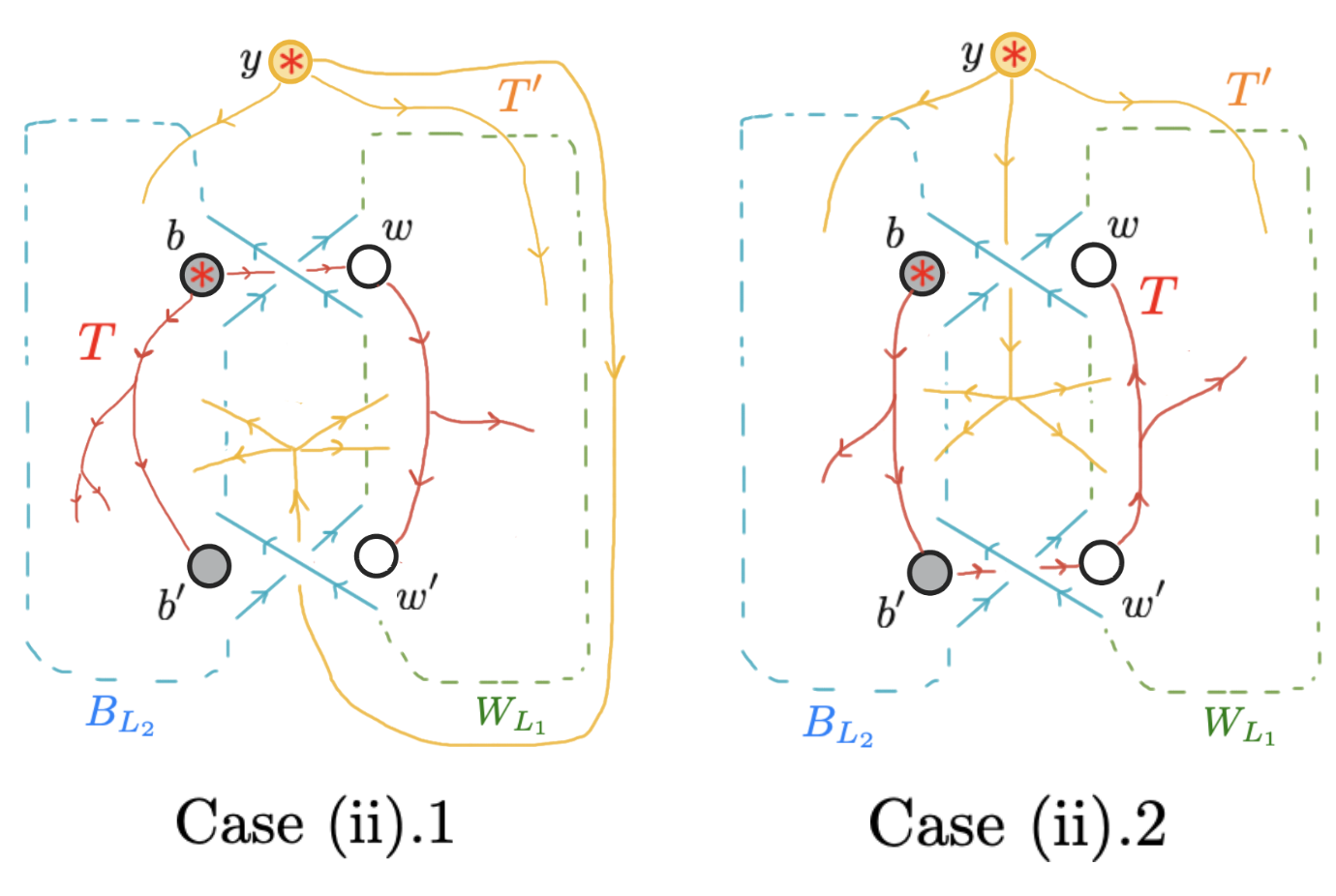}
\caption{The possible cases for the paths $\beta$ from $b$ to $b'$ and $\omega$ from $w$ to $w'$ in a rooted, directed spanning tree $T$ of $B_L$, together with the dual spanning tree $T'$ of $W_L$.}\label{Trichotomypftrees}
\end{figure}

We now look at these cases one-by-one. In each case, we decompose the dual spanning trees $T$ and $T'$ into dual spanning trees of the Tait graphs of $L_1$ and $L_2$, and relate their weights to the weights of Kauffman states in $\mathcal{D}_{(B_{L_1},W_{L_1})}$ and $\mathcal{D}_{(B_{L_2}, W_{L_2})}$.

\begin{figure}[h]
\centering
\includegraphics[scale=0.35]{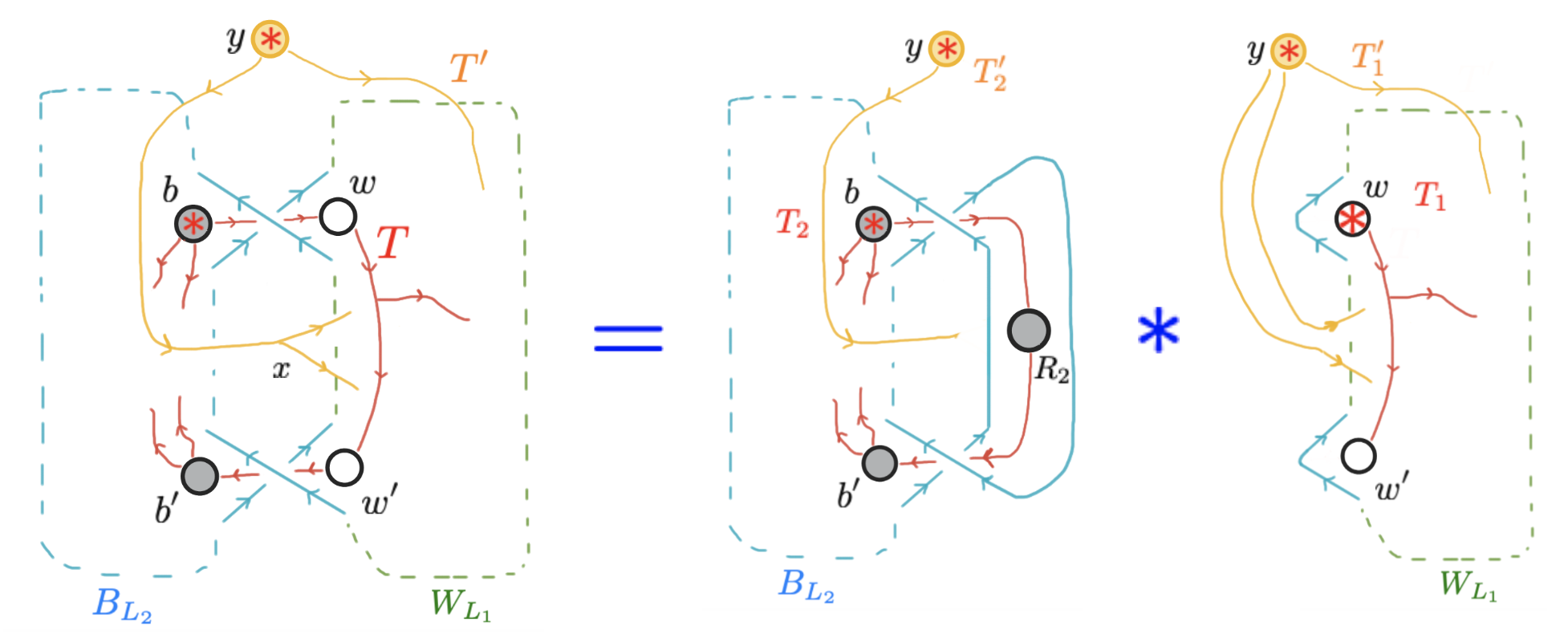}
\caption{Decomposition of spanning trees in Case (i).}\label{Caseianalysis}
\end{figure}

\textbf{Case (i):} The edges of $T \cap W_{L_1}$ form a spanning tree $T_1$ of $W_{L_1}$ rooted at $w$. The rest of the edges form a spanning tree $T_2$ of $B_{L_2}$ after replacing the edges $(b,w)$ and $(w',b')$ with the edges $(b,R_2)$ and $(R_2,b')$. There is an analogous decomposition of the planar dual $T'$ of $T$ into spanning trees $T'_1$ of $B_{L_1}$ and $T'_2$ of $W_{L_2}$, with one small modification. Consider the vertex $x$ of $T'$ that lies between the summands $L_1$ and $L_2$, and between the two crossings between the two summands. The outgoing edges of $x$ become the outgoing edges of $y$ in $T'_1$. See Figure~\ref{Caseianalysis} for an illustration of the decomposition.

This decomposition gives us the map  
\begin{equation}\label{mapiinj}
\begin{split}
\psi \colon \mathcal{D}^{(i)}_{(B_L,W_L)} &\hookrightarrow  \mathcal{D}_{(B_{L_2},W_{L_2})} \times \mathcal{D}_{(W_{L_1},B_{L_1})}.\\
(T,T') &\mapsto \bigl((T_2,T'_2),(T_1,T'_1)\bigr)
\end{split}
\end{equation}
If $T_2$ contains the edges $(b, R_2)$ and $(R_2,b')$, then the construction can be reversed, so $\psi$ is injective. Let 
\[
\mathcal{D}^{(i)}_{(B_{L_2},W_{L_2})} := \{(T_2,T'_2) \in \mathcal{D}_{(B_{L_2},W_{L_2})} \, : \, (b, R_2),(R_2,b') \in T_2\}.
\]
Then
\begin{equation}\label{mapi}
\psi \colon \mathcal{D}^{(i)}_{(B_L,W_L)} \xrightarrow{\sim} \mathcal{D}^{(i)}_{(B_{L_2},W_{L_2})} \times \mathcal{D}_{(W_{L_1},B_{L_1})}  
\end{equation}
is a bijection.
This map just changes the ends of some of the edges and does not change their position with respect to crossings, hence it preserves the weights of states: If $\psi(T,T') = \bigl((T_2,T_2'), (T_1,T_1')\bigr)$, then
\begin{equation}\label{weightpres}
w(T) = w(T_2) \cdot w(T_1) \text{ \,and\, } w(T') = w(T'_2) \cdot w(T'_1).    
\end{equation}

\textbf{Case (ii):} We decompose dual pairs of spanning trees in $\mathcal{D}^{(ii).1}_{B_L, W_L}$ as in Figures~\ref{Caseii1analysis} and those in $\mathcal{D}^{(ii).2}_{B_L, W_L}$ as in Figure~\ref{Caseii2analysis}. Similarly to Case (i), we define the following sets of dual pairs of spanning trees: 
\[
\mathcal{D}^{(ii).1}_{(B_{L_2},W_{L_2})} := \{\,(T_2,T'_2) \in \mathcal{D}_{(B_{L_2},W_{L_2})} \, : \, (b, R_2) \in T_2\ , \ (R_2,b') \not\in T_2 \,\},
\]
\[
\mathcal{D}^{(ii).2}_{(B_{L_2},W_{L_2})} := \{\,(T_2,T'_2) \in \mathcal{D}_{(B_{L_2},W_{L_2})} \, : \, (R_2, b') \in T_2\ , \ (b,R_2) \not\in T_2 \,\}.
\]
Decomposing dual pairs of spanning trees as
\[
(T,T') \mapsto \bigl((T_2,T'_2),(T_1,T'_1)\bigr)
\]
give bijections
\begin{equation}\label{mapii.1}
\mathcal{D}^{(ii).1}_{(B_L,W_L)} \xrightarrow{\sim} \mathcal{D}^{(ii).1}_{(B_{L_2},W_{L_2})} \times \mathcal{D}_{(W_{L_1},B_{L_1})},
\end{equation}
\begin{equation}\label{mapii.2}
\mathcal{D}^{(ii).2}_{(B_L,W_L)} \xrightarrow{\sim} \mathcal{D}^{(ii).2}_{(B_{L_2},W_{L_2})} \times \mathcal{D}_{(W_{L_1},B_{L_1})}.
\end{equation}
These maps preserve the weights of the Kauffman states in the sense of equation~\eqref{weightpres}.

A spanning tree of $B_{L_2}$ rooted at $b$ contains unique directed paths from $b$ to $b'$ and from $b$ to $R_2$. There are three possibilities for these paths, giving the partition
\[
\mathcal{D}_{(B_{L_2},W_{L_2})} = \mathcal{D}^{(i)}_{(B_{L_2},W_{L_2})} \sqcup \mathcal{D}^{(ii).1}_{(B_{L_2},W_{L_2})} \sqcup \mathcal{D}^{(ii).2}_{(B_{L_2},W_{L_2})}.
\]
We can now compute the weighted sum of all the dual pairs in Cases~(i) and~(ii). Using the weight-preserving maps in equations~\eqref{mapi}, \eqref{mapii.1} and~\eqref{mapii.2}, we can write 
\[
\sum\limits_{\mathcal{D}^{(i)}_{(B_{L},W_{L})} \ \sqcup\  \mathcal{D}^{(ii).1}_{(B_{L},W_{L})}\  \sqcup\  \mathcal{D}^{(ii).2}_{(B_{L},W_{L})}} \ w(T) \cdot w(T') = 
\]
\begin{equation*}
\begin{split}
&\left(\sum\limits_{\mathcal{D}^{(i)}_{(B_{L_2},W_{L_2})} \ \sqcup\  \mathcal{D}^{(ii).1}_{(B_{L_2},W_{L_2})}\  \sqcup\  \mathcal{D}^{(ii).2}_{(B_{L_2},W_{L_2})}} \ w(T_2) \cdot w(T'_2)\right) \cdot\\
&\left(\sum\limits_{\mathcal{D}_{(W_{L_1},B_{L_1})}} \ w(T_1) \cdot w(T'_1)\right) =
\end{split}
\end{equation*}
\[
\left(\sum\limits_{\mathcal{D}_{(B_{L_2},W_{L_2})}} \ w(T_2) \cdot w(T'_2)\right) \left(\sum\limits_{\mathcal{D}_{(W_{L_1},B_{L_1})}} \ w(T_1) \cdot w(T'_1)\right)
= \tilde{\Delta}_{L_2} \cdot \tilde{\Delta}_{L_1}.
\]
The last equality comes from the Kauffman state-sum formula. Note that, although the roots of the spanning trees of $W_{L_1}$ is $w$ in Cases (i) and (ii).1 and is $w'$ in Case (ii).2, this does not affect the result of the state-sum.

The product of two symmetric trapezoidal polynomials (i.e., the convolution of their coefficient sequences) with coefficients of alternating signs is also trapezoidal; see \cite[Proposition~2.1]{murasugi_1985}. Since the trapezoidal conjecture for special alternating links by Hafner, Mészáros, and Vidinas~\cite{KarolaLogconcavityOT}, the sequence of coefficients of $\tilde{\Delta}_{L_1} \cdot \tilde{\Delta}_{L_2}$ is trapezoidal.

\begin{figure}[h]
\centering
\includegraphics[scale=0.35]{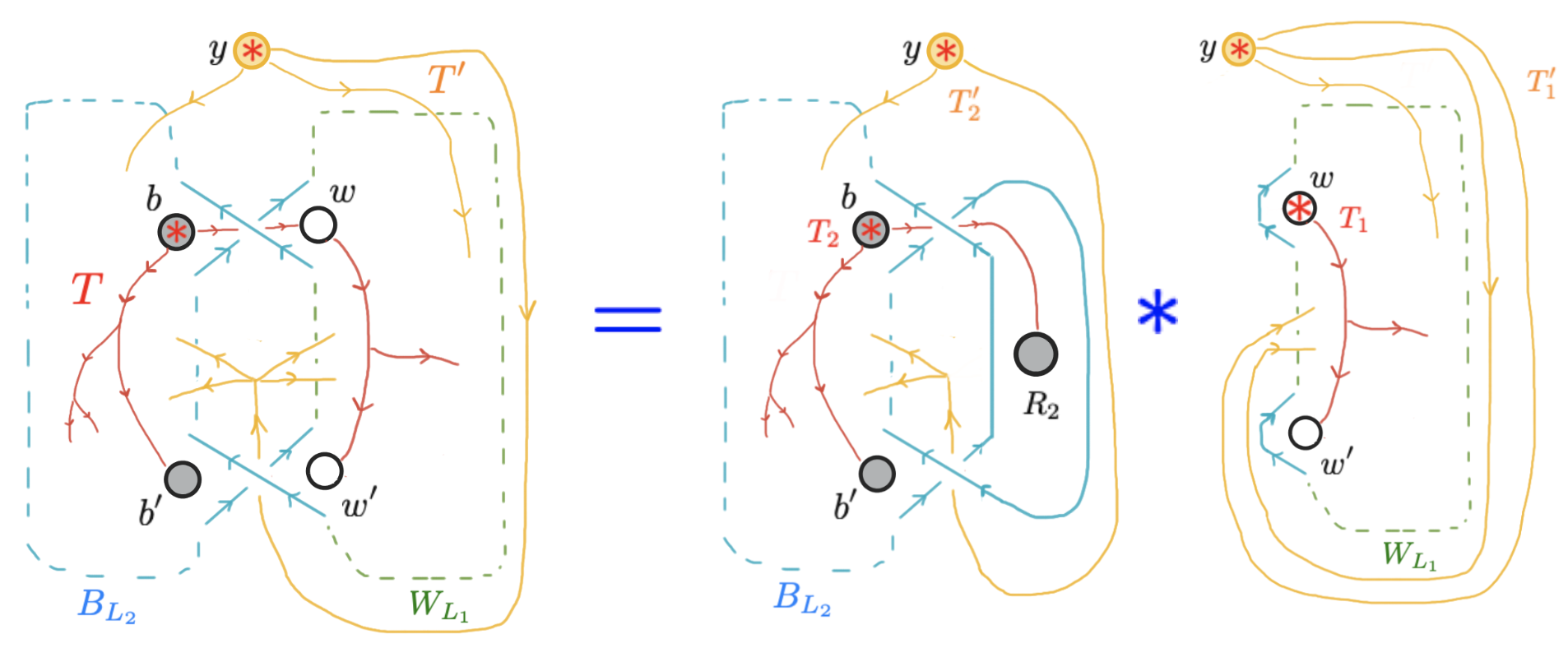}
\caption{Decomposition of spanning trees in Case (ii).1.}\label{Caseii1analysis}
\end{figure}
\begin{figure}[h]
\centering
\includegraphics[scale=0.35]{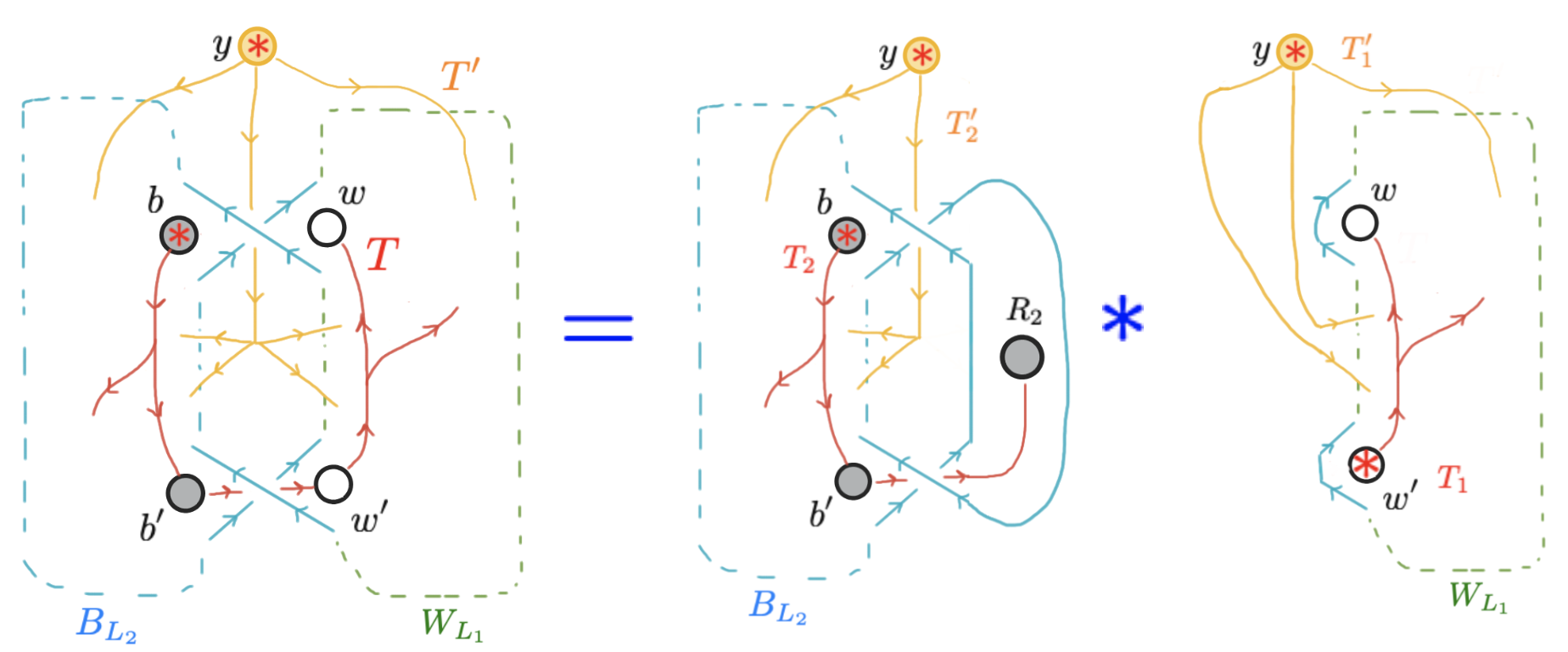}
\caption{Decomposition of spanning trees in Case (ii).2.}\label{Caseii2analysis}
\end{figure}

\textbf{Case (iii):} What remains is to show that the weighted sum of all the spanning trees in Case (iii) is also trapezoidal. Let $\widetilde{L}_2$ be the result of deleting $R_2$ from $L_2$ and smoothing the two crossings on its boundary. Let $W_{\widehat{L}_1}$ be the result of contracting $w$ and $w'$ in $W_{L_1}$ to a single vertex $\{w,w'\}$. The link $\widehat{L}_1$ can be constructed using the median construction from this graph. Note that $B_{\widetilde{L}_2}$ sits inside $B_L$. For a spanning tree $T$ of $B_L$, let $\widetilde{T}_2 := T \cap B_{\widetilde{L}_2}$. Due to the definition of Case (iii), the graph $\widetilde{T}_2$ is a rooted spanning tree of $B_{\widetilde{L}_2}$.  The rest of the edges of $T$ form a spanning forest of $W_{L_1}$ with two components rooted at $w$ and $w'$. By contracting $w$ and $w'$, this forest becomes a spanning tree $\widehat{T}_1$ of $W_{\widehat{L}_1}$ rooted at $\{w,w'\}$. This decomposition also extends to the planar duals $T'$, $\widehat{T}_1'$, and $\widetilde{T}_2'$, with a small modification. The outgoing edges of $x$ turn into the outgoing edges of $y$ in $\widetilde{T}_2'$. See Figure~\ref{Caseiii1analysis}.

\begin{figure}[h]
\centering
\includegraphics[scale=0.35]{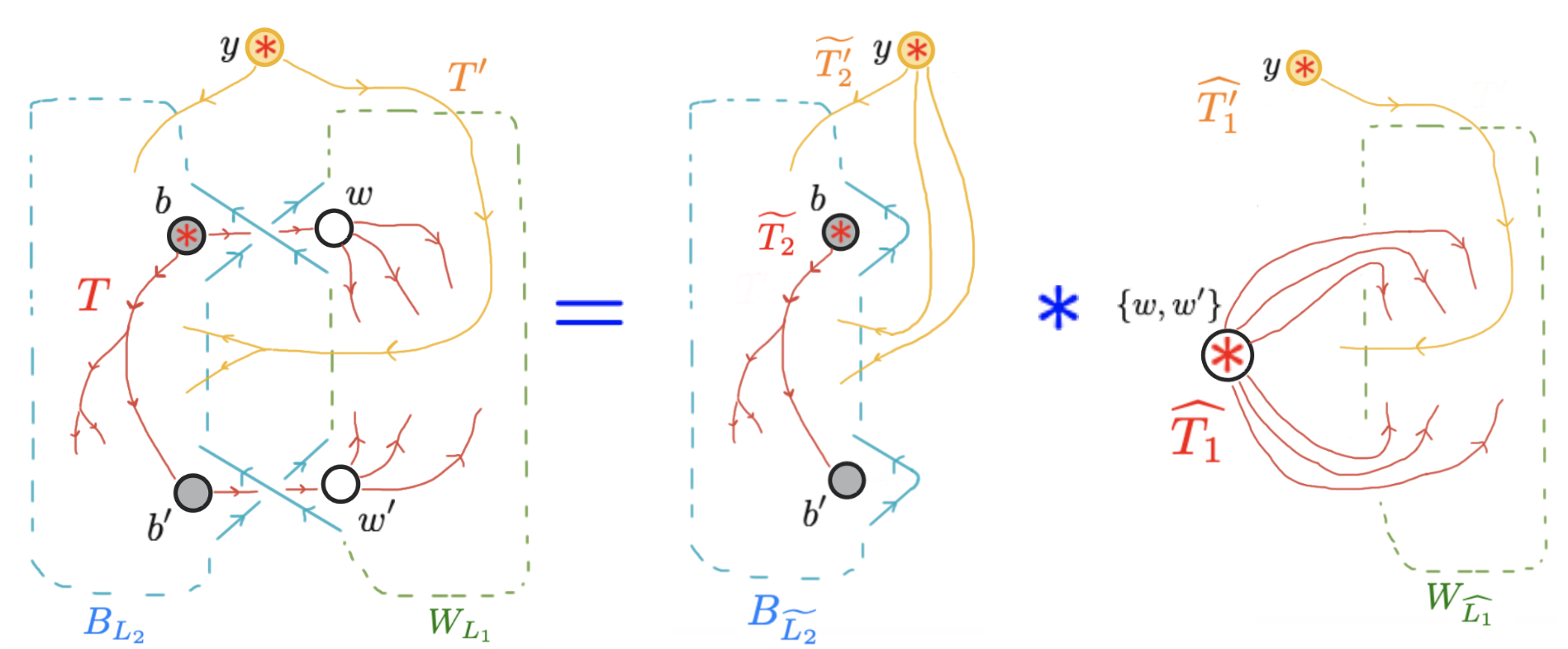}
\caption{Decomposition of spanning trees in Case (iii).}\label{Caseiii1analysis}
\end{figure}

Similarly to the previous cases, this construction is reversible, so gives us a bijection
\begin{equation}
\begin{split}
\mathcal{D}^{(iii)}_{(B_L,W_L)} &\xrightarrow{\sim} \mathcal{D}_{(B_{\widetilde{L}_2},W_{\widetilde{L}_2})} \times \mathcal{D}_{(W_{\widehat{L}_1},B_{\widehat{L}_1})}.\\
(T, T') &\mapsto \bigl((\widetilde{T}_2,\widetilde{T}_2'),(\widehat{T}_1,\widehat{T}_1')\bigr)
\end{split}
\end{equation}
This map is also weight-preserving in the sense of equation~\ref{weightpres}. This part of the argument is a bit different from the previous cases since this decomposition deletes the edges $(b,w)$ and $(b',w')$. Both of these edges have weight~$1$, and hence this does not affect the overall weight of the Kauffman states. This decomposition also deletes two of the edges of $T'$. These are the edges that pass over the other two crossings on the boundary of the type~2 Seifert cycle $C$. One can see that these two edges also have weight $1$.

Using this weight-preservation, we can finally compute the sum of the weights of all the dual pairs of spanning trees in Case (iii) as follows:
\[
\sum_{\mathcal{D}^{(iii)}_{(B_{L},W_{L})}} \ w(T) \cdot w(T') = 
\]
\[
\left(\sum\limits_{\mathcal{D}_{\bigl(B_{\widetilde{L}_2},W_{\widetilde{L}_2}\bigr)}} \ w \bigl(\widetilde{T}_2 \bigr) \cdot w \bigl(\widetilde{T}_2' \bigr)\right) \left(\sum_{\mathcal{D}_{\bigl(W_{\widehat{L}_1},B_{\widehat{L_1}}\bigr)}} \ w \bigl(\widehat{T}_1 \bigr) \cdot w \bigl(\widehat{T}_1' \bigr)\right) = \tilde{\Delta}_{\widetilde{L}_2} \cdot \tilde{\Delta}_{\widehat{L}_1}.
\]

This polynomial is also trapezoidal as both $\widetilde{L}_2$ and $\widehat{L}_1$ have bipartite Tait graphs and hence are special alternating links. This is clear for $B_{\widetilde{L}_2} = B_{L_2} \setminus \{R_2\}$. The graph $B_{\widehat{L}_1}$ is also bipartite as its planar dual $W_{\widehat{L}_1}$ comes from contracting two non-adjacent vertices in $W_{L_1}$. Since $B_{L_1}$ is bipartite, degrees of all vertices in $W_{L_1}$ are even and as a result all of the degrees in $W_{\widehat{L}_1}$ are also even.

Let $\widetilde{L}_1$ be the result of deleting $R_1$ from $L_1$ and smoothing the two crossings on its boundary. In $W_{\widetilde{L}_1}$, the vertices $w$ and $w'$ (corresponding to neighboring regions of $R_1$) are contracted to a single vertex $\{w,w'\}$. One can see that $W_{\widetilde{L}_1}$ and $W_{\widehat{L}_1}$ only differ in some loop edges connected to the vertex $\{w,w'\}$. Loop edges of the Tait graph pass over reducible crossings and do not change the isotopy class of the link. Hence, the links $\widetilde{L}_1$ and $\widehat{L}_1$ are equivalent, so $\tilde{\Delta}_{\widetilde{L}_1}=\tilde{\Delta}_{\widehat{L}_1}$.

These computations lead to the following final equation for the Alexander polynomial of the diagrammatic plumbing:
\begin{equation} \label{Alexanderplumbingformula}
    \tilde{\Delta}_{L} = \tilde{\Delta}_{L_1} \cdot \tilde{\Delta}_{L_2} + \tilde{\Delta}_{\widetilde{L}_1} \cdot \tilde{\Delta}_{\widetilde{L}_2}.
\end{equation}
We only need to compare the degrees of the two trapezoidal summands $\tilde{\Delta}_{L_1} \cdot \tilde{\Delta}_{L_2}$ and $\tilde{\Delta}_{\widetilde{L}_1} \cdot \tilde{\Delta}_{\widetilde{L}_2}$ to prove that $\tilde{\Delta}_L$ is trapezoidal. To be precise, we show that
\begin{equation}\label{degreedif}
 \text{deg} \bigl(\tilde{\Delta}_{L_1} \cdot \tilde{\Delta}_{L_2} \bigr) = \text{deg} \bigl(\tilde{\Delta}_{\widetilde{L}_1} \cdot \tilde{\Delta}_{\widetilde{L}_2} \bigr) + 1;
\end{equation}
see \cite[Proposition 2.1]{murasugi_1985} and note that we are working with symmetrized Alexander polynomials.
We determine the degrees of these polynomials by computing the genera of $\widehat{L}_1$ and $\widetilde{L}_2$. Since the median construction on the black Tait graphs yield minimal genus Seifert surfaces, we only need to compute the rank of the first homology of these graphs. For a graph $B$, we have
\[
\rank(H_1(B))=\#(\text{edges})-\#(\text{vertices})+\#(\text{components}).
\]
This formula can be used to deduce that 
\[
\rank \bigl(H_1 \bigl(B_{\widetilde{L}_2} \bigr) \bigr) = \rank \bigl(H_1 \bigl(B_{L_2} \bigr) \bigr)-1. 
\]
Deleting $R_2$ does not increase the number of components of the black Tait graph since the diagram is reduced and the degree of $R_2$ is two. Similarly, we have 
\[
\rank \bigl(H_1 \bigl(B_{\widetilde{L}_1} \bigr) \bigr) = \rank \bigl(H_1 \bigl(B_{L_1} \bigr) \bigr)-1.
\]
This proves equation~\eqref{degreedif} and concludes the proof of Theorem~\ref{plumbingofspecialthm}.
\end{proof}

We end Subsection~\ref{plumbingofspecial} by mentioning some corollaries of Theorem~\ref{plumbingofspecialthm} and its proof. The formula derived for the Alexander polynomial of diagrammatic plumbings might be of independent interest:

\begin{prop}\label{Alexanderformulaplumbing}
    Let the link $L = L_1 * L_2$ be the diagrammatic plumbing of the links $L_1$ and $L_2$. Construct $\widetilde{L}_1$ and $\widetilde{L}_2$ as in the proof of Theorem~\ref{plumbingofspecial}; i.e., by deleting the regions involved in the plumbing of $L_1$ and $L_2$. Then
    \[
    \tilde{\Delta}_{L}=\tilde{\Delta}_{L_1} \cdot \tilde{\Delta}_{L_2} + \tilde{\Delta}_{\widetilde{L}_1} \cdot \tilde{\Delta}_{\widetilde{L}_2}.
    \]
\end{prop}

As mentioned before, one can use an inductive argument to generalize Theorem~\ref{plumbingofspecialthm} to the diagrammatic plumbing of several special alternating links.

\begin{theo}
    The trapezoidal conjecture holds for the diagrammatic plumbing $L_1 * \cdots * L_n$ of special alternating links $L_1, \dots, L_n$.
\end{theo}

\begin{proof}
 The proof proceeds by considering the innermost Seifert cycle $C_m$ of type~2. The interior disk $C_m^{+}$ bounded by $C_m$ is also a Seifert domain and contains a summand of the diagrammatic plumbing. Without loss of generality, assume this summand is $L_n$ and define $L'=L_1*\cdots*L_{n-1}$.
 
 Now note that a diagrammatic Murasugi decomposition always happens over a tree $T$. This is due to the inclusion hierarchy of type~2 Seifert cycles, which induces a partial order on the summands. In this Murasugi sum tree, $L_n$ corresponds to a leaf vertex $v$ in $T$. Let $L_i$ be the summand corresponding to the neighbor of $v$ in $T$.
 
 We use an induction on the number of summands in the diagrammatic plumbing. Applying Proposition~\ref{Alexanderformulaplumbing} to $L = L_n * L'$ gives us 
    \[
    \tilde{\Delta}_{L}=\tilde{\Delta}_{L_n} \cdot \tilde{\Delta}_{L'} + \tilde{\Delta}_{\widetilde{L}_n} \cdot \tilde{\Delta}_{\widetilde{L}'}.
    \]
 The trapezoidal conjecture holds for $L_n$ and $\widetilde{L}_n$ which are special alternating. It also holds for $L'$ due to the induction hypothesis. Furthermore, one can show that $$\widetilde{L}' = L_1 * \cdots *\widetilde{L}_i * \cdots * L_{n-1},$$
 and hence the trapezoidal conjecture holds for $\widetilde{L}'$ as well. Thic concludes the proof. 
\end{proof}

The proof of Theorem~\ref{plumbingofspecialthm} does not need the Murasugi sums to be plumbings and a weaker condition also suffices. One can modify the proof to work in the case when all diagrammatic Murasugi sums are of length one or two. This leads to Theorem~\ref{Trapezoidalsumovertree}.

\begin{theo}\label{Trapezoidalsumovertree}
The trapezoidal conjecture holds for the diagrammatic Murasugi sum $L_1 * \cdots * L_n$ of special alternating links $L_1, \dots, L_n$ such that the length of all Murasugi sums are less than three. 
\end{theo}

\begin{rema}
The condition that the length of a Murasugi sum is less than three means that we allow $x_1$, $x_2$, $y_1$, $y_2$ to be any positive integers, but there are no $x_i$, $y_i$ for $i \geq 3$, using the notation of Figure~\ref{diagramPlumbingpatch}. The most important change is in the definition of $\widetilde{L}_1$ and $\widetilde{L}_2$ in the formula of Proposition~\ref{Alexanderformulaplumbing}. In this general setting, one must construct $\widetilde{L}_1$ (resp.\ $\widetilde{L}_2$) by gluing an unknotted band to $L_1$ (resp.\ $L_2$) inside the region $R_1$ (resp.\ $R_2$) and separating the two groups of twists along its boundary. Gluing these bands changes the Tait graph by a planar contraction of two vertices (or a decomposition of a vertex into two in the dual plane graph), and this is exactly what we need for generalizing the proof of Theorem~\ref{plumbingofspecialthm}. See Figures~\ref{diagMurasugisumoflen2}, \ref{L1tilde}, and~\ref{L2tilde} for some examples.

Note that the gluing construction generalizes the definition in the case of plumbings, as gluing a band to a bigon region results in two kinks that can be removed by Reidmeister~1 moves, hence is equivalent to smoothing the two crossings along the boundary. 
\end{rema}

\begin{rema}
Suppose one tries to use the same method for Murasugi sums of length at least three. In that case, extra terms appear in the decomposition that cannot be easily grouped together such that they sum up to known trapezoidal polynomials. This can also be seen in Alrefai--Chbili~\cite{small3braids}, where the terms appearing in the formula for the Alexander polynomial of 3-braids of length $3$ is much more complicated than 3-braids of length $2$.    
\end{rema}

\begin{figure}[h]
\centering
\includegraphics[scale=0.35]{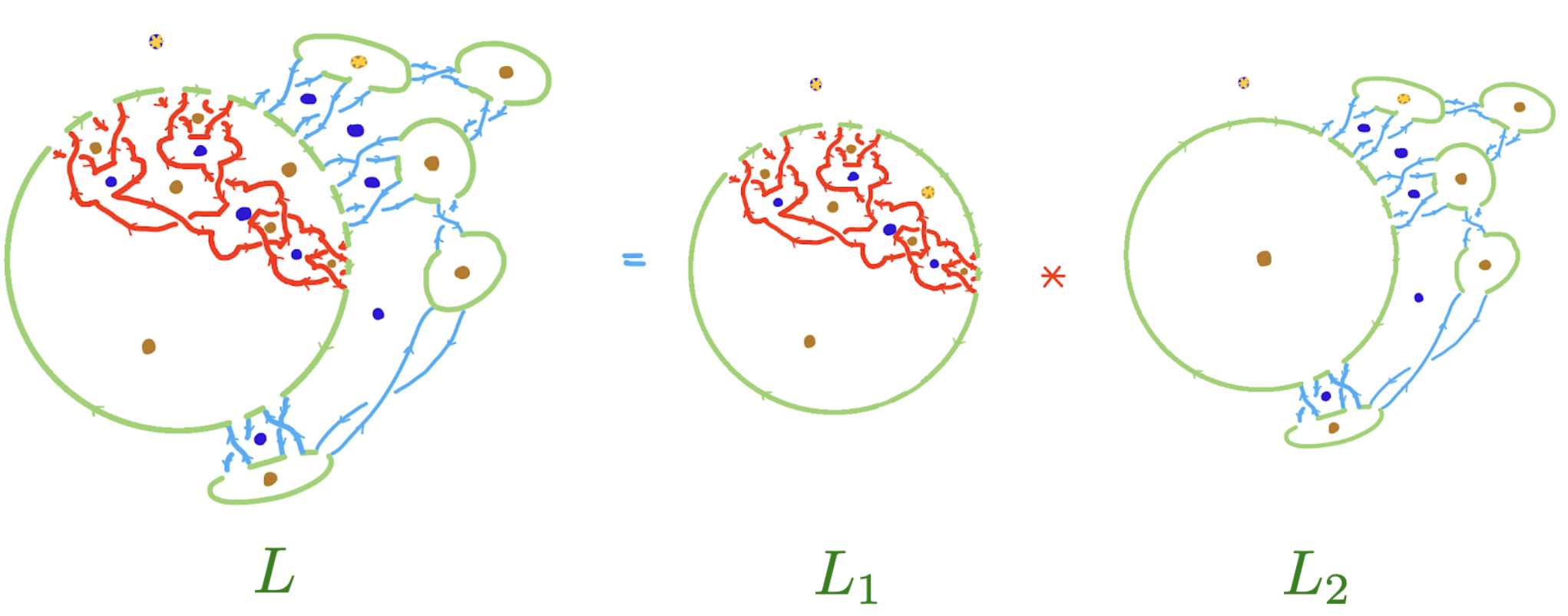}
\caption{A diagrammatic Murasugi sum of length two.}\label{diagMurasugisumoflen2}
\end{figure}

\begin{figure}[h]
\centering
\includegraphics[scale=0.35]{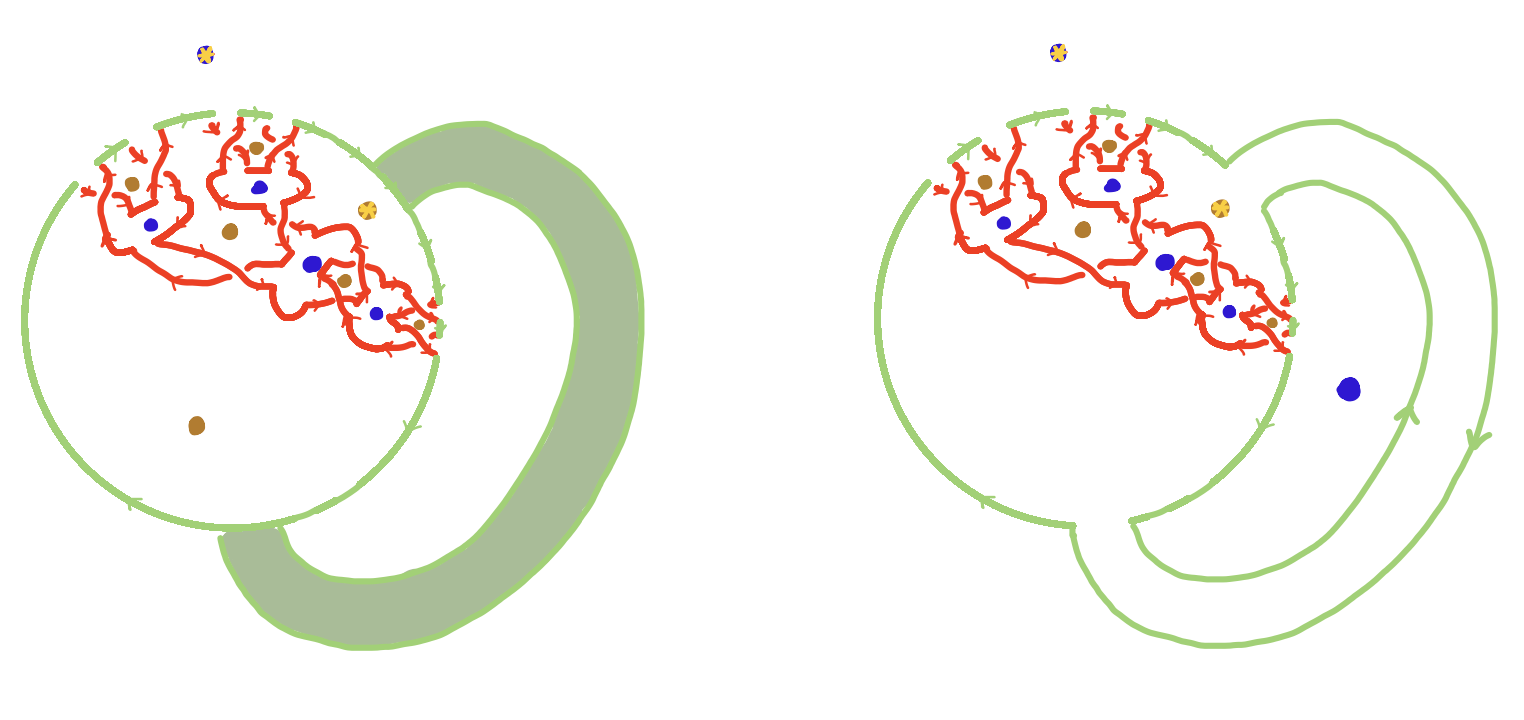}
\caption{Construction of $\widetilde{L}_1$ in the example of Figure~\ref{diagMurasugisumoflen2}.}\label{L1tilde}
\end{figure}

\begin{figure}[h]
\centering
\includegraphics[scale=0.35]{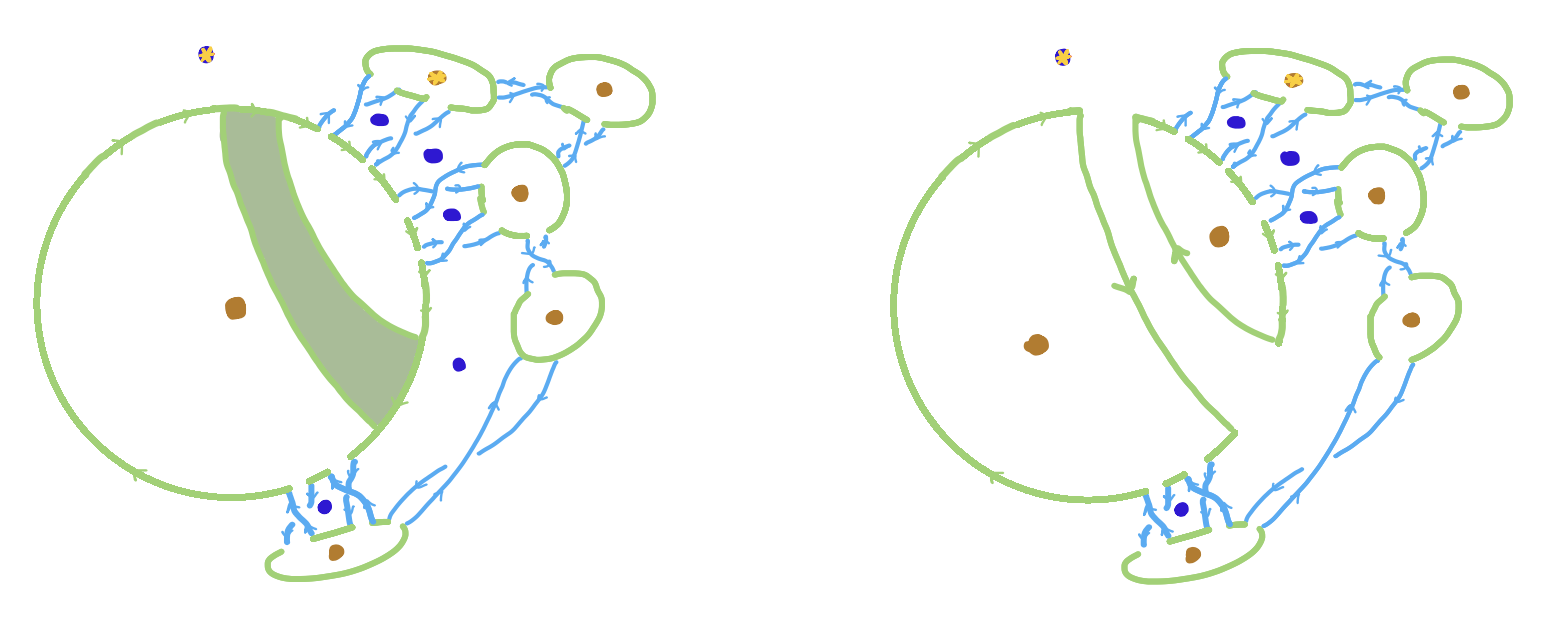}
\caption{Construction of $\widetilde{L}_2$ in the example of Figure~\ref{diagMurasugisumoflen2}.}\label{L2tilde}
\end{figure}

\subsection{A multivariable refinement of the Alexander polynomial for alternating 3-braids}\label{3braids}

In this subsection, we prove some inequalities between coefficients of the Alexander polynomial of alternating 3-braids. We start with a combinatorial generalization of the work of Hafner, Mészáros, and Vidinas~\cite{KarolaLogconcavityOT} on special alternating knots. First, we review some of their main ideas.

Consider a special alternating link $L$ with diagram $D$. Based on Subsection~\ref{plumbingofspecial} (see Figures~\ref{turnspecial} and~\ref{SpecialCrowellgraph}), we know that one of the Tait graphs of $L$ is a bipartite plane graph. We denote this graph by $B$. Furthermore, let $G_D$ be the Crowell graph associated to $D$.

\begin{figure}[h]
\centering
\includegraphics[scale=0.4]{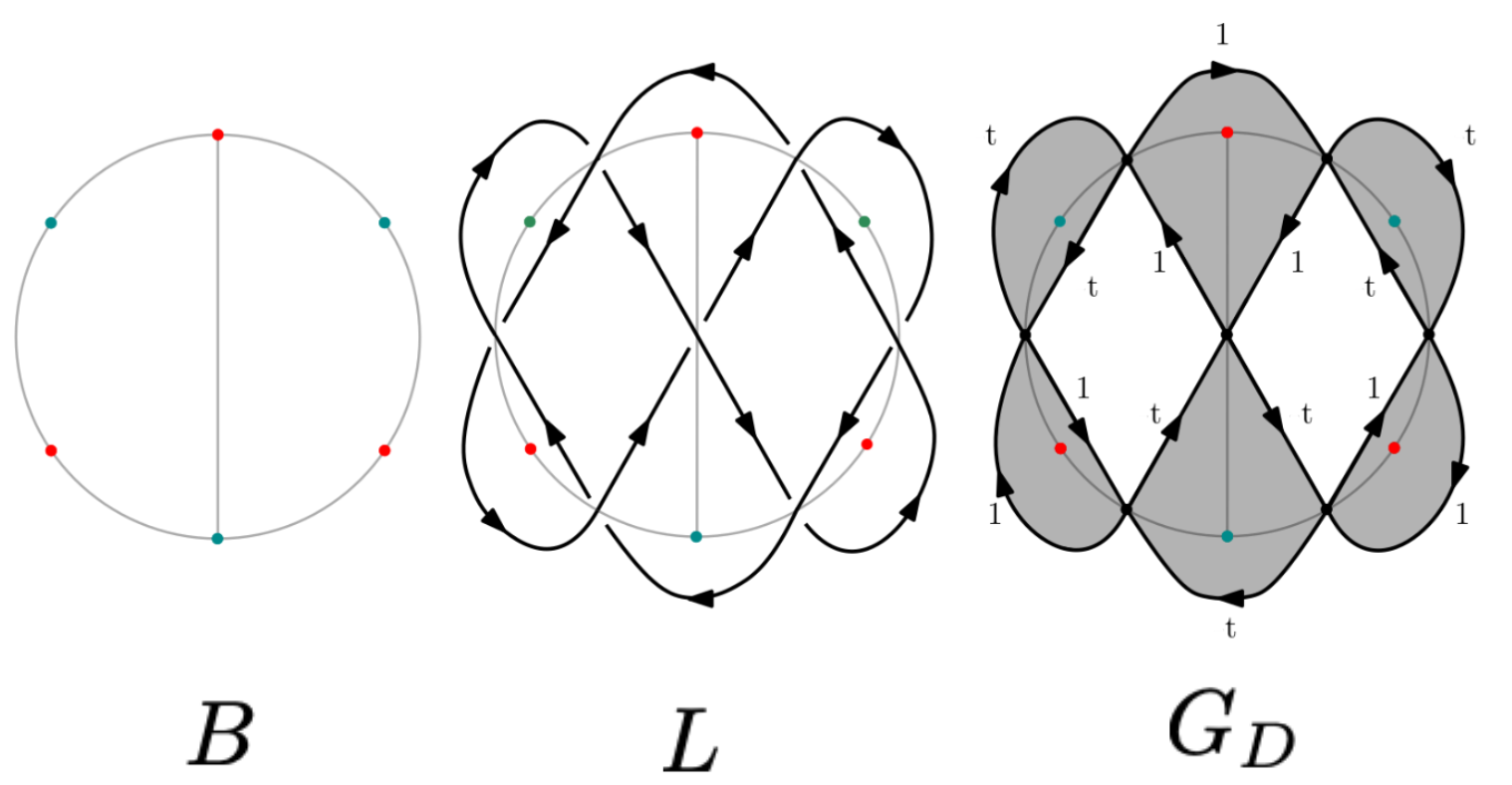}
\caption{A bipartite Tait graph $B$, the corresponding special alternating link $L$, and the Crowell graph $G_D$; see~\cite{KarolaLogconcavityOT}.}\label{SpecialCrowellgraph}
\end{figure}

Note that a special alternating link is either positive or negative. Indeed, suppose that there were two consecutive crossings of different signs. See Figure~\ref{PosNeg}, where the boundaries of three different Seifert cycles are highlighted. These Seifert cycles cannot all bound regions of the link diagram, so they cannot all be Seifert cycles of type~1. Hence, all crossings have the same sign in a special alternating diagram.

\begin{figure}[h]
\centering
\includegraphics[scale=0.3]{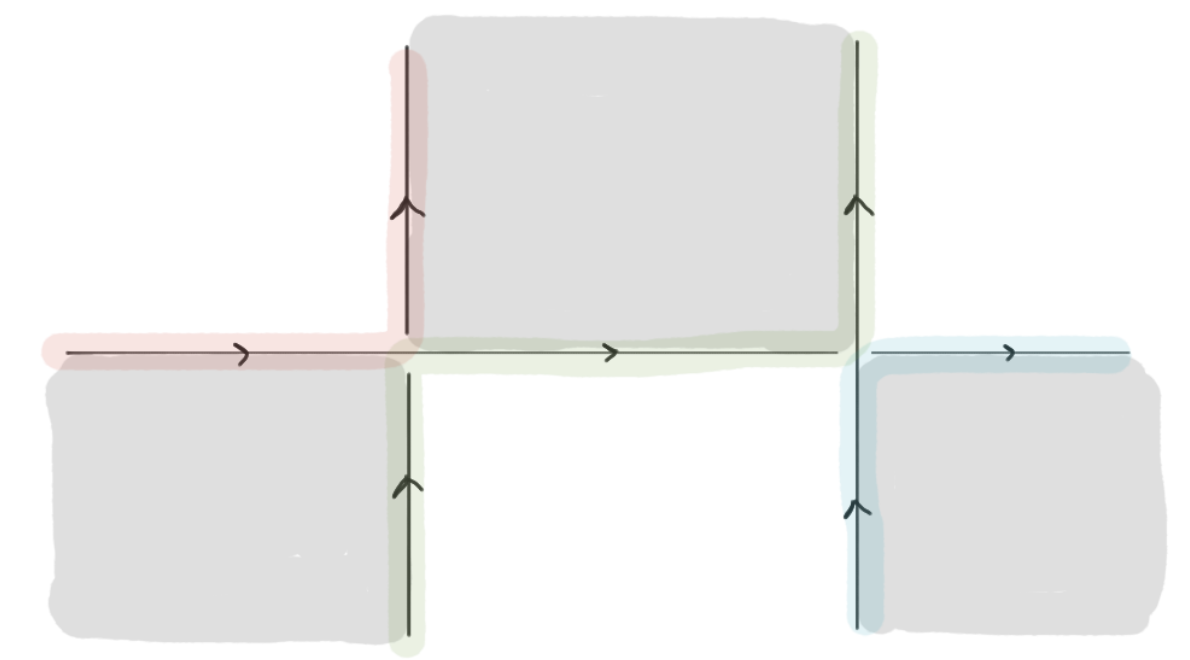}
\caption{Two consecutive crossings of different signs and the Seifert cycles in their neighbourhood.}\label{PosNeg}
\end{figure}

Up to taking its mirror, one can assume that the special alternating diagram is positive. Since the mirror of a knot has the same Alexander polynomial, we only consider positive diagrams in this subsection.

One of the main ideas of Hafner, Mészáros, and Vidinas is to assign new weights to the edges of $G_D$. They choose the weights carefully so that Crowell's spanning tree weighted sum gives a multi-variable refinement of the Alexander polynomial.

Recall that vertices of $B$ correspond to black regions in $G_D$. We use the same symbol to refer to a vertex and the corresponding region.
The Crowell graph of a special alternating diagram has the following property. 

\begin{prop}\label{SpecialCrowellprop}
The edges of $G_D$ are oriented such that the boundary of any black region $v \in B$ is oriented clockwise. Furthermore, there is a partition $V(B)= S \sqcup R$ of the vertex set of $B$ such that all the edges on the boundary of a region in $S$ have weight $t$, and all the edges on the boundary of a region in $R$ have weight $1$.
\end{prop}

\begin{proof}
    Using the fact that the link is special alternating, the result follows from a local analysis of the orientations and weights of the edges of the Crowell graph around vertices of the bipartite Tait graph $B$, as shown in Figure~\ref{localSpecialCrowell}; see also \ Figure~\ref{SpecialCrowellgraph}.
\end{proof}
    
\begin{figure}[h]
\centering
\includegraphics[scale=0.3]{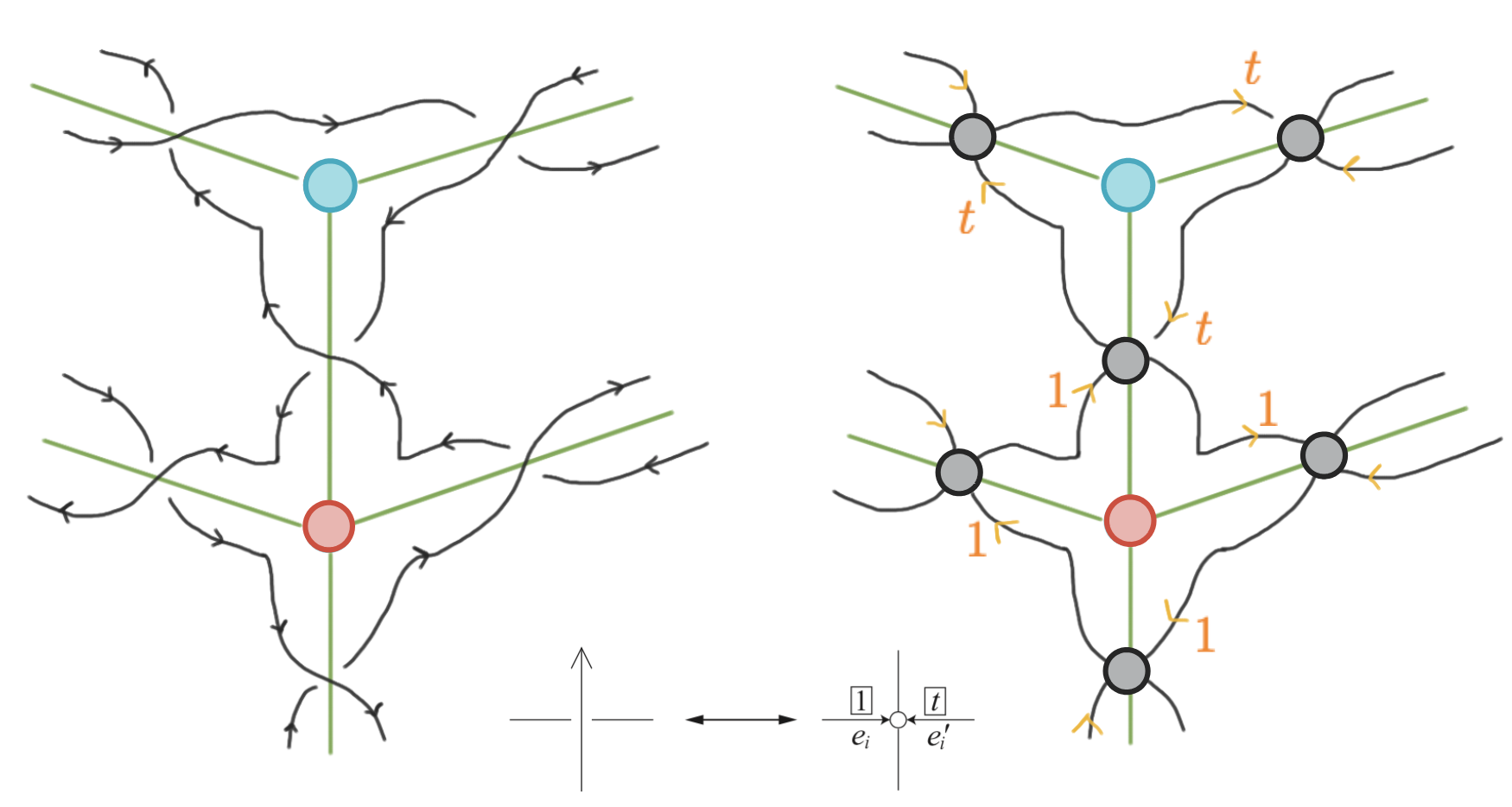}
\caption{A special alternating link (left) and its Crowell graph (right). The blue and red vertices are in $S$ and $R$, respectively.}\label{localSpecialCrowell}
\end{figure}

We define the multi-variable refinement of the Alexander polynomial as follows. Let $S =\{v_1,\dots,v_n\}$ and $R = \{v_{n+1},\dots,v_{n+m}\}$, and let $\widetilde{G}_D$ be the same oriented rooted graph as the Crowell graph $G_D$, but with different weights. Let the weight of the edges along the boundary of the region corresponding to $v_i$ be $x_i$. Recall that $\mathcal{T}\bigl(\widetilde{G}_D, c_1 \bigr)$ is the set of all maximal oriented trees in $\widetilde{G}_D$ rooted at $c_1$. We define
\[
P_{\widetilde{G}_D, c_1}(x_1,\dots,x_n,x_{n+1},\dots,x_{n+m}) = \sum_{T \in \mathcal{T}\left(\widetilde{G}_D, c_1 \right)} W(T),
\]
where $W(T)$ is the product of the weights of the edges in $T$; see Figure~\ref{refinedCrowell}.

\begin{figure}[h]
\centering
\includegraphics[scale=0.35]{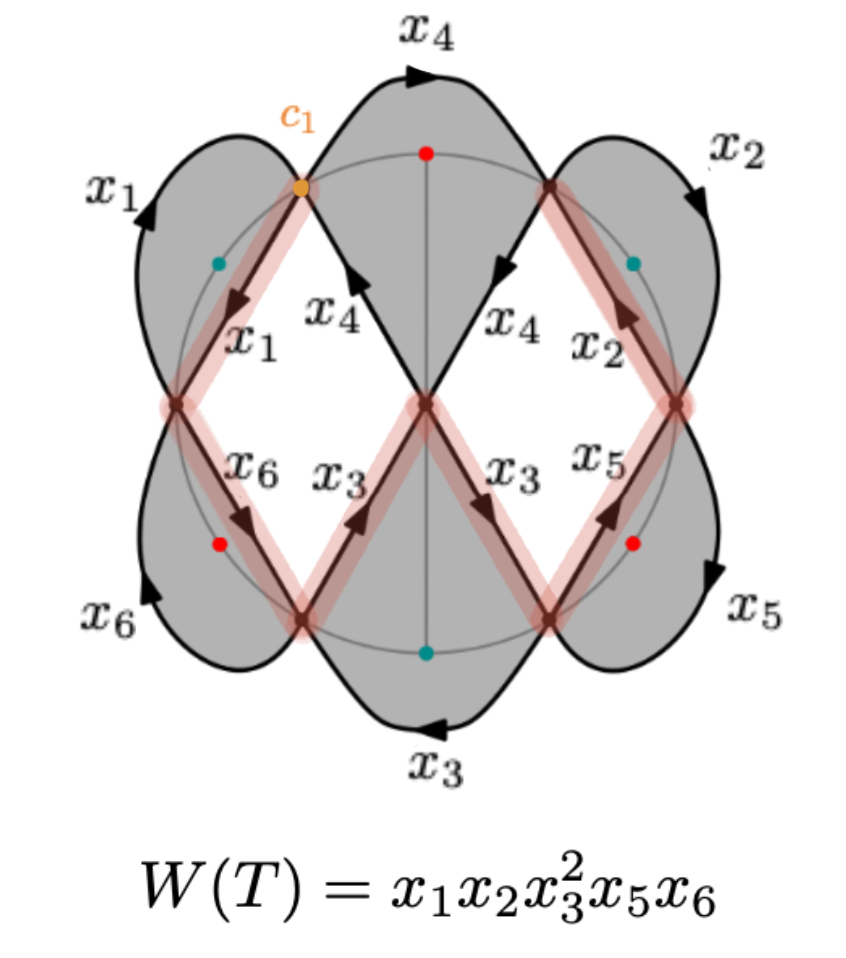}
\caption{Weights on the graph $\widetilde{G}_D$; see~\cite{KarolaLogconcavityOT}.}\label{refinedCrowell}
\end{figure}

Note that $P_{\widetilde{G}_D, c_1}$ is a refinement of the Crowell polynomial for special alternating links, as 
\[
P_{\widetilde{G}_D,c_1}(t,\dots,t,1,\dots,1) = P_{(G_D,c_1)}(t) = \Delta_{L}(-t).
\]
The advantage of working with this refinement is that all of its non-zero coefficients turn out to be $1$. Furthermore, its support has a specific combinatorial convexity property called \textit{M-convexity}; see Definition~\ref{Mconvexdef}. (Here, the support $\text{Supp}(P)$ of a multi-variable polynomial $P$ is the set of all $(a_1,\dots, a_{n+m}) \in \mathbb{Z}^{n+m}$ such that the monomial $x_1^{a_1}\cdots x_{n+m}^{a_{n+m}}$ has non-zero coefficient in $P$.) These facts follow from a result proved by the third author~\cite{KALMAN2013823}; see also \cite{KarolaLogconcavityOT}:

\begin{lemm}\label{Kalman}
Let $C_1,\dots,C_n,C_{n+1},\dots,C_{n+m}$ be cycles in $\widetilde{G}_D$ such that $C_i$ is the boundary of the region $v_i$. Let $T$ be any spanning tree in the unoriented graph underlying $\widetilde{G}_D$. Let $a_i(T)$ be the number of edges of $T$ lying in $C_i$. Then there exists a unique oriented tree $A$ rooted at $c_1$ such that $a_i(A) = a_i(T)$ for all $i \in \{1, \dots, n+m\}$.
\end{lemm}

Using Lemma~\ref{Kalman}, one can conclude that all non-zero coefficients of $P_{\widetilde{G}_D,c_1}$ are $1$. Furthermore,
\[
\text{Supp}\bigl(P_{\widetilde{G}_D,c_1}\bigr) = \bigl\{(a_1(T),a_2(T),\dots,a_{n+m}(T))\ : \ T\in T(\widetilde{G}_D )\bigr\},
\]
where $T(\widetilde{G}_D)$ is the set of all spanning trees of the unoriented graph underlying $\widetilde{G}_D$. The set $T(\widetilde{G}_D)$ is \textit{M-covex} as a subset of $2^{E(\widetilde{G}_D)}$. (In fact, it is the basis of the graphical matroid of the unoriented graph underlying $\widetilde{G}_D$.)

These facts were used by Hafner, Mészáros, and Vidinas~\cite{KarolaLogconcavityOT} to prove that $P_{\widetilde{G}_D,c_1}$ is \textit{denormalized Lorentzian}, which is a generalization of log-concavity developed by Brändén and Huh~\cite{Huh}. This implies that the Alexander polynomial is log-concave.

The main problem with generalizing this idea to non-special diagrams is that Proposition~\ref{SpecialCrowellprop} does not hold. Type~2 Seifert cycles still correspond to oriented cycles in the Crowell graph, but they contain edges with weights both $t$ and $1$. To construct a multi-variable refinement of the Alexander polynomial, one needs to use two variables for the edges of a type~2 Seifert cycle. See Figure~\ref{refinedCrowell3braid} for an example. Note that, we also use the term ``Seifert cycle'' to refer to the corresponding oriented cycle in the Crowell graph.

The simplest class of non-special alternating links is given by closures of alternating 3-braids, which have only one type~2 Seifert cycle. We are going to focus on this family of links in the rest of this subsection.

Let $D$ be the standard diagram of the closure of an alternating 3-braid
\[
\sigma_1^{p_1}\sigma_2^{-q_1}\cdots\sigma_1^{p_n}\sigma_2^{-q_n}
\]
for $p_i$, $q_i \in \mathbb{Z}_{>0}$, where $\sigma_1$ and $\sigma_2$ are the standard generators of the braid group on three strands.
This diagram $D$ is depicted on the left of Figure~\ref{refinedCrowell3braid} as a long knot, with the second strand of the braid closing at $\infty$. This diagram contains two type~1 Seifert cycles, one following the leftmost strand and one following the rightmost. We denote them by $L$ and $R$, respectively, shown in blue and green in Figure~\ref{refinedCrowell3braid}. The diagram also has a type~2 Seifert cycle following the middle strand. 

In the Crowell graph $G_D$, the left Seifert cycle is oriented counterclockwise with weight~$t$ and the right one is oriented clockwise with weight~$1$. The middle strand is directed from top to bottom, and it contains two types of edges. The edges ending in $L$, denoted by $\textit{EL}$, with weight $1$, and the edges ending in $R$, denoted by $\textit{ER}$, with weight $t$. The refined Crowell graph $\widetilde{G}_D$ is shown on the right of Figure~\ref{refinedCrowell3braid}. The refined weights are $x_1$, $x_2$, $x_3$, and $x_4$ on $L$, $\textit{ER}$, $\textit{EL}$, and $R$, respectively. In order to simplify the notation in the following argument, we use the color-coding of Figure~\ref{refinedCrowell3braid}. We use colors \textit{blue, red, black, green} instead of $x_1,x_2,x_3,x_4$, respectively. 

\begin{figure}[h]
\centering
\includegraphics[scale=0.4]{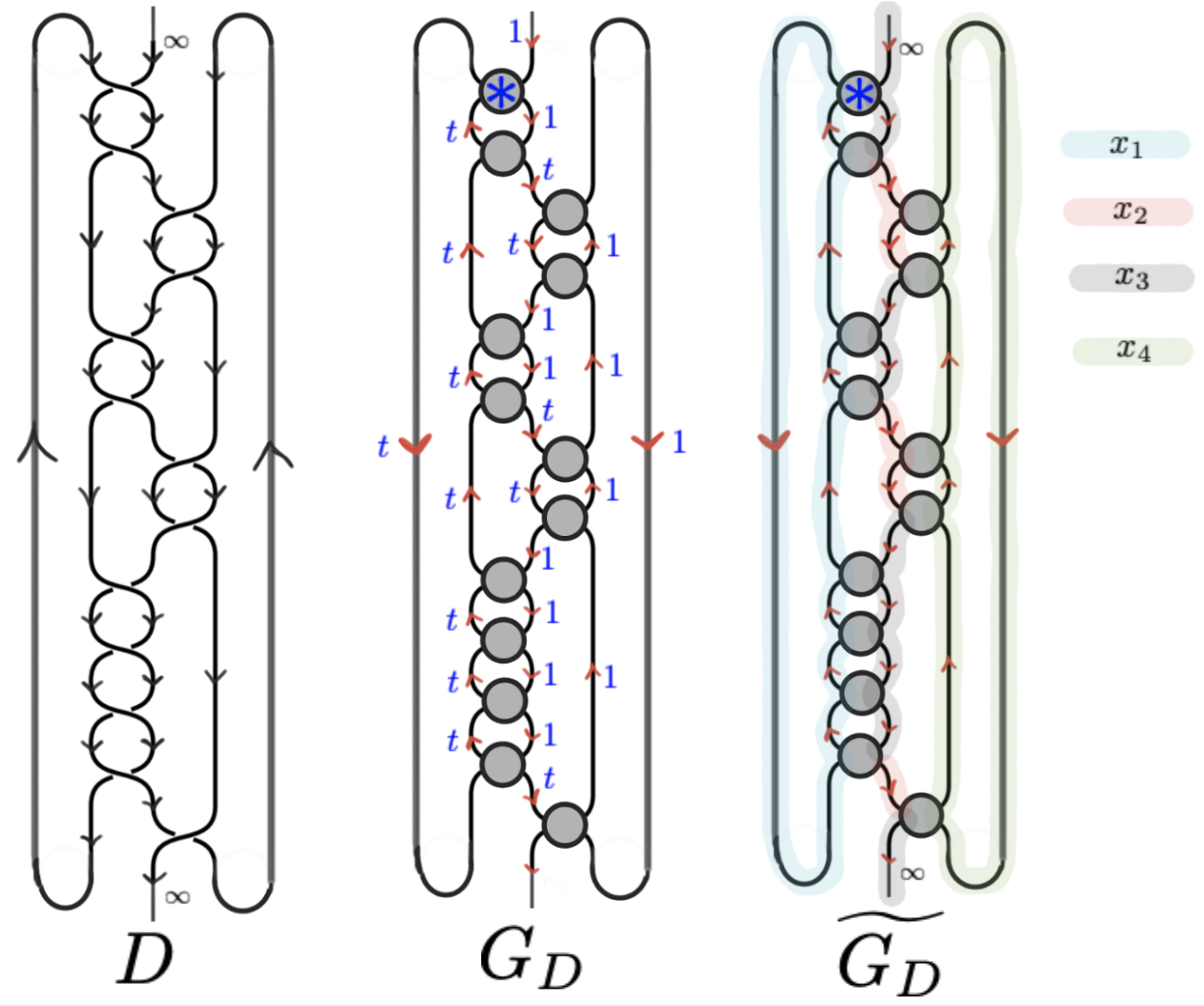}
\caption{An alternating 3-braid (left), its Crowell graph (middle), and the refined Crowell graph (right). The root vertex $c_1$ is denoted by a star.}\label{refinedCrowell3braid}
\end{figure}

The refined multi-variable polynomial $P_{\widetilde{G}_D,c_1}(x_1,x_2,x_3,x_4)$ is defined analogously to the case of special alternating links; i.e., as the weighted sum of oriented spanning trees rooted at $c_1$. Unfortunately, this polynomial has coefficients that are greater than $1$, but it has a nice support:

\begin{lemm}\label{Supportof3braid}
Let $P$ and $Q$ be the number of vertices in $L$ and $R$, respectively. Note that $P = p_1+\cdots+p_n$ and $Q = q_1+\cdots+q_n$. Then 
\[
Supp\bigl(P_{\widetilde{G}_D,c_1}\bigr) =
\]
\[
\{\,(a_1,a_2,a_3,a_4) \in \mathbb{Z}^4 \,:\, a_1,a_3,a_4\geq0 ,\ a_2>0,  \ a_1+a_3=P-1 , \ a_2+a_4=Q \,\}.
\]
\end{lemm}

\begin{proof}
Consider the braid diagram shown in Figure~\ref{refinedCrowell3braid}, with the standard height function in the plane. Let $c_1$ be the topmost vertex of $L$. Any oriented spanning tree $T$ rooted at $c_1$ contains a unique edge ending at each vertex other than $c_1$. So $T$ must contain at least one edge of type $\textit{EL}$. The set of edges in $T$ with weight $x_1$ or $x_3$ is precisely the set of edges ending in $L$.  Hence, the number of edges in $T$ with weight $x_1$ or $x_3$ is $|L|-1 = P-1$. The same argument shows that the number of edges in $T$ with weight $x_2$ or $x_4$ is $|R| = Q$.

Hence, it suffices to construct a tree $T \in \mathcal{T}(\widetilde{G}_D,c_1)$ with weight $W(T) = x_1^{a_1}x_2^{a_2}x_3^{a_3}x_4^{a_4}$ for any $(a_1,a_2,a_3,a_4) \in \mathbb{Z}^4$ with $a_1, a_3, a_4 \geq 0$, $a_2>0$, $a_1 + a_3 = P-1$, and $a_2+a_4 = Q$. 
Denote the topmost vertex of $R$ by $c'_1$. Consider the following subsets of the edge set $E(\widetilde{G}_D)$:
\begin{align*} 
      &E_1=\text{the first}\  a_1 \  \text{edges of}\  L\  \text{starting at} \ c_1,\\
      &E_2=\text{the topmost}\  a_2\  \text{red edges,}\ \\
      &E_3=\text{the topmost}\  a_3\  \text{black edges,}\ \\
      &E_4=\text{the first}\ a_4 \ \text{edges of}\  R\  \text{starting at} \  c'_1.
\end{align*}
Let $E_T=E_1 \cup E_2 \cup E_3 \cup E_4$. We claim that the edges of $E_T$ form the desired oriented rooted tree $T$.

Since $a_1 + a_3 = P-1$ and $a_2 + a_4 = Q$, for any vertex $c$ other than the root $c_1$, there is exactly one edge in $E_T$ leading to $c$. Indeed, the edges in $E_1$ lead to the lowest $a_1$ vertices of $L$, and the edges in $E_3$ lead to the topmost $a_3$ vertices of $L$, excluding $c_1$. 
Similarly, edges in $E_2$ lead to the topmost $a_2$ vertices of $R$, and edges in $E_4$ lead to the lowest $a_4$ vertices of $R$. Since $a_2 > 0$, the edges $E_T$ do not contain the cycle $R$.
See the right of Figure~\ref{Producttrees} for the edges in $E_T$ when $(a_1,a_2,a_3,a_4)=(3,3,4,2)$, for the diagram in Figure~\ref{refinedCrowell3braid}.

By the above, for any vertex $v \neq c_1$, we can define the parent vertex $p(v)$ of $v$ such that $(p(v),v) \in E_{T}$; i.e., as the source of the unique edge in $E_T$ leading to $v$. Note that $p(v)$ is higher than $v$, unless possibly when $(p(v),v) \in E_1 \cup E_4$; i.e., when $(p(v),v)$ is an edge in one of the cycles $L$ and $R$. For some $k \in \mathbb{N}$, the sequence 
\[
(p^k(v), \dots, p(v),v)
\]
is a path from $c_1$ to $v$ in $E_T$. So $E_T$ is the edge set of an oriented rooted tree $T$ with the desired weight.
\end{proof}
\begin{figure}[h]
\centering
\includegraphics[scale=0.4]{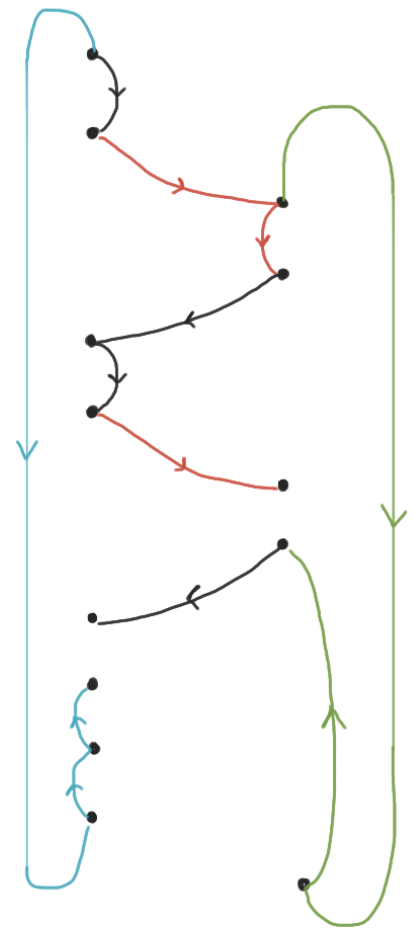}
\caption{An oriented spanning tree $T$ of the refined Crowell graph shown in Figure~\ref{refinedCrowell3braid} for $(a_1,a_2,a_3,a_4)=(3,3,4,2)$.}\label{Producttrees}
\end{figure}

There is a $P \times Q$ matrix $A=[A_{i,j}]$ for $(i, j) \in \{0,\dots, P-1\} \times \{0,\dots, Q-1\}$ such that
\begin{equation}\label{formof4-varmatrix}
P_{\widetilde{G}_D,c_1}= \sum\limits_{\substack{0\leq i \leq P-1 \\ 0 \leq j \leq Q-1}} A_{i,j}\ x_1^{P-1-i}x_2^{Q-j}x_3^{i}x_4^{j}.
\end{equation}
We can recover the Crowell polynomial by setting $x_1=x_2=t$ and $x_3=x_4=1$; i.e.,
\begin{equation}\label{Alexanderfrom4var}
    \Delta_{L}(-t)\doteq P_{G_D,c_1} =\sum\limits_{0 \leq k \leq P+Q-2}\ \left(\sum\limits_{i+j=k} A_{i,j} \right) \ t^{P+Q-1-k}.
\end{equation}
The main result of this subsection are inequalities between coefficients $A_{i,j}$, which lead to inequalities between coefficients of the Alexander polynomial.

\begin{theo}\label{badinequality}
Let $D$ be the diagram of the closure $L$ of the alternating 3-braid
\[
\sigma_1^{p_1}\sigma_2^{-q_1}\cdots\sigma_1^{p_n}\sigma_2^{-q_n}
\]
shown in Figure~\ref{refinedCrowell3braid}. Let $P=p_1+\cdots+p_n$ and $Q=q_1+\cdots+q_n$. If $A_{i,j}$ is the $P\times Q$ matrix of coefficients of $P_{\widetilde{G}_D,c_1}$, as above, then 
\[
n A_{i,j+1} \geq A_{i,j}.
\]
If 
\[
\Delta_{L}(-t) = \sum\limits_{0 \leq k \leq P+Q-2} a_k t^{k}, 
\]
then $a_{k+1} \geq a_{k}/n + A_{0,k+1} > a_k/n$ for $k \in \{0, \dots, P-2\}$.
\end{theo}

\begin{proof}
    Let $\mathcal{T}_{i,j}$ be the set of oriented rooted spanning trees of $\widetilde{G}_D$ with weight $x_1^{P-1-i}x_2^{Q-j}x_3^{i}x_4^{j}$; i.e., ones contributing to the coefficient $A_{i,j}$. Equivalently, $\mathcal{T}_{i,j}$ is the set of maximal oriented rooted trees in $\widetilde{G}_D$  with $i$ black edges and $j$ green edges.
    
    By manipulating trees, we build a map $L \colon \mathcal{T}_{i,j} \rightarrow \mathcal{T}_{i,j+1}$ which is at most $n$-to-$1$. This implies that $n A_{i,j+1} \geq A_{i,j}$. The construction of the map $L$ is as follows.

    Recall that $c_1$ and $c'_1$ were used to denote the highest vertices of $L$ and $R$, respectively. Let $c_1, \dots, c_P$ be the vertices of $L$, and $c'_1, \dots, c'_Q$ be the vertices of $R$, both labelled from top to bottom.

    Let $T \in \mathcal{T}_{i,j}$ be an oriented, rooted spanning tree. We call $c'_i$ a \textit{maximal vertex} if the path in $T$ from $c_1$ to $c'_i$ only includes vertices in $L$, except for $c'_i$. Let the lowermost maximal vertex be $c'_m$. Consider the longest oriented path of green edges in $T$ originating from $c'_m$. Note that the path might be of length zero. Assume that this path ends at $c'_k$ and hence does not reach $c'_{k-1}$. Then there is a red edge in $T$ leading to $c'_{k-1}$, which we denote by $e$. The map $L$ is defined by 
    \[
    L(T) := T \setminus \{e\} \cup \{(c'_k , c'_{k-1})\}.
    \]
    See Figure~\ref{nto1functionL} for an example, where the maximal vertices are $c'_1$ and $c'_7$, the lowest maximal vertex is $c'_m = c'_7$, and the green path goes from $c'_7$ to $c'_k = c'_6$.

    We first show that $\text{Im}(L) \subseteq \mathcal{T}_{i,j+1}$. Based on the construction of $L(T)$, the weight $W(L(T)) = W(T) x_4 x_2^{-1}$. Hence, we only need to show that it is an oriented rooted spanning tree. The rooted tree $T$ induces a partial order $>_{T}$ on the vertices $V(\widetilde{G}_D)$ of $\widetilde{G}_D$. The relation $v >_{T} w$ holds if and only if there is a path from $v$ to $w$ in $T$. For any vertex $v$, the subset $T_{v} = \{x \in V(\widetilde{G}_D) :  v \geq_{T} x\}$ forms an oriented subtree of $T$ rooted at $v$. On the other hand, any partial order on $V(\widetilde{G}_D)$ with $c_1$ as its maximum also gives an oriented rooted spanning tree of $\widetilde{G}_D$.

    Note that the construction of $L(T)$ can be summarized by setting $c'_k$ to be the parent of $c'_{k-1}$. This operation results in an oriented rooted spanning tree if and only if $c'_{k} \notin T_{c'_{k-1}}$; see Figure~\ref{treeoperations2}. This condition holds for the operation in the construction of $L(T)$. Indeed, if $c'_{k} \in T_{c'_{k-1}}$, then there is an oriented path from $c'_{k-1}$ to $c'_{k}$, which contradicts the assumption that $c'_{m}$ is a maximal vertex.
    
\begin{figure}[h]
\centering
\includegraphics[scale=0.3]{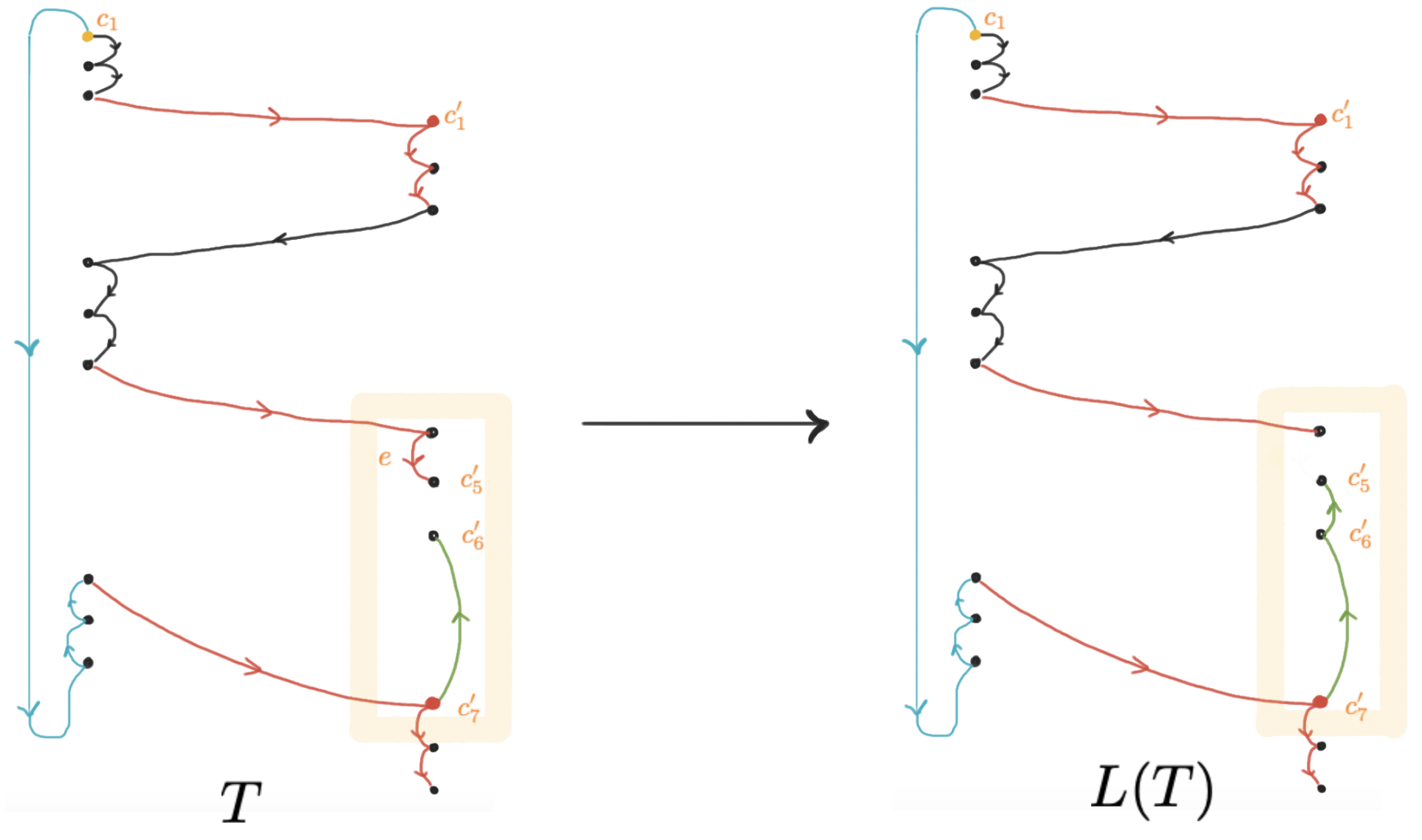}
\caption{The map $L \colon \mathcal{T}_{i,j} \rightarrow \mathcal{T}_{i,j+1}$.}\label{nto1functionL}
\end{figure}

We now prove that $L$ is at most $n$-to-$1$; i.e., we show that, for any $A \in \mathcal{L}_{i,j+1}$, we have $|L^{-1}(A)| \leq n$.
To find the elements of $L^{-1}(A)$, first notice that the lowermost maximal vertices of $T$ and $L(T)$ are the same, since the construction of $L(T)$ does not affect the path from $c_1$ to $c'_m$, and just turns $c'_{k-1}$ from a possibly maximal vertex to a non-maximal one.

Using this fact, consider the lowermost maximal vertex of $A$ and denote it by $c'_m$. Denote by $J$ the largest path of green edges in $A$ starting from $c'_m$. By construction of $L$, we know that, if $T \in L^{-1}(A)$, then $A$ can be constructed by taking the longest path of green edges in $T$ starting from $c'_m$ and adding an edge at the end. The trees $T$ and $A$ only differ in two edges, one green and one red. Denote the one extra red edge of $T$ by $e_T=(x,y)$. Based on the explanation above, $y$ must be in $J$.

\begin{figure}[h]
\centering
\includegraphics[scale=0.4]{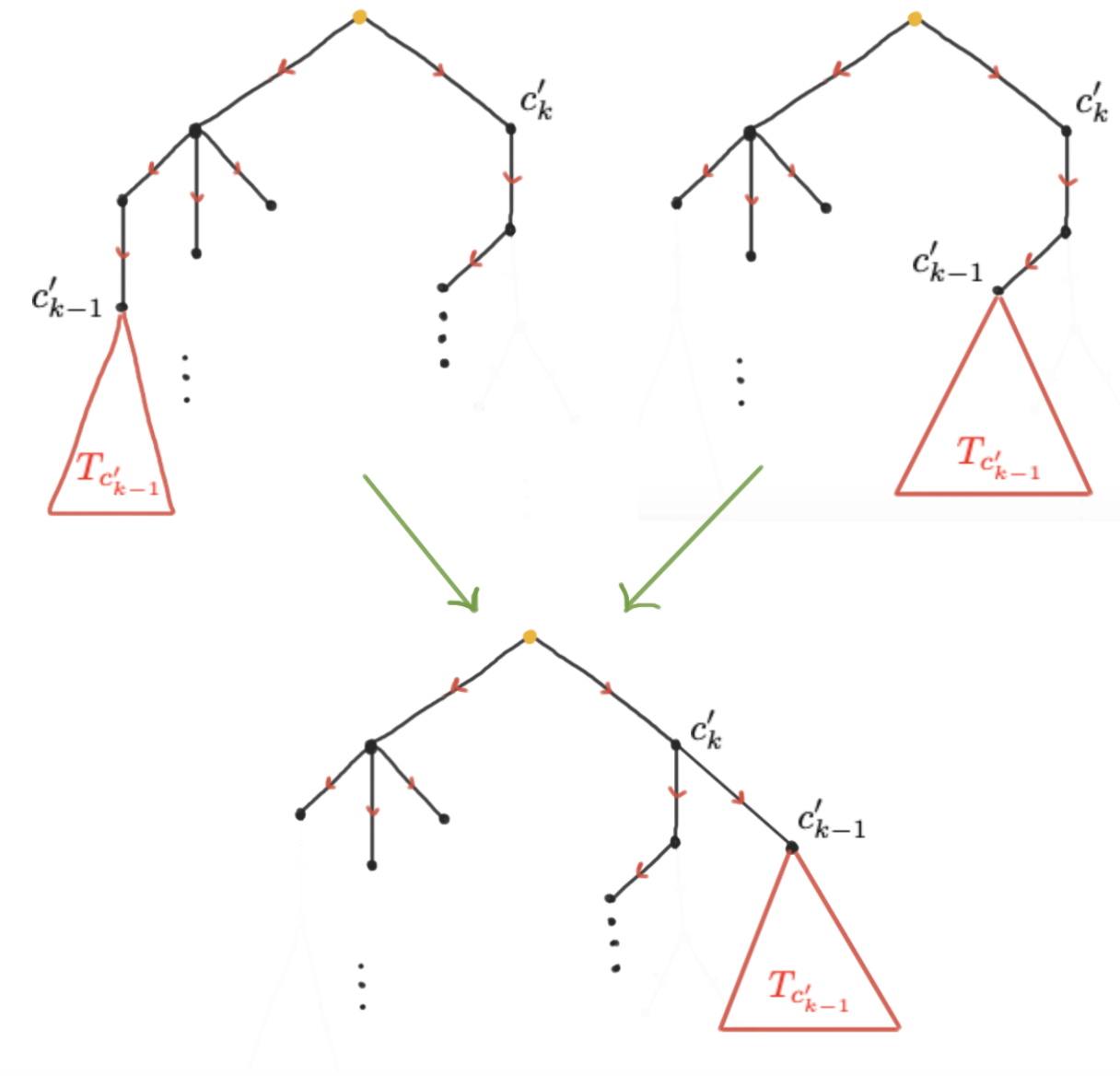}
\caption{Two different examples of the \textit{change of parent operation} in a rooted tree. In both examples, $c'_k \notin T_{c'_{k-1}}$ and the result of the operation is still an oriented rooted tree.}\label{treeoperations2}
\end{figure}

\begin{figure}[h]
\centering
\includegraphics[scale=0.4]{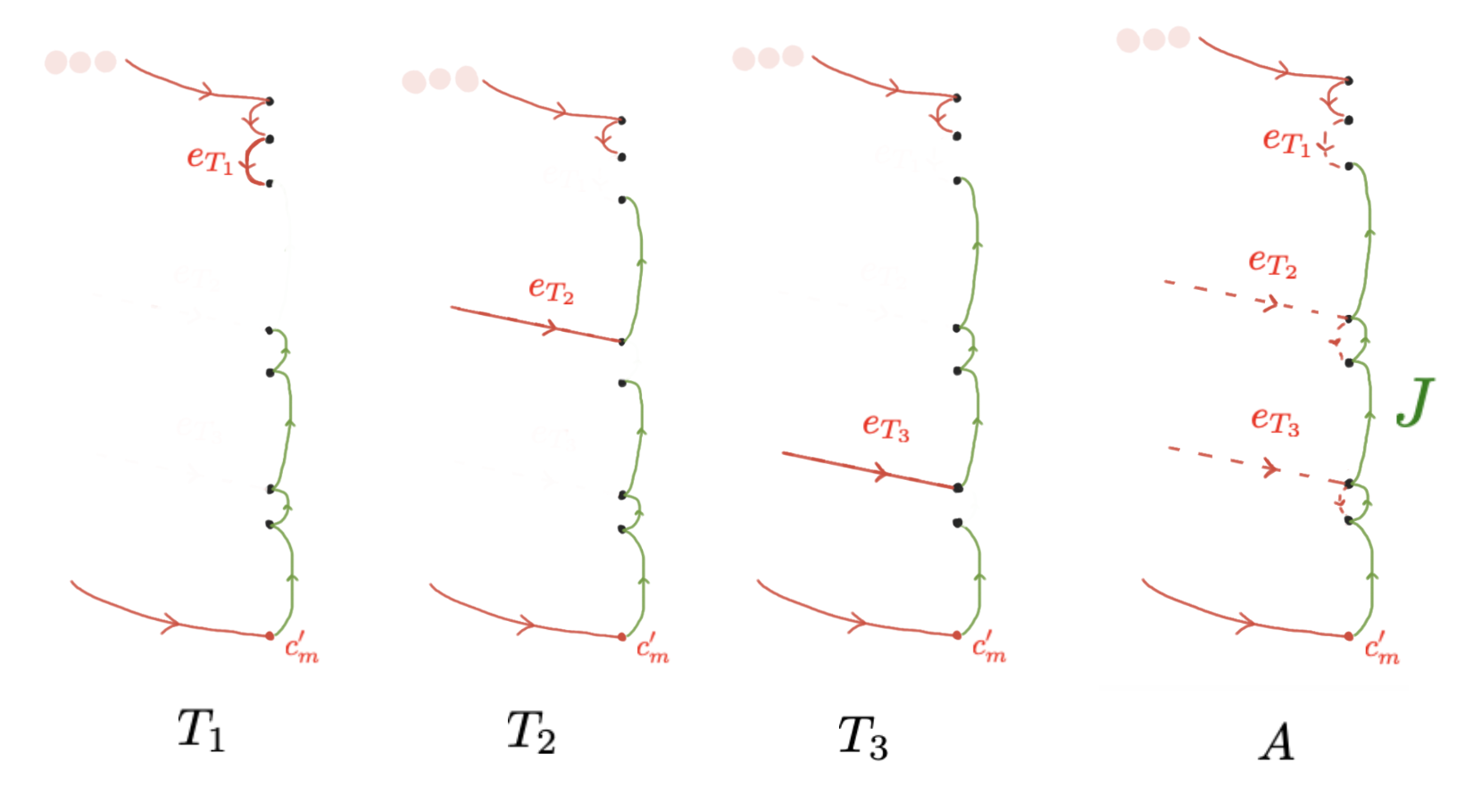}
\caption{Construction of the inverse image $L^{-1}(A)$. The possible options for $e_T$ come from the red edges leading to $J$ from outside. One can see that using a red edge with both ends in $J$ does not work.}\label{inverseimageofL}
\end{figure}

To construct $T$ from $A$, we delete the green edge leading to $y$ and replace it with $e_T$. As a result, $x$ must be outside of $J$, otherwise the green edge from $(y,x)$ in $A$ is also in $T$ and forms a cycle with $e_T$. This restricts the options for $e_T$: it is either one of the edges starting in $L$ and ending in $R$, or it is an edge leading to the final leaf of $J$. So there are at most $n$ options for $e_T$, hence $|L^{-1}(A)| \le n$. See Figure~\ref{inverseimageofL} for an example.
This completes our proof of the first inequality. The inequality on the Alexander polynomial coefficients follows from equation~\eqref{Alexanderfrom4var}. 
\end{proof}

We continue this subsection with some corollaries of Theorem~\ref{badinequality}. First, note that, by taking the mirror of the braid diagram and turning it upside down, one can replace the roles of $\sigma_1$ and $\sigma_2$ without changing the Alexander polynomial.

Furthermore, one can improve the bound on the number of choices for $e_T$, and obtain the following strengthening of Theorem~\ref{badinequality}:

\begin{coro}\label{strongbadinequality}
Let $D$ be the standard diagram of the closure $L$ of the alternating 3-braid
\[
\sigma_1^{p_1}\sigma_2^{-q_1}\cdots\sigma_1^{p_n}\sigma_2^{-q_n}.
\]
Let $P = p_1+\cdots+p_n$ and $Q = q_1+\cdots+q_n$. If 
\[
\Delta_{L}(-t) = \sum\limits_{0 \leq k \leq P+Q-2} a_k t^{k}
\]
and $[A_{i,j}]$ is the $P \times Q$ matrix of coefficients of $P_{\widetilde{G}_D,c_1}$, as above, then 
\[
a_{k+1} \geq \frac{a_{k}}{\min(n,k+1)} + A_{0,k+1} > \frac{a_k}{\min(n,k+1)}
\]
for $k \in \{0, \dots, \max(P-2,Q-2)\}$.
\end{coro}

\begin{proof}
The majority of the proof is the same as proof of Theorem \ref{badinequality}. We only improve the upper bound on $|L^{-1}(A)|$. As we discussed, $|L^{-1}(A)|$ is the same as the number of possible choices for the edge $e_T$, which is a red edge leading to $J$. An upper bound on the number of these choices is the length of the path $J$. Each edge in $J$ counts with weight $1$ in the computation of $W(T)$. Hence $|J| < P+Q-2-\text{pw}(A)$, where $W(A) = t^{\text{pw}(A)}$.
\end{proof}

One might ask if this argument can be generalized to alternating links other than alternating 3-braids. The main ingredient of the proof of Theorem~\ref{badinequality} is the shape of the Crowell graph around the Seifert cycle $R$. In fact, this only relies on the Seifert cycle $R$ being of type~1.

The local shape of the diagram around a Seifert cycle of type~1 is depicted in the top of Figure~\ref{Seiferttypeone}. The diagram consists of $n$ twist regions along the boundary of the cycle. The orientations and weights of the Crowell graph are shown on bottom of Figure~\ref{Seiferttypeone}. The edges are colored analogously to a neighborhood of the Seifert cycle $R$ in the Crowell graph of an alternating 3-braid. Using this coloring, the argument of the proof of Theorem~\ref{badinequality} can be repeated, leading to Theorem~\ref{badinequalitygeneral}.

\begin{figure}[h]
\centering
\includegraphics[scale=0.4]{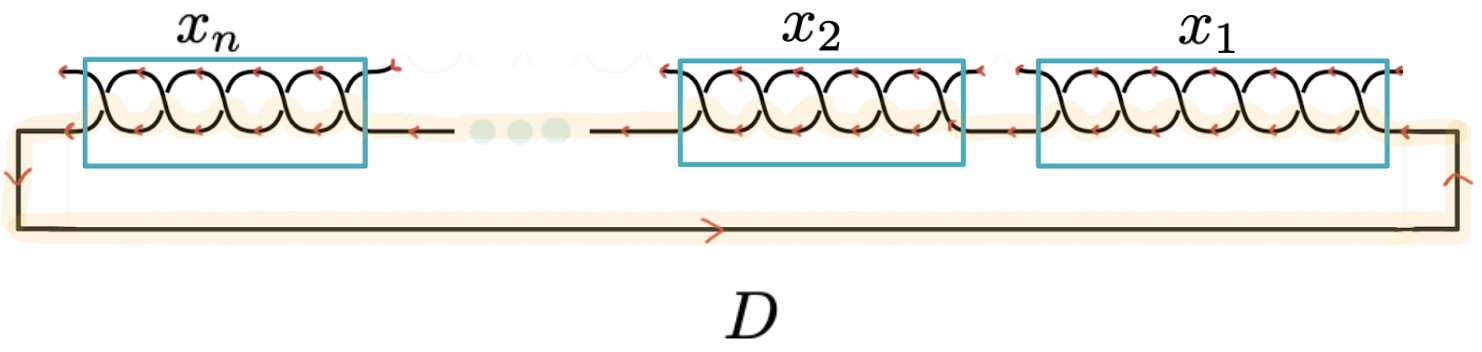}
\includegraphics[scale=0.4]{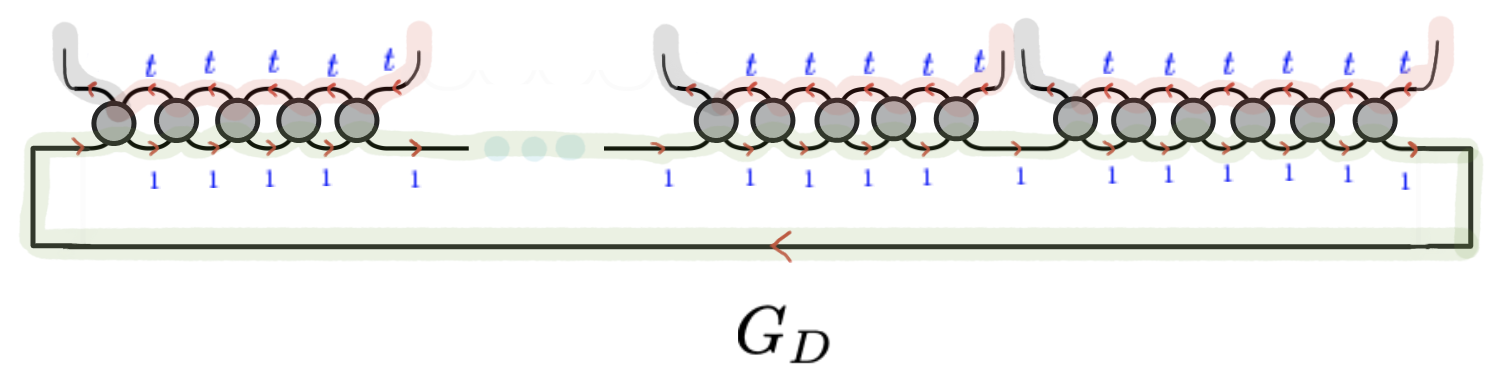}
\caption{Top: neighborhood of a type~1 Seifert cycle in an alternating diagram. Bottom: the corresponding Crowell graph, together with a coloring of its edges.}\label{Seiferttypeone}
\end{figure}

\begin{theo}\label{badinequalitygeneral}
Let $D$ be an alternating diagram of a link $L$. Assume that $D$ contains a Seifert cycle $C$ of type~1 such that there are $n$ twist regions around the $C$ with $x_1,\dots,x_n$ crossings, respectively. Let $c = x_1+\cdots+x_n$ and $\Delta_{L}(-t) = \sum\limits_{0 \leq k} a_k t^{k}$. Then 
\[
a_{k+1} \geq a_{k}/n
\]
for $k \in \{0, \dots, c-2\}$.
\end{theo}


Recall that the construction of the refined Crowell graph for alternating 3-braids (Figure~\ref{refinedCrowell3braid}) was motivated by the result of Hafner, Mészáros, and Vidinas~\cite{KarolaLogconcavityOT}. One might try to repeat their argument for alternating 3-braids. Unfortunately, this does not work, as stated in the following Proposition.
\begin{prop}\label{notLorentzian}
Let $P_{\widetilde{G}_D,c_1}$ be the refined Crowell polynomial of the closure of an alternating 3-braid $B=\sigma_1^{p_1}\sigma_2^{-q_1}\cdots\sigma_1^{p_n}\sigma_2^{-q_n}$. Then $P_{\widetilde{G}_D,c_1}$ is not denormalized Lorentzian unless the 3-braid is a connected sum of two torus links; i.e., $B = \sigma_1^{p_1}\sigma_2^{-q_1}$.    
\end{prop}

Proposition~\ref{notLorentzian} follows from the results we prove in the rest of this section, more specifically Propositions~\ref{Lorisproduct} and \ref{extermalA}. To explain these results, we first need to recall some of the basic definitions regarding Lorentzian polynomials.

\begin{defi}\label{Mconvexdef}
A subset $J \subset \mathbb{N}^n$ is called \textit{M-convex} if, for any $\alpha=(\alpha_1,\dots,\alpha_n) \in J$ and $\beta=(\beta_1,\dots, \beta_n) \in J$ such that $\alpha_i > \beta_i$ for some $i \in \{1,\dots,n\}$, there exist $j \in \{1,\dots,n\}$ satisfying 
\[
\alpha_j < \beta_j, \ \ \alpha-e_i+e_j \in J, \ \text{and} \ \beta-e_j+e_i \in J,
\]
where $e_i$ and $e_j$ are the $i$-th and $j$-th standard basis vectors of $\mathbb{R}^n$, respectively.    
\end{defi}

Let $H_n^{d}$ be the space of degree $d$ homogeneous polynomials with real coefficients in the variables $x_1,\dots,x_n$. The Hessian of a homogeneous quadratic polynomial $f \in H_n^{2}$ is the symmetric matrix $H=(H_{ij})_{i,j \in [n]}$, where $H_{ij}=\frac{\partial}{\partial x_i}\frac{\partial}{\partial x_j}f$ and $[n] = \{1,\dots,n\}$. The set $L_n^{d}$ of \textit{Lorentzian polynomials} of degree $d$ in $n$ variables is defined as follows. Let $L_n^{1} \subseteq H_n^{1}$ be the set of all linear polynomials with nonnegative coefficients. Let  $L_n^{2} \subseteq H_n^{2}$ be the subset of quadratic polynomials with nonnegative coefficients and M-convex support whose Hessians have at most one positive eigenvalue. For $d>2$, define $L_n^{d} \subseteq H_n^{d}$ as the subset of polynomials with nonnegative coefficients and M-convex support that satisfy
\[
\frac{\partial}{\partial x_{i_1}} \frac{\partial}{\partial x_{i_2}} \cdots \frac{\partial}{\partial x_{i_{d-2}}} f \in L_n^{2} \ \text{for all} \ i_1,i_2,\dots,i_{d-2} \in [n].
\]
The \textit{normalization operator} $N$ on $\mathbb{R}[x_1,\dots,x_n]$ is defined by linear extension of the map 
\[
N(x^{\alpha})=\frac{x^{\alpha}}{\alpha!},
\]
where $\alpha=(\alpha_1,\dots,\alpha_n)\in\mathbb{N}^n$, $x^{\alpha}:=x_1^{\alpha_1}\cdots x_n^{\alpha_n}$, and $\alpha!:=\prod\limits_{1 \leq i \leq n} \alpha_i!$.

\begin{defi}
A polynomial $f \in H_n^{d}$ is called \textit{denormalized Lorentzian} if $N(f) \in L_n^{d}$.
\end{defi}

Note that 
\[
\frac{\partial^{\alpha_1}}{\partial x^{\alpha_1}_{1}}\cdots\frac{\partial^{\alpha_i-1}}{\partial x^{\alpha_i-1}_{i}}\cdots\frac{\partial^{\alpha_j-1}}{\partial x^{\alpha_j-1}_{j}}\cdots \frac{\partial^{\alpha_n}}{\partial x^{\alpha_n}_{n}} \ \left(\frac{x^{\alpha}}{\alpha!} \right) = x_i x_j.
\]
Hence, the denormalized Lorentzian property can be reformulated as follows. Let $\beta=(\beta_1,\dots,\beta_n) \in \mathbb{N}^n$ such that $\beta_1+\dots+\beta_n=d-2$. For $f=\sum\limits_{\alpha}a_{\alpha}x^{\alpha} \in H^d_n$, define $f_{\beta} := \sum\limits_{\alpha \geq \beta}a_{\alpha}x^{\alpha}$. Then 
\[
\frac{\partial^{\beta}}{\partial x^{\beta}}N(f) = x^{-\beta}f_{\beta} = \sum\limits_{\alpha \geq \beta}a_{\alpha}x^{\alpha-\beta}.
\]
So $f$ is denormalized Lorentzian if and only if $x^{-\beta}f_{\beta} \in L_n^{2}$ for any $\beta = (\beta_1,\dots,\beta_n) \in \mathbb{N}^n$ such that $\beta_1 + \dots + \beta_n = d-2$. 
This means that $f$ has nonnegative coefficients and M-convex support, and the Hessian of $x^{-\beta}f_{\beta}$ has at most one positive eigenvalue for any $\beta$ as above.

Now let us consider a polynomial of the form
\[
f=\sum\limits_{\substack{0\leq i \leq P-1 \\ 0 \leq j \leq Q-1}} A_{i,j}\ x_1^{P-1-i}x_2^{Q-j}x_3^{i}x_4^{j}.
\]
The possible choices for $\beta$ for which $x^{-\beta} f_\beta$ 
is non-zero and quadratic fall into the following three cases.

\textbf{Case I.} We have
\[
\beta=(P-2-i, Q-1-j,i,j)
\]
for $i \in \{0,\dots,P-2\}$ and $j \in \{0,\dots,Q-2\}$. Then
\[
x^{-\beta}f_{\beta} = A_{i,j}x_1x_2+ A_{i+1,j}x_2x_3+A_{i,j+1}x_1x_4+A_{i+1,j+1}x_3x_4,
\]
and the Hessian is
\begin{equation}\label{Hessian1}
\begin{bmatrix}
0 & A_{i,j} & 0 & A_{i,j+1}\\
A_{i,j} & 0 & A_{i+1,j} & 0\\
0 & A_{i+1,j} & 0 & A_{i+1,j+1}\\
A_{i,j+1} & 0 & A_{i+1,j+1} & 0\\
\end{bmatrix}.
\end{equation} 

\textbf{Case II.} We have 
\[
\beta=(0, Q-2-j,P-1,j)
\]
for $j \in \{0,\dots,Q-2\}$. Then
\[
x^{-\beta}f_{\beta} = A_{P-1,j} x_2^2 + A_{P-1,j+1} x_2x_4 + A_{P-1,j+2} x_4^2,
\]
and the Hessian is
\begin{equation}\label{Hessian2} 
\begin{bmatrix}
A_{P-1,j} & A_{P-1,j+1}\\
A_{P-1,j+1} & A_{P-1,j+2} \\
\end{bmatrix}.
\end{equation}

\textbf{Case III.} We have
\[
\beta=(P-3-i, 1,i,Q-1)
\]
for $i \in \{0,\dots,P-3\}$. Then 
\[
x^{-\beta}f_{\beta} = A_{i,Q-1} x_1^2 + A_{i+1,Q-1}x_1x_3 + A_{i+2,Q-1} x_3^2,
\]
and the Hessian is
\begin{equation}\label{Hessian3} 
\begin{bmatrix}
A_{i,Q-1} & A_{i+1,Q-1}\\
A_{i+1,Q-1} & A_{i+2,Q-1} \\
\end{bmatrix}.
\end{equation}
Based on this computation, we can derive a necessary and sufficient condition for $f$ to be denormalized Lorentzian.  In the following lemma, we find the signs of the eigenvalues of the matrices occurring as Hessians of $x^{-\beta}f_{\beta}$. The condition is then stated in Proposition~\ref{Lorisproduct}.

\begin{lemm}\label{rootsofmatrices}
For $a,b,c,d \in \mathbb{N}$, let
\[
A=
\begin{bmatrix}
0 & a & 0 & d\\
a & 0 & b & 0\\
0 & b & 0 & c\\
d & 0 & c & 0\\
\end{bmatrix}
\ \text{ and } \
B=\begin{bmatrix}
a & b\\
b & c \\
\end{bmatrix}.
\]
The matrix $B$ has at most one positive eigenvalue if and only if $b^2 \geq ac$. The matrix $A$ has at most one positive eigenvalue if and only if $ac=bd$.
\end{lemm}

\begin{proof}
    The characteristic polynomial of $B$ is $x^2-(a+c)x+(ac-b^2)$, so $a+c$ is the sum and $ac-b^2$ is the product of the roots. As $a+c \ge 0$, at least one of the roots is nonnegative. So at most one of the roots is positive if and only if $ac-b^2 \leq 0$.

    The characteristic polynomial of $A$ is 
    \[
    \chi_A=x^4-(a^2+b^2+c^2+d^2)x^2 + (a^2c^2+b^2d^2-2abcd).
    \]
    The discriminant of this polynomial with respect to $x^2$ is 
    \[
    \Delta_{\chi_A}= (a^2+b^2+c^2+d^2)^2-4a^2c^2 - 4b^2d^2 + 8abcd \geq 0.
    \]
    So there exist real numbers $r_1$ and $r_2$ such that 
    \[
    \chi_A=(x^2-r_1)(x^2-r_2),
    \]
    \[
    r_1 +r_2 = a^2 + b^2 +c^2 + d^2 \geq 0, \text{ and}
    \]
    \[
    r_1r_2 = a^2c^2+b^2d^2-2abcd \geq 0.
    \]
    Hence $r_1$, $r_2 \geq 0$. If both of them are positive, then $A$ has exactly two positive eigenvalues. If one of them is zero, then $0=r_1r_2=a^2c^2+b^2d^2-2abcd$, so $ac = bd$. The result follows.
\end{proof}

\begin{prop}\label{Lorisproduct}
    For a $P \times Q$ matrix $A=(A_{ij})$ with positive entries, the polynomial 
    \[
    f=\sum\limits_{\substack{0\leq i \leq P-1 \\ 0 \leq j \leq Q-1}} A_{i,j}\ x_1^{P-1-i}x_2^{Q-j}x_3^{i}x_4^{j}
    \]
    is denormalized Lorentzian if and only if there exist sequences $(z_0,\dots,z_{P-1})$ and $(y_0,\dots,y_{Q-1})$ with positive entries such that:
    \begin{enumerate}
    \item $A_{i,j} = z_iy_j$,
    \item $(z_i)_{0 \leq i \leq P-1}$ is log-concave; i.e., $z_i^2 \geq z_{i-1}z_{i+1}$ for $i \in \{1,\dots,P-2\}$, and
    \item $(y_j)_{0 \leq j \leq Q-1}$ is log-concave; i.e., $y_j^2 \geq y_{j-1}y_{j+1}$ for $j \in \{1,\dots,Q-2\}$.
    \end{enumerate}
\end{prop}

\begin{proof}
Note that since the matrix $A$ has positive entries, we have 
\[
\text{Supp}(f) = \{(P-1-i,Q-j,i,j) \in \mathbb{Z}^4 \,:\, 0\leq i \leq P-1 ,\ 0 \leq j \leq Q-1\}.
\]
Hence, $f$ has M-convex support. In fact, $\text{Supp}(f)$ is exactly the integer points inside a $P \times Q$ rectangle embedded in a $2$-dimensional hyperplane in $\mathbb{R}^4$. 

We now look at the Hessians. Using the criteria of Lemma~\ref{rootsofmatrices} on Hessians of case~I, which have form \eqref{Hessian1}, leads to 
\[
A_{i,j}A_{i+1,j+1}=A_{i+1,j}A_{i,j+1}
\]
for $i \in \{0,\dots, P-2\}$ and $j \in \{0, \dots, Q-2\}$.
This means that all the rows of $A$ are multiples of a vector $\mathbf{y} = (y_0,\dots,y_{Q-1})$, and all the columns of $A$ are multiples of a vector $\mathbf{z} = (z_0,\dots,z_{P-1})$. Hence $A_{i,j} = z_i y_j$, possibly after scaling $\mathbf{y}$ or $\mathbf{z}$.

Using the criteria of Lemma~\ref{rootsofmatrices} on Hessians of case~II, which have form~\eqref{Hessian2}, and case~III, which have form~\eqref{Hessian3}, leads to log-concavity of the last row and column of $A$. This implies that $\mathbf{y}$ and $\mathbf{z}$ are log-concave.
\end{proof}

The criteria of Proposition \ref{Lorisproduct} are too restrictive. For the refined Crowell polynomial of the closure of an alternating 3-braids $B = \sigma_1^{p_1}\sigma_2^{-q_1}\cdots\sigma_1^{p_n}\sigma_2^{-q_n}$ with coefficient matrix $A = [A_{i,j}]$, a calculation of extremal entries proves that these criteria only hold when $n=1$, which means that the closure of the alternating 3-braid is $T_{2,p_1} \# T_{2,-q_1}$. This is an immediate corollary of the following: 

\begin{prop}\label{extermalA}
    The extermal entries of the matrix $A = [A_{i,j}]$ of coefficients of the refined Crowell polynomial of the closure of an alternating 3-braid $B=\sigma_1^{p_1}\sigma_2^{-q_1}\cdots\sigma_1^{p_n}\sigma_2^{-q_n}$ satisfy the following: 
    \begin{itemize}
        \item $A_{i,0}=1$ for $i \in \{0,\dots,P-1\}$,
        \item $A_{P-1,j}=1$ for $j \in \{0,\dots,Q-1\}$, and
        \item $A_{0,j} \geq n$ for $j \in \{1,\dots, Q-1\}$.
    \end{itemize}
    In other words,
    \begin{equation*}
    A=
 \begin{bmatrix}
   1 & \geq n & \cdots & \geq n \\
   1 & . & \cdots & . \\
   \vdots  & \vdots  & \ddots & \vdots  \\
   1 & 1 & \cdots & 1
 \end{bmatrix}.
\end{equation*}
\end{prop}

\begin{proof}
One can characterize the oriented rooted spanning trees contributing to each term and then count them. We will show that the families $\mathcal{T}_{i,0}$ and $\mathcal{T}_{P-1,j}$ only contain one tree each; see Figure~\ref{extermaltrees1}.

A tree $T \in \mathcal{T}_{i,0}$ contains all the red edges and no green edge. Let $J=(c_{i},c_{i-1},\dots,c_{i-t})$ be a maximal path of blue edges in $T$. Then either $c_{i}=c_{1}$, or the parent of vertex $c_{i}$ in $T$ must be a vertex in $R$. This leads to a contradiction, as edges of $J$, red edges, and the black edge leading to $c_{i}$ form a cycle. Hence, all the blue edges form a path of length $P-1-i$ starting at $c_1$. The black edges of $T$ are uniquely determined, as they form a disjoint union of paths reaching the vertices in $L \setminus J$.

A tree $T \in \mathcal{T}_{P-1,j}$ contains all the black edges and no blue edge. Let $J$ be the maximal path in $T$ starting at $c_1$. This path must go outside $L$, hence it contains the red edge leading to $c'_1$. Deleting this part of the tree, one can think of $c'_1$ as the root of the remaining subtree, and we repeat the argument of the previous paragraph.

The trees in the family $\mathcal{T}_{0,j}$ contain all the blue edges and no black ones. So the restriction of $T$ to the cycle $L$ is just a path of length $P-1$. One can then pick one of the $n$ red edges connecting $L$ to $R$. Assume this edge leads to $c'_{i}$. Add an oriented path $J$ consisting of $j$ green edges starting at $c'_{i}$. Then one can build an element of $\mathcal{T}_{0,j}$ by adding a disjoint union of paths of red edges leading to vertices in $R \setminus J$. See Figure~\ref{extermaltrees2} for an example. 
\end{proof}

\begin{figure}[h]
\centering
\includegraphics[scale=0.4]{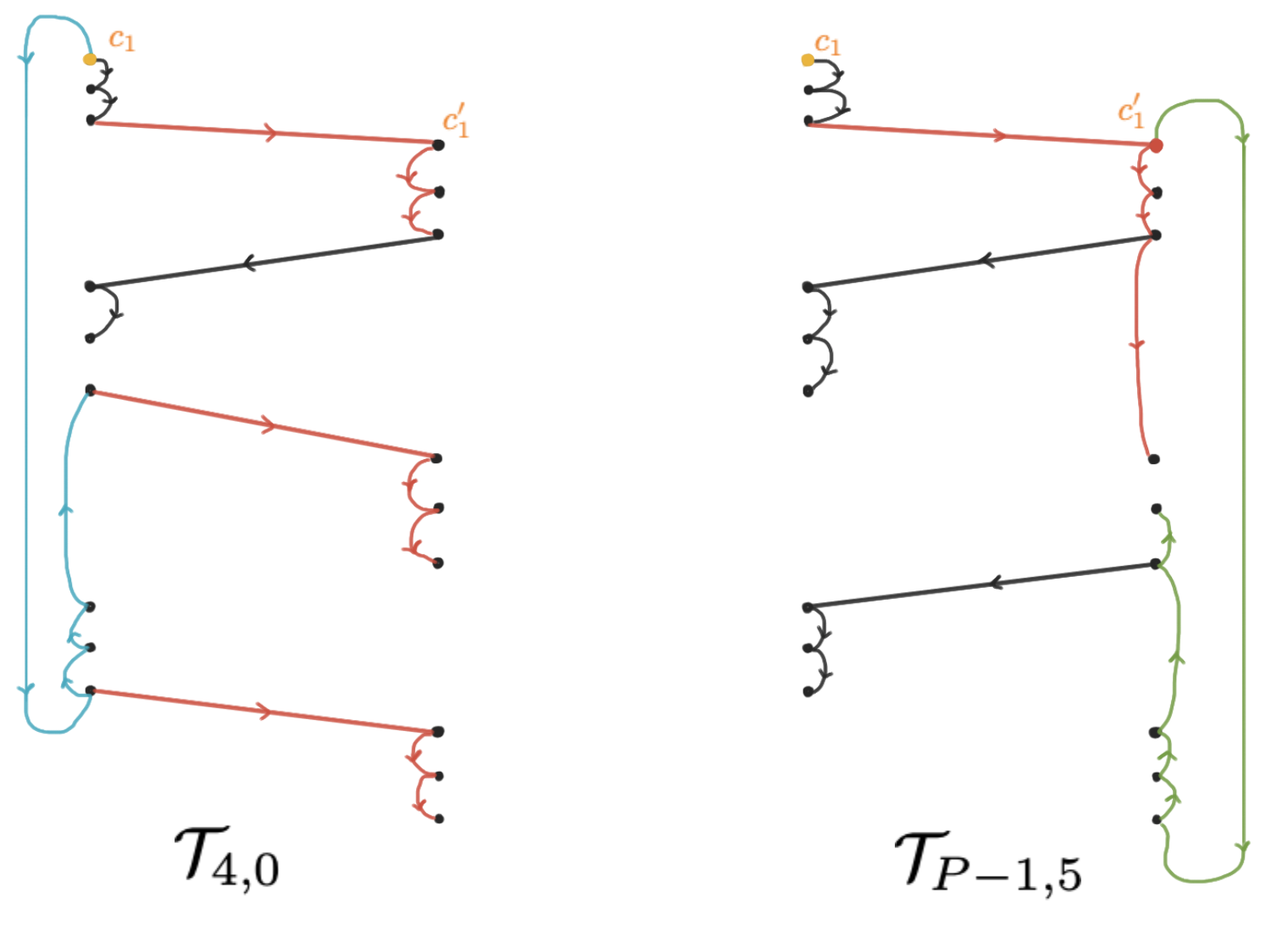}
\caption{The trees in $\mathcal{T}_{i,0}$ and $\mathcal{T}_{P-1,j}$.}\label{extermaltrees1}
\end{figure}

\begin{figure}[h]
\centering
\includegraphics[scale=0.4]{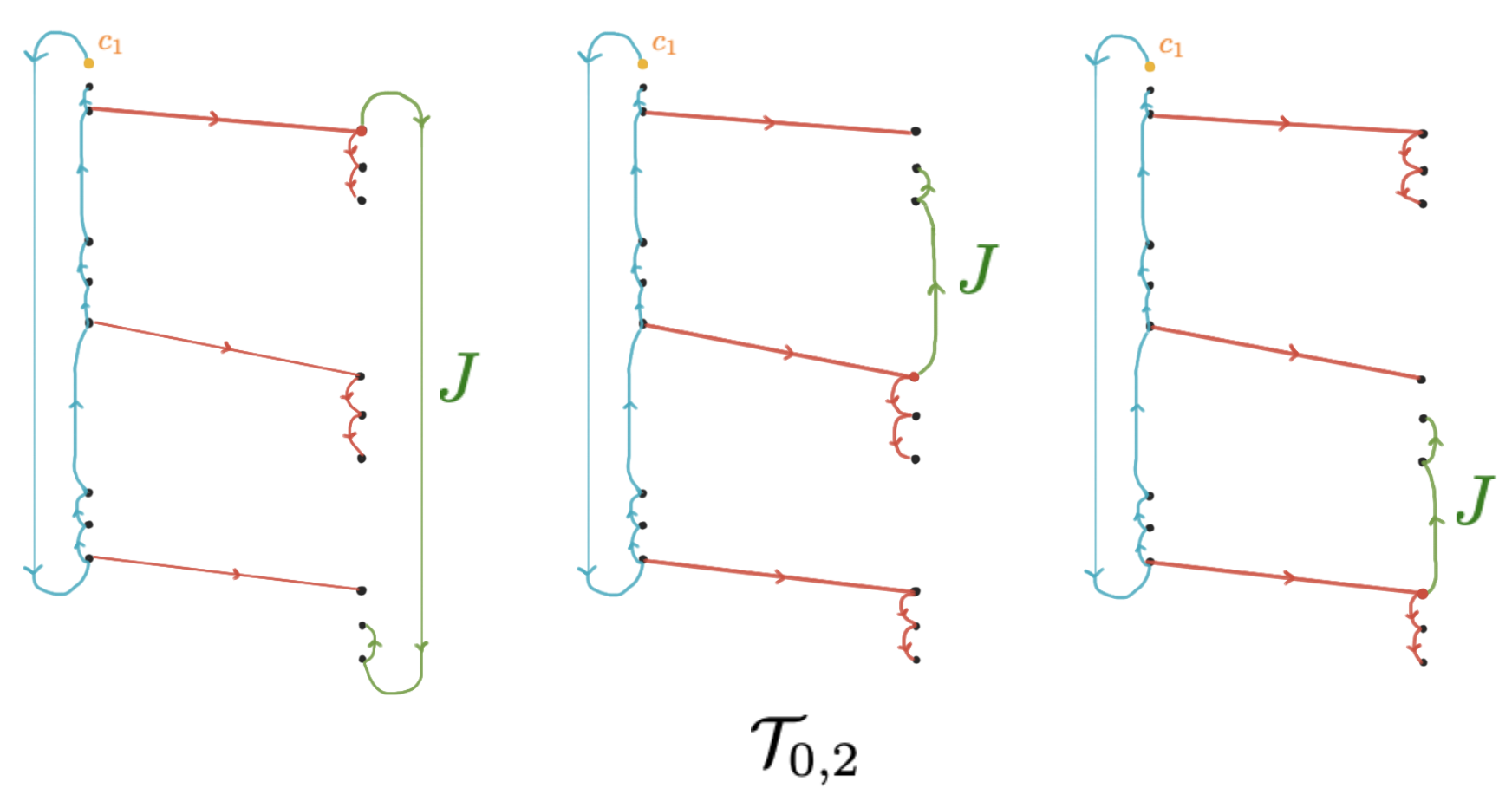}
\caption{This shows $n$ trees in $\mathcal{T}_{0,j}$.}\label{extermaltrees2}
\end{figure}

We conclude this subsection with some final remarks. As we discussed, a direct generalization of the multi-variable refinement idea does not work for 3-braids. One can try other weight refinements of the Crowell graph, but even for simple alternating 3-braids the support of the resulting polynomial fails to be M-convex. Another idea is to try to decompose this multi-variable polynomial into denormalized Lorentzian parts; i.e., decompose the $P \times Q$ matrix $A$ into matrices of the form of Proposition~\ref{Lorisproduct}. 
Finding a general scheme for such a decomposition seems improbable based on small examples. Equations~\ref{Coeffmatrix1} and~\ref{Coeffmatric2} show a few examples.

\begin{equ}[!ht]
    \begin{equation}\label{Coeffmatrix1}
    B= \sigma_1 \sigma_2^{-1} \sigma_1 \sigma_2^{-1} \sigma_1 \sigma_2^{-1},\ \ \ A=
 \begin{bmatrix}
   1 & 3 & 3 \\
   1 & 2 & 3 \\
   1 & 1 & 1  \\
 \end{bmatrix}
\end{equation}
\begin{equation}\label{Coeffmatric2}
    B= \sigma_1 ^{2} \sigma_2^{-2} \sigma_1^{2} \sigma_2^{-2} \sigma_1^{3} \sigma_2^{-1},\ \ \ A=
 \begin{bmatrix}
 1 & 3 & 5 & 5 & 3 \\
 1 & 5 & 10 & 11 & 6\\
 1 & 5 & 11 & 14 & 8 \\
 1 & 4 & 9 & 13 & 8 \\
 1 & 3 & 6 & 9 & 6 \\
 1 & 2 & 3 & 4 & 3 \\
 1 & 1 & 1 & 1 & 1 \\
 \end{bmatrix}.
\end{equation}
\caption*{Small examples of the refined Crowell polynomial.}
\end{equ}

In special cases, one can find certain log-concavity properties of the coefficient matrix $A$. As mentioned before, for alternating knots, being special is equivalent to one of the Tait graphs being bipartite. For alternating links, this is not true. One can change the orientation of some of the components of the link without changing the Tait graph. We call alternating links with a bipartite Tait graph \textit{almost special alternating links}. The Crowell graph of an almost special alternating link only differs from the Crowell graph of the corresponding special alternating link in some of its weights. Alternating $3$-braids $\sigma_1^{p_1} \sigma_2^{-q_1} \cdots \sigma_1^{p_n} \sigma_2^{-q_n}$, where all $p_i$ (or $q_i$) are even for $i \in \{1,\dots,n\}$, are examples of such braids. See Figure~\ref{almostspecial} when $B = \sigma_1^{-2}\sigma_2\sigma_1^{-2}\sigma_2$.

\begin{figure}[h]
\centering
\includegraphics[scale=0.3]{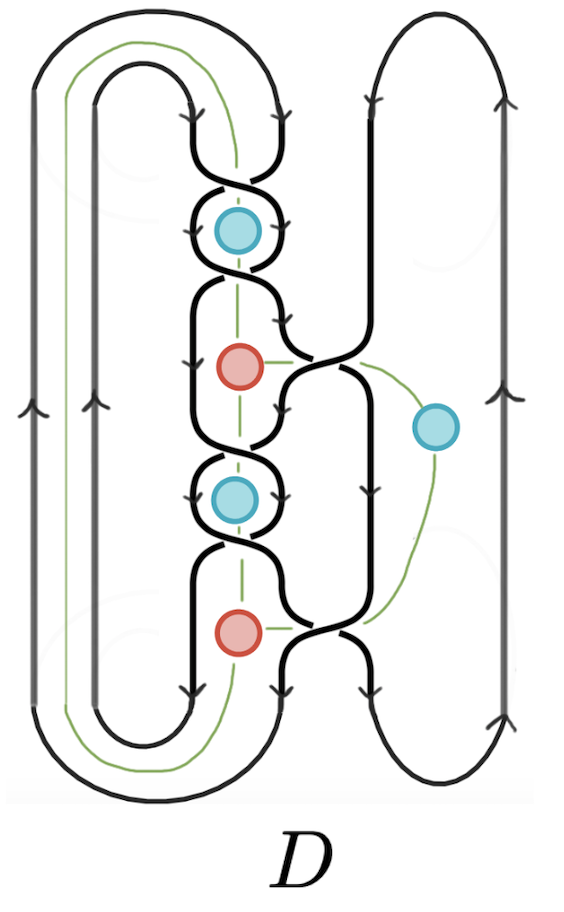}
\includegraphics[scale=0.3]{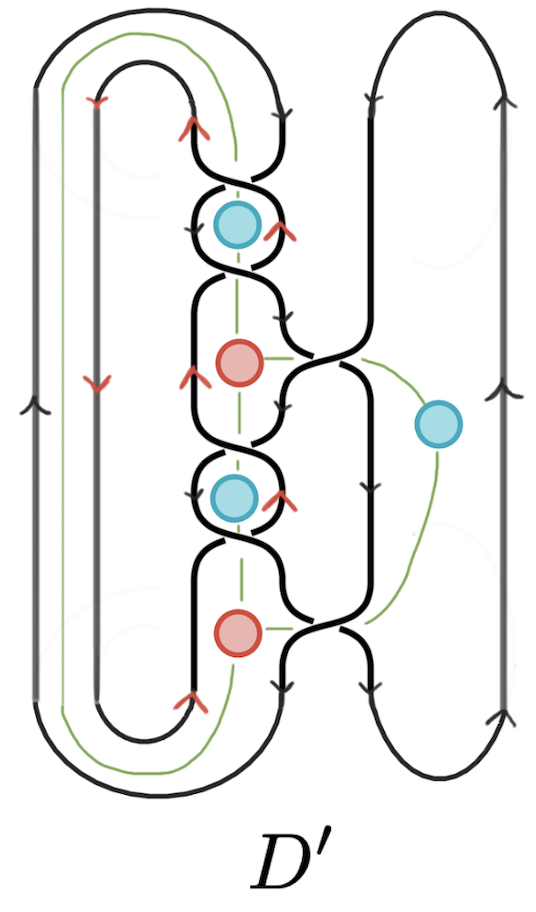}
\includegraphics[scale=0.3]{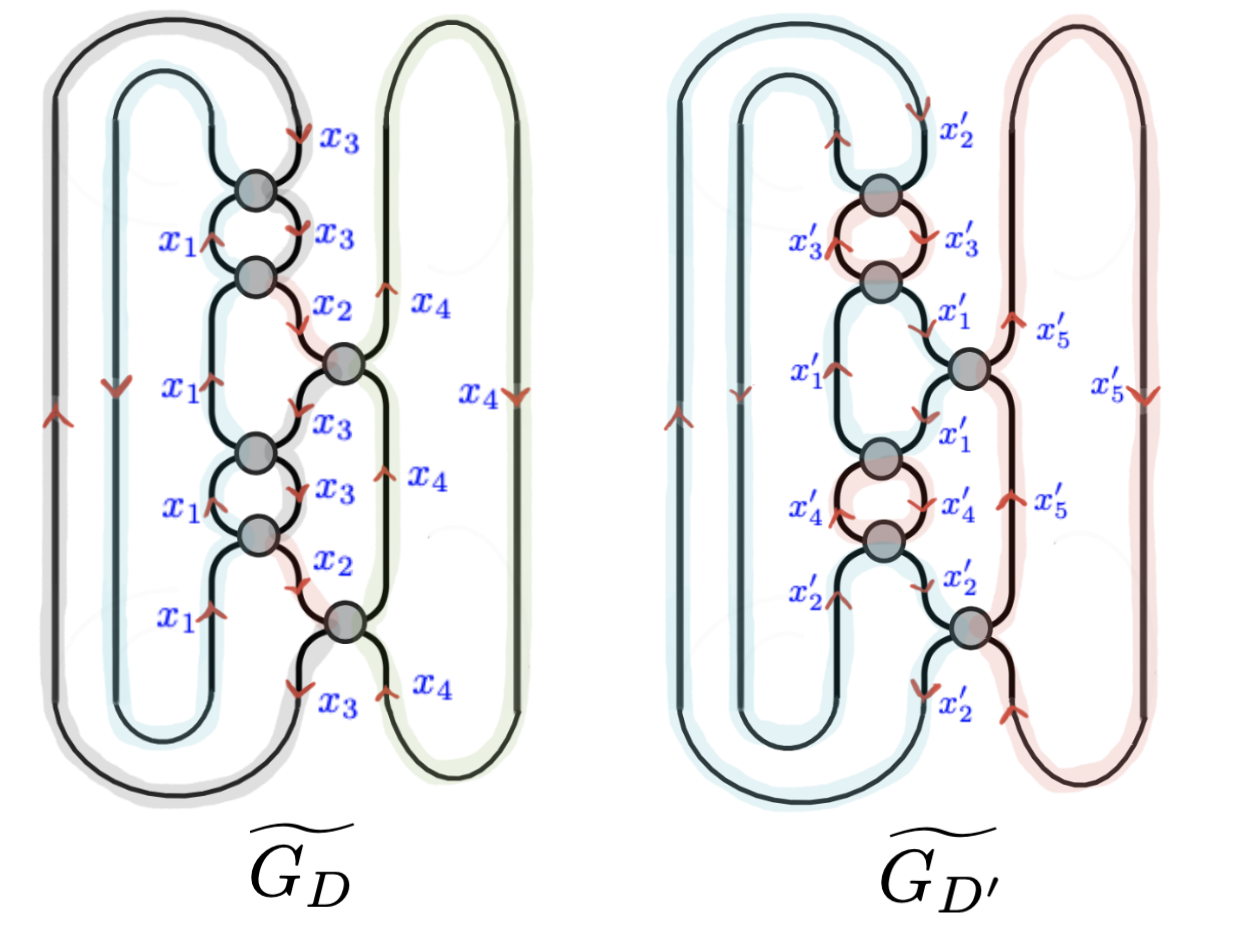}
\caption{The diagram $D$ is almost special alternating. The special alternating diagram $D'$ is obtained by changing the orientation of one of the components. The two figures on the right compare their Crowell graphs.}\label{almostspecial}
\end{figure}

Looking at the example of Figure~\ref{almostspecial}, one notices that, by identifying some of the variables, the refined Crowell polynomials $P_{\widetilde{G}_D}$ and $P_{\widetilde{G}_{D'}}$ become equal. More precisely, the set of edges with weight $x_1$, $x_2$, $x_3$ in $\widetilde{G}_D$ coincide with edges with weight $x'_1,\dots,x'_4$ in $\widetilde{G}_{D'}$. Hence, we have 
\[
P_{\widetilde{G}_D}(X,X,X,Y) = P_{\widetilde{G}_{D'}}(X,X,X,X,Y).
\]
Based on the work of Hafner, Mészáros, and Vidinas~\cite{KarolaLogconcavityOT}, the polynomial $P_{\widetilde{G}_{D'}}$ is denormalized Lorentzian. Identifying variables preserves the Lorentzian property by \cite[Lemma 2.5]{KarolaLogconcavityOT}, so $P_{\widetilde{G}_D}(X,X,X,Y)$ is also Lorentzian. Based on what we know about the general form of $P_{\widetilde{G}_D}$ (equation~\eqref{formof4-varmatrix}), we can deduce that the sequence of non-zero coefficients of $P_{\widetilde{G}_D}(X,X,X,Y)$ is $\sum\limits_{j}  A_{ij}$ for $j \in \{0,\dots,P-1\}$; i.e., the sum of the rows of the matrix $A$. We can conclude that this sequence is log-concave: 

\begin{theo}\label{evenalternating}
Let $D$ be the diagram of the closure of an alternating 3-braid word 
\[
B=\sigma_1^{p_1}\sigma_2^{-q_1}\cdots\sigma_1^{p_n}\sigma_2^{-q_n}. 
\]
Let $A=(A_{i,j})$ be the $P \times Q$ matrix of coefficients of $P_{\widetilde{G}_D,c_1}$. If all $p_i$ are even, then the sum of the rows of $A$ form a log-concave sequence. If all $q_i$ are even, then the sum of the columns of $A$ form a log-concave sequence. 
\end{theo}

\section{A conjecture of Hirasawa and Murasugi}\label{Section2}
\subsection{Introduction}\label{HMintro}

In 2013, Hirasawa and Murasugi proposed an extension of the trapezoidal conjecture. They conjectured that the signature gives an upper bound on the length of the stable part of the sequence of absolute values of coefficients of the Alexander polynomial of an alternating knot. 
In 2017, Chen~\cite{Chen2017OnTA} generalized their conjecture to alternating links. 

\HMrestatable*

We call this the \textit{H-M inequality}, and the alternating links for which equality holds \textit{H-M sharp}. We define the \textit{stable length of $L$} to be 
\[
\sl(L) := l-2(i_0-1).
\]

This conjecture originates from the study of the roots of the Alexander polynomial of alternating links, which has been the subject of numerous papers, most notably regarding questions related to Hoste's conjecture~\cite{Lyubich-Murasugi}. Hirasawa and Murasugi~\cite{hirasawa2013various} introduced the class of \textit{stable alternating links}, which consists of all alternating links $L$ for which $\Delta_L$ has only real roots. Lemma~\ref{HMstable} implies that Conjecture~\ref{HMconjecture} holds for this class of links.

\begin{lemm}[\cite{hirasawa2013various}]\label{HMstable}
    Let $L$ be a stable alternating link and 
    \[
    \Delta_L(t) = \sum\limits_{i=0}^{l-1}(-1)^i a_i t^{i}. 
    \]
    Then $(a_0, a_1, \cdots, a_{l-1})$ is a strictly log-concave sequence, and hence it is unimodal; i.e., it is trapezoidal with stable length one. Furthermore, the signature of $L$ is zero.
\end{lemm}

Chen~\cite{Chen2017OnTA} showed that Conjecture~\ref{HMconjecture} holds for 2-bridge links, and Alrefai and Chbili~\cite{small3braids} proved it for closures of alternating 3-braids of length less than four. There also has been some effort to construct families of stable alternating links by plumbing Hopf bands to a disk; see Hirasawa and Murasugi~\cite{hirasawa2013various} and Stoimenow~\cite{StoimenowHopfbands}.

Note that the H-M inequality is trivial for special alternating links. This comes from the combinatorial formula for the signature of alternating links stated in Lemma~\ref{signatureformula}, below. Using Lemma~\ref{signatureformula}, we will establish Lemma~\ref{specialaltsignature} about the signature of special alternating links. Recall that we only consider positive special diagrams, unless otherwise stated.

\begin{lemm}[\cite{Traczyk2004ACF},\cite{OSineq}]\label{signatureformula}
    Let $D$ be a reduced alternating diagram, $c$ a crossing in $D$, and $D_0$ the diagram obtained by taking the oriented smoothing of $c$. Then 
    \[
    \sigma(D_0) = \sigma(D) + \text{sign}(c).
    \]
    Hence 
    \[
    \sigma(D) = \#(\text{black regions}) - \#(\text{positive crossings}) - 1,
    \]
    where the black-white coloring convention for alternating links is as in Figure~\ref{OSmu-1}. 
\end{lemm}

\begin{figure}[h]
    \centering
    \includegraphics[scale=0.3]{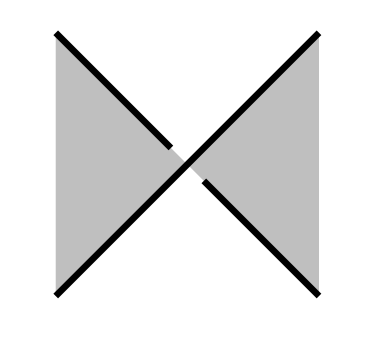}
    \caption{Black-white coloring convention for regions of an alternating diagram, following Ozsv\'ath and Szab\'o~\cite{OSineq}.}
    \label{OSmu-1}
\end{figure}

\begin{lemm}\label{specialaltsignature}
For a special alternating link $L$, we have 
\[
-\sigma(L) = 2g(L) = \deg(\Delta_L).
\]
Furthermore, if an alternating link $L$ satisfies this condition, then it is special alternating (up to taking its mirror). 
\end{lemm}

\begin{proof}
    Suppose that $L$ is a positive special alternating link. Using the second formula of Lemma~\ref{signatureformula}, we can compute the signature of $L$ based on its black Tait graph $B$. The black regions correspond to the vertices of $B$, and the crossings correspond to the edges. Hence, we have 
    \[
    \sigma(L) = |V(B)| - |E(B)| - 1 = - \text{rank}(H_1(B))= -2g(L).
    \]
    The last equality holds because the black Tait surface, constructed by median construction on $B$, is a minimal genus Seifert surface.
    
    To prove the converse, we use Proposition~\ref{Murasugidecomp} to decompose $L$ into the diagrammatic Murasugi sum of special alternating components such that genus and signature are additive under this decomposition. The absolute value of the signature is sub-additive under Murasugi decomposition, so the equality $2g(\Delta_{L}) = |\sigma(L)|$ holds only when the signatures of all the special alternating components of $L$ have the same sign. Hence, all of these components must be positive (resp.\ negative) alternating links, which means that $L$ is a positive (resp.\ negative) alternating link. Consequently, $L$ or its mirror is a special alternating link.
\end{proof}

In the rest of this subsection, we prove some new results about the H-M inequality based on our work in Section~\ref{Section1}. In subsection~\ref{HMsharp}, we focus on the study of H-M sharp knots. In Subsection~\ref{dual}, we connect the H-M inequality with concordance.
We are going to use the results of Subsection~\ref{stablisesection}, so we now summarize these.

\begin{prop}\label{summaryofsection1.2}
Let $n$ be the number of crossings in the largest coherent twist region of $L$, denoted by $R$. Let $L_i$ be the alternating link obtained by changing $R$ to a twist region containing $i$ crossings. The polynomials $P_1$ and $P'_2$ are Crowell polynomials of the graphs depicted in Figure~\ref{Iteratedformulapics}.

\begin{figure}[h]
    \centering
    \includegraphics[scale=0.3]{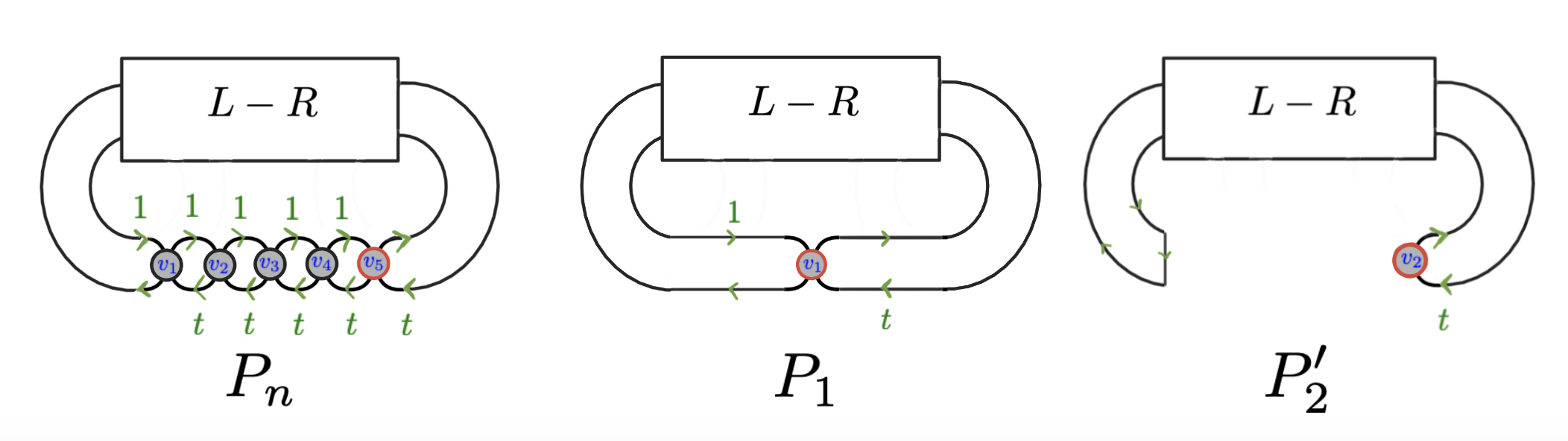}
    \caption{The Crowell graphs of equations~\eqref{summaryof1.2.eq1} and \eqref{summaryof1.2.eq2}.}
    \label{Iteratedformulapics}
\end{figure}

If $n$ is odd, then 
\begin{equation}\label{summaryof1.2.eq1}
\Delta_{L}(-t) = P_n = t^{n-1}P_1+(1+\cdots+t^{n-2})P'_2.    
\end{equation}
Furthermore, the first $n-2$ terms of $P_n$ come from only $(1+\cdots+t^{n-2})P'_2$.

Similarly, if $n$ is even, then
\begin{equation}\label{summaryof1.2.eq2}
\Delta_{L}(-t)=P_n=t^{n-2}P_2+(1+\cdots+t^{n-3})P'_2,   
\end{equation}
and the first $n-3$ terms of $P_n$ come from only $(1+\cdots+t^{n-3})P'_2$. 
\end{prop}

Using Proposition~\ref{summaryofsection1.2}, we now establish the first result of this subsection.

\begin{theo}\label{stablelengthofhightwist}
Let $L$ be an alternating link and 
\[
\Delta_L(T) \dot{=} \sum\limits_{i=0}^{2g(L)+|L|-1} (-1)^i a_i T^i. 
\]
If $\text{MT}(L)-3 \geq g(L) + |L|/2$, then $(a_i)$ is trapezoidal and 
\[
\text{sl}(L) = 2g(L) + |L| - 2(\deg(P'_2)-1),
\]
where $P'_2$ is the Crowell polynomial of the graph depicted in Figure~\ref{Iteratedformulapics}.
\end{theo}

\begin{proof}
Let $n = \text{MT}(L)$, and assume that $n$ is odd. The trapezoidality of the sequence $(a_i)$ follows from Theorem~\ref{twistconcentrated}.
Let $c_0, \dots, c_{n-2 + \deg(P'_2)}$ denote the coefficients of $(1+\cdots+t^{n-2})P'_2$. Then $(c_i)_{1 \leq i \leq n-2}$ are cumulative sums of coefficients of $P'_2$.
Hence, if $n-2 > \deg(P'_2)$, then
\[
c_{\deg(P'_2)} = \cdots = c_{n-2},
\]
and they are all equal to the sum of the coefficients of $P'_2$.
Using Proposition~\ref{summaryofsection1.2}, we know that, for $i \in \{0,\dots,n-2\}$, we have $a_i = c_i$. The symmetry of the sequence $(a_i)$ implies the result. 
The argument is analogous for $n$ even.
\end{proof}

In Subsection~\ref{stablisesection}, we used the shape of twist regions in the Seifert graph to prove that $g(L_n) = g(L_1) + \frac{n-1}{2}$; see Figure~\ref{seifertcycs}. Using this, one can rewrite the conclusion of Theorem~\ref{stablelengthofhightwist} as
\begin{equation}\label{rephrasestablelenofhightwist}
\text{sl}(L_n) = 2(g(L_1)) + n - 1 + |L_n| - 2(\deg(P'_2)-1).
\end{equation}
We can use this reformulation and prove Corollary~\ref{ReductionoftwistsinHM}. Recall that we call an alternating link satisfying $\text{MT}(L)-3 > g(L) + |L|/2$ a \textit{twist-concentrated alternating link}. 

\begin{coro}\label{ReductionoftwistsinHM}
Using the notation of Proposition~\ref{summaryofsection1.2}, if $L_n$ is twist-con\-cen\-trat\-ed, then the H-M inequality for $L_{n}$ follows from the H-M inequality for $L_{n-2}$. 
\end{coro}

\begin{proof}
Based on equation~\eqref{rephrasestablelenofhightwist}, we have 
\[
\text{sl}(L_n) = \text{sl}(L_{n-2})+2.
\]
Without loss of generality, assume that the twisting region consists of negative crossings (as in Figure~\ref{twistregion}). Using Lemma~\ref{signatureformula}, we have 
\[
\sigma(L_n) = \sigma(L_{n-2})+2.
\]
The combination of these two equalities gives us the conclusion.
\end{proof}

Note that, using equation~\eqref{rephrasestablelenofhightwist}, one only needs to compute $\deg(P'_2)$ to obtain an explicit formula for the stable length of alternating twist-concentrated links. One example is twist-concentrated alternating 3-braids, for which the computation is done in Proposition~\ref{degP'2_3braids}. 

\begin{prop}\label{degP'2_3braids}
 Let $B = \sigma_1^{p_1}\sigma_2^{-q_1}\cdots\sigma_1^{p_n}\sigma_2^{-q_n}$, and let $L$ be the corresponding alternating 3-braid. Assume that $P'_2$ is defined as in Proposition~\ref{summaryofsection1.2}. Then we have
\[
\deg(P'_2) = \deg(P_1)-1 = 2g(L_1) + |L_1| - 2.
\]   
\end{prop}

\begin{proof}
Twist regions of $L$ are in one-to-one correspondence with powers of $\sigma_1$ and $\sigma_2$ in $B$. Without loss of generality, assume that 
\[
p_n=\max \{p_1,\cdots,p_n,q_1,\cdots,q_n\}. 
\]
We use the same notations as in Subsection~\ref{stablisesection}. Let $N = |V(G_1)| = |V(G'_2)|$.

We identify rooted spanning trees with the highest weight in $G_1$ and $G'_2$. First, note that edges with weight $t$, excluding the one that leads to the root $v$, form a spanning tree $T \in \mathcal{T}(G_1,v)$. As a result, $\deg(P_1) = N-1$. A similar argument gives that $\mindeg(P_1)=1$, hence 
\[
\deg(P_1) = \deg(\Delta_{L_1})=2g(L_1)+|L_1|-2. 
\]
Note  that we showed the existence of these two trees in Lemma~\ref{Supportof3braid} as well. 

Now we compute $\deg(P'_2)$. We are going to show that there is a spanning tree $T' \in \mathcal{T}(G'_2,v)$ with only one edge with weight $1$. Let $e_{ou}=(v, v_{ou})$ (resp.~$e_{od}=(v,v_{od})$) be the upper (resp.~lower) outgoing edge of the root $v$ in $G_1$. Also let $e_{iu}=(v_{iu},v)$ be the upper incoming edge of the root $v$ in $G_1$. To construct $G'_2$, we deltet the edges $e_{ou}$ and $e_{iu}$ and replace them with an edge $e_{io}=(v_{iu},v_{ou})$. We are going to construct $T'$ from $T$ using the following modifications. We delete the edge leading to $v_{iu}$ in $T$ and replace it with $(v_{od},v_{iu})$, and then replace $e_{ou}$ with $e_{io}$. Note that, following this construction, $(v_{od},v_{iu})$ will be the only edge with weight $1$ in $T'$. See Figure~\ref{highdegreetrees} for an illustration of the proof. 
\end{proof}

Using Proposition~\ref{degP'2_3braids} and Lemma~\ref{sign_3braids}, we can deduce the H-M inequality as stated in Theorem~\ref{HM3-braids}, below.

\begin{lemm}(\cite{small3braids},\cite{Erle})\label{sign_3braids}
    Let $B = \sigma_1^{p_1}\sigma_2^{-q_1}\cdots\sigma_1^{p_n}\sigma_2^{-q_n}$ be an alternating 3-braid, and let $L_B$ be its closure.
    Then 
    \[
    \sigma(L_B) = \sum p_i - \sum q_j.
    \]
\end{lemm}

\begin{theo}\label{HM3-braids}
     Let $B=\sigma_1^{p_1}\sigma_2^{-q_1}\cdots\sigma_1^{p_n}\sigma_2^{-q_n}$ be an alternating 3-braid, and let $L_B$ by its closure. Assume that 
     \[
     p_1 \geq \sum\limits_{i\geq 2} p_i + \sum q_j + 6.
     \]
     Then the Hirawasa--Murasugi conjecture holds for $L_B$. 
\end{theo}
 
\begin{figure}[h]
    \centering
    \includegraphics[scale=0.25]{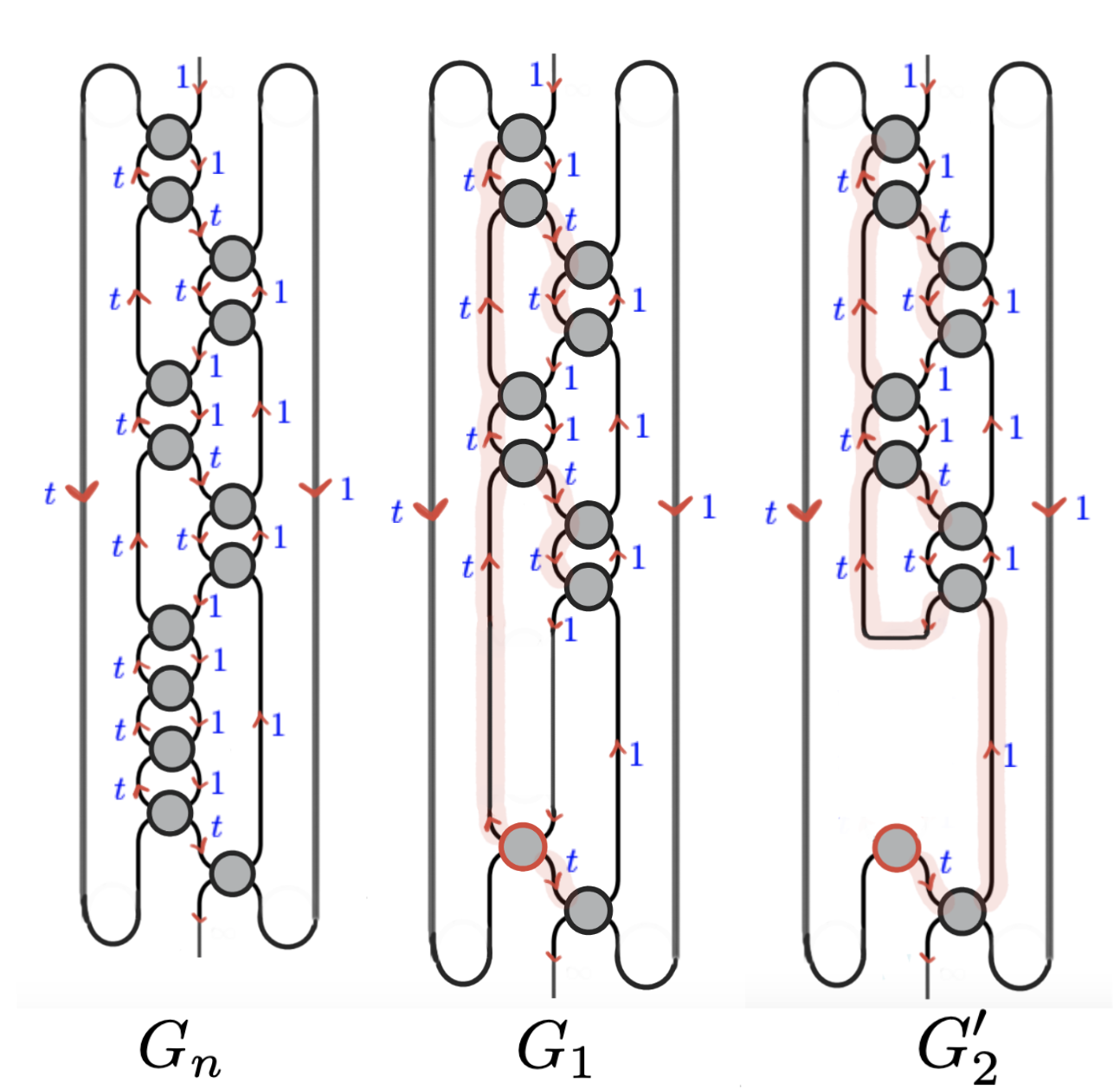}
    \caption{The rooted spanning trees with the highest degree weights in $P_1$ and $P'_2$ for an alternating 3-braid.}
    \label{highdegreetrees}
\end{figure}

If Fox's trapezoidal conjecture is true, one can replace the assumption that $\text{MT}(L) - 3 \geq g(L) + |L|/2$ in Theorem~\ref{stablelengthofhightwist} with 
\[
\text{MT}(L) - 2 > \deg(P'_2).
\]

We now use the results of Subsection~\ref{plumbingofspecial} to study the stable length and the Hirasawa--Murasugi inequality. 

\begin{lemm}\label{stablelenproductoftrapezoidals}
    Let $K_1$ and $K_2$ be alternating links such that the trapezoidal conjecture holds for both. Assume that $\deg(\Delta_{K_1}) \geq \deg(\Delta_{K_2})$. Then
    \[
    \text{sl}(K_1 \# K_2) = \min \{\text{sl}(K_1) - \deg(\Delta_{K_2}), 0\}.
    \]
    In particular, if $\text{sl}(K_1 \# K_2) > 0$, then $\text{sl}(K_1) = \text{sl}(K_1 \# K_2) + \deg(\Delta_{K_2})$. 
\end{lemm}

\begin{proof}
The sequence of coefficients of $\Delta_{K_1 \# K_2} = \Delta(K_1) \Delta(K_2)$ is the convolution of the coefficient sequences of $\Delta(K_1)$ and $\Delta(K_2)$. This convolution alongside a sketch of the proof appears in Figure~\ref{convolutionoftrapezoidal}.
 
The trapezoidal inequalities between coefficients become equality only when the shorter sequence falls in the stable part of the longer sequence (blue in Figure~\ref{convolutionoftrapezoidal}). As soon as the shorter sequence moves out of this stable part, the associated coefficient drops (red in Figure~\ref{convolutionoftrapezoidal}).
 
One can derive this explicitly by expanding the terms and using the trapezoidal condition, but we do not include the computation here, which is a generalization of the rearrangement inequality.
\end{proof}

\begin{figure}[h]
    \centering
    \includegraphics[scale=0.3]{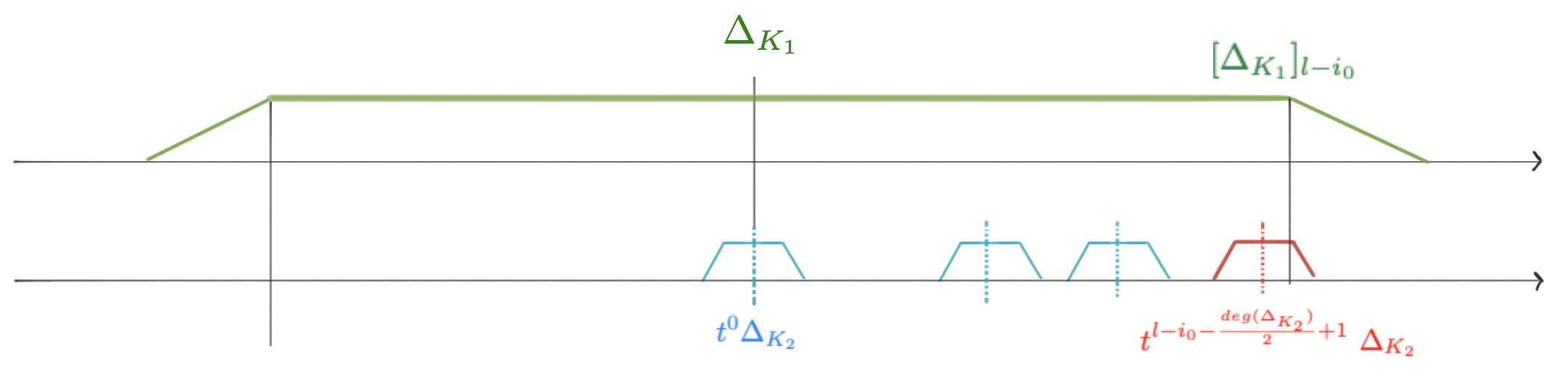}
    \caption{Convolution of two trapezoidal sequences. The top figure is a graph of the sequence of coefficients of $\Delta_{K_1}$. The bottom figure shows different copies of the sequence of coefficients of $\Delta_{K_2}$ used for computing the convolution at different points.}
    \label{convolutionoftrapezoidal}
\end{figure}

A simple corollary of Lemma~\ref{stablelenproductoftrapezoidals} is that one only needs to prove the H-M inequality for prime alternating links.

\begin{coro}\label{HMnonprime}
If the H-M inequality holds for prime alternating links, then it also holds for arbitrary alternating links.  
\end{coro}

\begin{proof}
    Assume that the H-M inequality fails for $K_1 \# K_2$. This means that
    \[
    \text{sl}(K_1 \# K_2) > |\sigma(K_1 \# K_2)| + 1 = |\sigma(K_1) + \sigma(K_2)| + 1.
    \]
    Using Lemma~\ref{stablelenproductoftrapezoidals}, we have
    \[
    \text{sl}(K_1) > |\sigma(K_1) + \sigma(K_2)| + \deg(\Delta_{K_2})+1 \geq |\sigma(K_1)| - |\sigma(K_2)| + \deg(\Delta_{K_2})+1,
    \]
    so $\text{sl}(K_1) > |\sigma(K_1)| + 1$. Hence, if there exists a non-prime counterexample to the H-M inequality, then there exists a prime counterexample. 
\end{proof}

One can extend this result to diagrammatic plumbings, and, even more generally, to Murasugi sums of length less than $3$:

\begin{coro}\label{H-Mplummbing}
    The H-M inequality holds for any link $L = L_1 * \cdots * L_n$ which is a diagrammatic Murasugi sum of special alternating links $L_1, \dots, L_n$, such that the length of all Murasugi sums are less than $3$. 
\end{coro}

\begin{proof}
Let $L'_1 = L_1$ and $L'_2 = L_2 * \cdots * L_n$. Then $L = L'_1 * L'_2$, and the length of this Murasugi sum is less than 3. By Theorem~\ref{Trapezoidalsumovertree}, the trapezoidal conjecture holds for $L$, $L'_1$, and $L'_2$.

By Proposition~\ref{Alexanderformulaplumbing}, for a plumbing, we can write
\begin{equation}\label{repeatAlexanderplumbformula}
\tilde{\Delta}_{L'_1*L'_2} = \tilde{\Delta}_{L'_1} \tilde{\Delta}_{L'_2} + \tilde{\Delta}_{\widetilde{L}'_1}  \tilde{\Delta}_{\widetilde{L}'_2}.
\end{equation}
Since the Alexander polynomials are trapezoidal, one can deduce that 
\[
\text{sl}(L'_1*L'_2) = \min \left\{\text{sl}(\Delta_{L'_1} \times \Delta_{L'_2}), \text{sl}\left(\Delta_{\widetilde{L}_1'} \Delta_{\widetilde{L}_2'} \right) \right\}.
\]
Recall that signature is additive under diagrammatic Murasugi sum (Proposition~\ref{Murasugidecomp}). The result now follows analogously to the proof of Corollary~\ref{HMnonprime}.
The argument for other Murasugi sums of length less than 3 is similar.
\end{proof}

\subsection{H-M sharp links}\label{HMsharp}

In this subsection, we study the family of H-M sharp alternating links with non-zero stable lengths. Recall that we call an alternating link $L$ H-M sharp when
\[
\left\lfloor \frac{|\sigma(L)|+1}{2} \right\rfloor = \left\lfloor \frac{\text{sl}(L)}{2} \right\rfloor.
\]
We start by looking at the results of Subsection~\ref{HMintro} and analyze when equality holds. 

\begin{lemm}\label{HMsharpcomposites}
Let $K_1$ and $K_2$ be alternating knots satisfying the Hirawasa--Murasugi conjecture, and assume $\deg(\Delta_{K_1}) \geq \deg(\Delta_{K_2})$. The connected sum $K_1 \# K_2$ is H-M sharp if and only if, up to mirroring,
\begin{enumerate}
    \item $\sl(K_1) = \sigma(K_1) + 1$,
    \item $\sigma(K_2) + \deg(\Delta_{K_2}) = 0$, and
    \item $\sl(K_1) \geq \deg(\Delta_{K_2})$.
\end{enumerate}
Furthermore, $K_1$ is H-M sharp with positive signature and $K_2$ is special alternating. 
\end{lemm}

\begin{proof}
Using Lemma~\ref{stablelenproductoftrapezoidals}, we have
\[
\text{sl}(K_1 \# K_2) = |\sigma(K_1) + \sigma(K_2)| + 1.
\]
Hence
\[
\text{sl}(K_1) \geq |\sigma(K_1) + \sigma(K_2)| + \deg(\Delta_{K_2}) + 1 \geq |\sigma(K_1)| - |\sigma(K_2)| + \deg(\Delta_{K_2})+1,
\]
which implies that
\[
\text{sl}(K_1) \geq (|\sigma(K_1)|+1) + (\deg(\Delta_{K_2}) - |\sigma(K_2)|).
\]
The assumption that the H-M inequality holds for $K_1$ gives us the three conditions. The second part of the statement follows from Lemma~\ref{specialaltsignature}.
\end{proof}

Repeated application of Lemma~\ref{HMsharpcomposites} implies the following:

\begin{coro}\label{nonprimeHMsharp}
Assume that the Hirawasa--Murasugi conjecture is true for $K_1, \dots, K_n$. Let $K$ be an H-M sharp knot with prime decomposition $K = K_1\#\cdots\#K_n$, such that $\deg(\Delta_{K_1}) \geq \deg(\Delta_{K_i})$ for $i \geq 2$.
Then the following hold up to taking the mirror of $K$:
\begin{enumerate}
    \item $K_1$ is H-M sharp with $\sigma(K_1) > 0$,
    \item $K_2,\dots,K_n$ are all special alternating knots,
    \item $\text{sl}(K_1) \geq \sum\limits_{i=2}^{n} deg(\Delta_{K_i})$.
\end{enumerate}
This characterization is necessary and sufficient. 
\end{coro}

Similarly to Corollary~\ref{H-Mplummbing}, we can generalize the argument of Corollary~\ref{nonprimeHMsharp} to diagrammatic plumbings $K = K_1 * K_2$, and, even more generally, to diagrammatic Murasugi sums of length less than~3. For this, we need to assume that the trapezoidal conjecture holds for the links $\widetilde{K}_1$ and $\widetilde{K}_2$, as well as for $K_1$ and $K_2$. Deriving necessary and sufficient criteria in this case is a bit more complicated as we need additional conditions on $\widetilde{K}_1$ and $\widetilde{K}_2$. Due to Theorem~\ref{Trapezoidalsumovertree}, the main cases of interest are diagrammatic Murasugi sums of special alternating knots (in particular, Murasugi sums of a pair), when the lengths of all sums are less than $3$. 

\begin{lemm}\label{HMsharplumbing}
Let $K_1$ and $K_2$ be alternating knots satisfying the Hirawasa--Murasugi conjecture, and assume $\deg(\Delta_{K_1}) \geq \deg(\Delta_{K_2})$. Let $K_1 * K_2$ be a diagrammatic Murasugi sum of length less than $3$. If the link $K_1 * K_2$ is H-M sharp, then $K_1$ is H-M sharp with positive signature and $K_2$ is special alternating. 
\end{lemm}

Repeated application of Lemma~\ref{HMsharplumbing} implies Corollary~\ref{H-Msharpplum}: 

\begin{coro}\label{H-Msharpplum}
Let $K$ be an H-M sharp alternating knot with Murasugi decomposition $K = K_1 * \cdots * K_n$, such that the lengths of all Murasugi sums are less than~$3$. Assume  $\deg(\Delta_{K_1}) \geq \deg(\Delta_{K_i})$ for $i\geq 2$.
Then the following hold, up to taking the mirror of $K$:
\begin{enumerate}
    \item $K_1$ is H-M sharp with $\sigma(K_1) > 0$ (a negative special alternating link),
    \item $K_2, \dots, K_n$ are positive special alternating links, and
    \item $\text{sl}(K_1) \geq \sum\limits_{i=2}^{n} \deg(\Delta_{K_i})$.
\end{enumerate}  
\end{coro}

Note that a special alternating link is H-M sharp if and only if the absolute values of all of the coefficients of the Alexander polynomial are equal. This follows from Lemma~\ref{specialaltsignature}. Alternating torus links $T_{2,q}$ are examples of H-M sharp special alternating links. Another example of such a family comes from taking the boundary of an $n$-twisted, unknotted annulus in $S^3$, denoted by $T'_{2,2n}$, which can also be obtained by changing the orientation of one of the components of $T_{2,2n}$. For special alternating knots, H-M sharpness only holds when the knot is $T_{2,2g+1}$; see Ni~\cite{Ni}.

Many of the small H-M sharp links come from diagrammatic plumbings of links of the form $T_{2,q}$ and $T'_{2,2n}$. Figure~\ref{HMsharpupto10} is the list of all H-M sharp prime knots with less than $11$ crossings.

 \begin{figure}[h]
    \centering
    \includegraphics[scale=0.5]{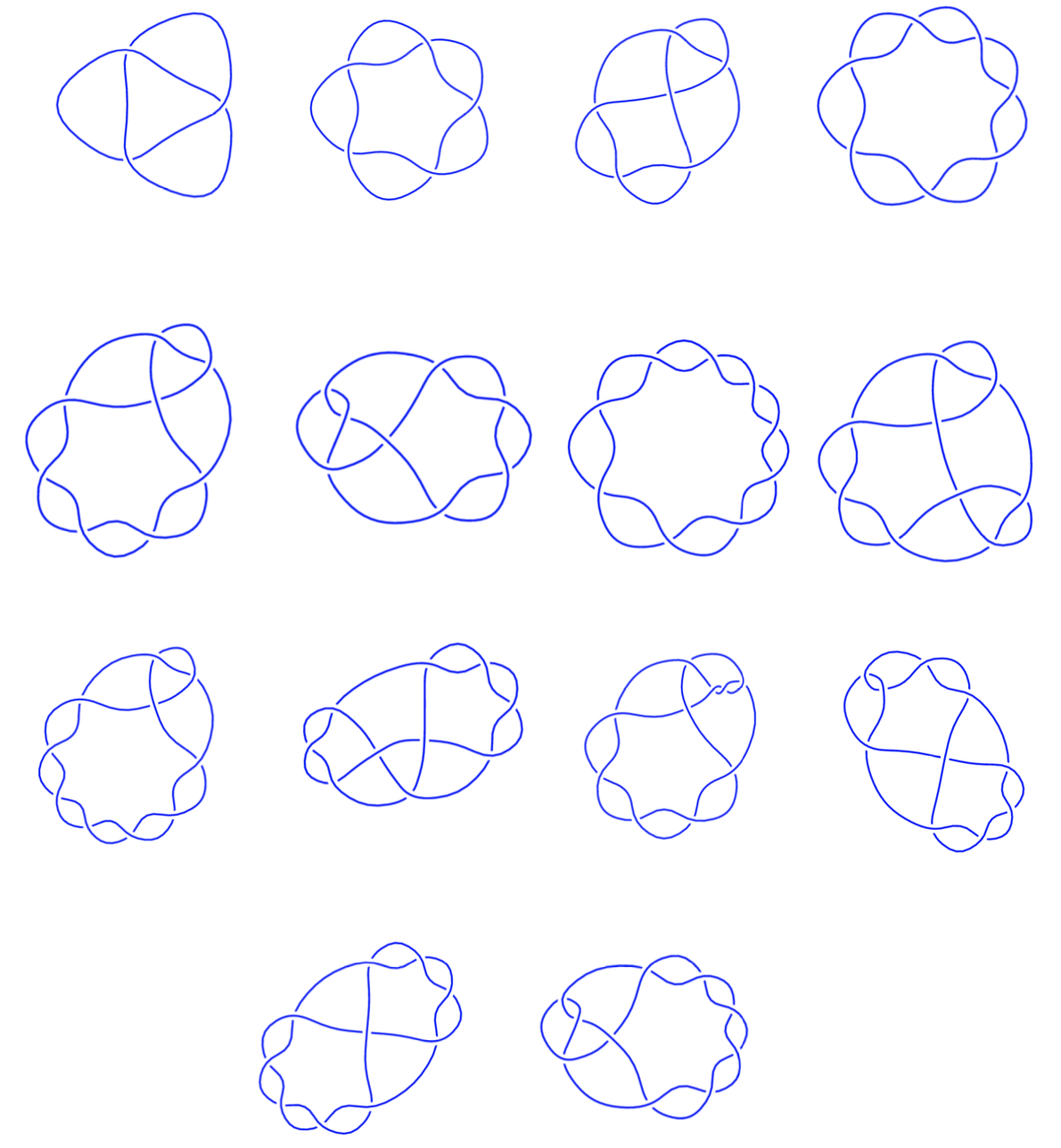}
    \caption{All H-M sharp prime knots up to 10 crossings.}
    \label{HMsharpupto10}
\end{figure}

\subsection{Dual inequality and concordance}\label{dual}

In this subsection, we investigate a connection between the H-M inequality and concordance based on the Fox--Milnor condition, stated in Theorem~\ref{FoxMilnor}.

\begin{theo}[\cite{FoxMilnor}]\label{FoxMilnor}
For concordant knots $K_1$ and $K_2$, there exists a polynomial $f(t)\in \mathbb{Z}[t]$ and $a \in \mathbb{Z}$ such that 
\[
\Delta_{K_1}(t) \Delta_{K_2}(t)  = t^{a} f(t)f(t^{-1}).
\]
\end{theo}

Note that, if one normalizes the Alexander polynomials so that they are both symmetric, then $a = 0$.
All results in this subsection rely on the following:

\begin{lemm}\label{Fox-Milnorstable}
For a non-zero polynomial $f(t) = \sum\limits_{i=0}^{n} a_it^i \in \mathbb{Z}[t]$, consider 
\[
f(t)f(t^{-1}) = \sum\limits_{i=0}^{n}c_i(t^{i}+t^{-i}).
\]
Then $|c_0| > |c_k|$ for all $k>0$. In particular, the stable length of $f(t)f(t^{-1})$ is zero. 
\end{lemm}

\begin{proof}
We just need to compute $c_0$, $c_k$, and their difference:
\[
c_0 = \sum\limits_{i=0}^{n} a_i^{2} , \ c_k = \sum\limits_{i=0}^{n-k} a_i a_{i+k},
\]
so
\[
c_0 -c_k = \sum\limits_{i=0}^{n-k} \frac{1}{2}(a_i-a_{i+k})^2 + \frac{1}{2} \sum\limits_{i=0}^{k-1} a_i^2 + \frac{1}{2} \sum\limits_{i=n-k+1}^{n} a_i^2 > 0. \qedhere
\]
\end{proof}

We analyzed the stable length of a convolution of two sequences in Figure~\ref{convolutionoftrapezoidal}. Together with Lemma~\ref{Fox-Milnorstable}, we obtain Proposition~\ref{stablelendegree}.
 
\begin{prop}\label{stablelendegree}
    Assume $K_1$ is an alternating link which is algebraically concordant to $K_2$. Then 
    \[
    \text{sl}(K_1) \leq 2\deg(\tilde{\Delta}_{K_2})+1.
    \]
\end{prop}

\begin{proof}
 If $\deg(\tilde{\Delta}_{K_1}) < \deg(\tilde{\Delta}_{K_2})$, then the inequality is obvious. Assume that $\deg(\tilde{\Delta}_{K_1}) \geq \deg(\tilde{\Delta}_{K_2})$. If $\text{sl}(K_1) > 2\deg(\tilde{\Delta}_{K_2}) + 1$, then at least two copies of the sequence of coefficients of $\tilde{\Delta}_{K_2}$ falls in the stable part of $\tilde{\Delta}_{K_1}$; see Figure~\ref{convolutionoftrapezoidal}. Hence, the stable length of $\tilde{\Delta}_{K_1} \tilde{\Delta}_{K_2}$ is non-zero, contradicting Lemma~\ref{Fox-Milnorstable}. Note that the proof does not rely on trapezoidal inequalities, and hence $K_2$ can be non-alternating.
\end{proof}

In the light of Corollary~\ref{stablelendegree}, we define the following invariant, which we use in Corollary~\ref{HMdual}.

\begin{defi}
    Define the \textit{concordance degree} of a knot $K$ as
    \[
    d_c(K) = \min \{\deg(\Delta_{K'}) \, : \, K' \ \text{algebraically concordant to} \ K \}.
    \]
\end{defi}

\begin{coro}\label{HMdual}
The H-M inequality holds for any alternating link $K$ satisfying 
\[
d_c(K) = |\sigma(K)|.
\]
This includes all alternating links which are algebraically concordant to special alternating links. 
\end{coro}

\begin{proof}
Using Proposition~\ref{stablelendegree}, we have that
\begin{equation*}\label{stablelendegreeineq}
    \text{sl}(K) \leq d_c(K) + 1 \Rightarrow \text{sl}(K) \leq  |\sigma(K)|+ 1 \Rightarrow 
\left\lfloor \frac{\text{sl}(K)}{2} \right\rfloor \leq \left\lfloor \frac{|\sigma(K)|+1}{2} \right\rfloor,
\end{equation*}    
where the last inequality is the H-M inequality for $K$. The second part of Corollary~\ref{HMdual} follows from Lemma~\ref{specialaltsignature}. 
\end{proof}

The H-M inequality can be seen as an optimization problem, where we maximize the stable length over all alternating links in an algebraic concordance class. Using inequality~\eqref{stablelendegreeineq}, one can construct a dual of this optimization problem in the form of minimizing the degree of the Alexander polynomial in any algebraic concordance class. In particular, we can search for knots satisfying 
\[
d_c(K) = \deg(\Delta_K) = |\sigma(K)|.
\]
    
Note that this duality is weak, and the dual problem cannot always prove the H-M inequality: 
    
\begin{exem}\label{weakdualexample}
Let $K$ be the knot $6_2$. We know that $\sigma(K )= -2$, $\deg(\Delta_K) = 4$, and $\sl(K) = 3$. Note that $\Delta_K$ is an irreducible element of $\mathbb{Z}[t]$, which is a UFD. Using the Fox--Milnor condition, one can deduce that, for all links $K'$ algebraically concordant to $K$, we have $\tilde{\Delta}_{K} | \tilde{\Delta}_{K'}$. Consequently, $d_c(K) = deg(\Delta_K) = 4$.
\end{exem}

We can use the same argument as in Example~\ref{weakdualexample} to compute the concordance degree of the figure eight knot.

We now show how to use the knot's symmetry to prove the H-M inequality in its algebraic concordance class. The figure eight knot is negative amphicheiral, and it satisfies the following: 

\begin{theo}[\cite{livingston2004survey}]
If a knot $K$ is concordant to a negative amphicheiral knot, then 
\[
\tilde{\Delta}_{K}(t^2) = F(t)F(t^{-1}).
\]
\end{theo}

Combining this with Lemma~\ref{Fox-Milnorstable} implies the following:
    
\begin{coro}\label{amphicheiral}
An alternating knot $K$ which is concordant to a negative amphicheiral knot has stable length zero, and hence satisfies the H-M inequality. 
\end{coro}

\bibliographystyle{alpha} 
\bibliography{bibtemplate}

\end{document}